%% file: _0NIV_dep.tex
\documentclass[12pt,a4paper,fleqn]{article}
\pdfoutput=1

\usepackage[english]{babel}
\RequirePackage{ucs}
\RequirePackage[utf8x]{inputenc}
\RequirePackage[T1]{fontenc}
\RequirePackage{lmodern}

\RequirePackage{amsthm,amsmath}
\RequirePackage{amssymb,amstext}
\RequirePackage{amsfonts}
\RequirePackage{amsbsy} 
\RequirePackage{mathrsfs}
\RequirePackage{bm}
\RequirePackage{bbm}

\usepackage[pdfpagemode=UseOutlines ,plainpages=false,
hypertexnames=false ,pdfpagelabels=TRUE ,
hyperindex=true,colorlinks=true]{hyperref}
	\makeatletter
	\Hy@breaklinkstrue
	\makeatother
\usepackage{color}
\definecolor{darkred}{rgb}{0.6,0.0,0.1}
\definecolor{darkgreen}{rgb}{0,0.5,0}
\definecolor{darkblue}{rgb}{0,0,0.5}
\hypersetup{colorlinks ,linkcolor=darkblue ,filecolor=darkgreen
,urlcolor=darkblue ,citecolor=black}

\usepackage{url}
\usepackage{verbatim}

\RequirePackage{paralist}
\usepackage{enumitem}

\usepackage{natbib}
\usepackage{hypernat}
\renewcommand{\cite}{\citet}
\usepackage{NIV_dep-abrev-package}

\definecolor{dgreen}{rgb}{0,0.5,0}
\definecolor{dblue}{rgb}{0,0,0.9}
\definecolor{dred}{rgb}{0.6,0.0,0.1}
\definecolor{dgold}{rgb}{0.5,0.3,0.0}
\definecolor{dvio}{rgb}{0.6,0.3,0.5}
\definecolor{gray}{rgb}{0.5,0.5,0.5}
\definecolor{dbraun}{rgb}{.5,0.2,0}
\definecolor{jaune}{RGB}{255,237,111}
\definecolor{tagada}{RGB}{237,162,153}
\definecolor{violett}{RGB}{250,50,100}
\definecolor{vert}{RGB}{190,242,182}

\newcommand{\dr}{\color{dred}}
\newcommand{\db}{\color{dblue}}
\newcommand{\dg}{\color{dgreen}}

\newcommand{\dgrau}{\color{gray}}

\newcommand{\ds}{\color{black}}
\newcommand{\dbraun}{\color{dbraun}}

\newtheoremstyle{mysc}
  {3pt}
  {3pt}
  {\it}
  {}
  {\color{darkred}\sc}
  {.}
  {.5em}
  {}

\newtheoremstyle{myas}
  {3pt}
  {3pt}
  {\it}
  {}
  {\color{darkblue}\sc}
  {.}
  {.5em}
  {}

\newtheoremstyle{myex}
  {10pt}
  {10pt}
  {\it}
  {}
  {\color{darkred}\sc}
  {.}
  {.5em}
  {}

\theoremstyle{mysc}\newtheorem{prop}{Proposition}[section]
\theoremstyle{mysc}\newtheorem{coro}[prop]{Corollary}
\theoremstyle{mysc}\newtheorem{theo}[prop]{Theorem}
\theoremstyle{mysc}\newtheorem{lem}[prop]{Lemma}

\theoremstyle{myas}\newtheorem{ass}{Assumption}

\theoremstyle{myex}
\theoremstyle{myex}\newtheorem*{illu}{Illustration}

\numberwithin{equation}{section} 

\makeatletter%
\def\@fnsymbol#1{\ensuremath{\ifcase#1\or 1 \or  2\or  * \or 3\or 4\or * \or \star \or  , \or 
g\or h\or i\else\@ctrerr\fi}}%
\makeatother%

\author{\begin{minipage}{.4\textwidth}\center
{\sc Nicolas Asin}\thanks{ISBA, Universit\'e catholique de Louvain, Voie du Roman Pays 20, 1348~Louvain-la-Neuve,
Belgium,  e-mail:
\url{nicolas.asin@uclouvain.be}}\\[.5ex]\small Universit\'e catholique
de Louvain\\\null\end{minipage}
\begin{minipage}{.45\textwidth}\center
{\sc Jan Johannes}
\thanks{IAM, Ruprecht-Karls-Universit\"at Heidelberg, Mathematikon, Im
  Neuenheimer Feld 205, D-69120 Heidelberg, Germany, e-mail:
\url{johannes@math.uni-heidelberg.de}}\\[.5ex]\small Ruprecht-Karls-Universit\"at Heidelberg \\\null\end{minipage}}
\date{} 
\title{Adaptive non-parametric instrumental regression\\ in the presence of dependence} 
\begin{document} 
\maketitle 

\begin{abstract}
We consider the estimation of  a structural function which models a
non-parametric relationship between a response and an endogenous
regressor given an instrument in presence of dependence in
the data generating process. Assuming an independent and identically
distributed (\iid) sample it has been shown in
\cite{JohannesSchwarz2010} that a least squares estimator based on
dimension reduction and thresholding can attain minimax-optimal rates
of convergence up to a constant. As this estimation procedure requires
an optimal choice of a dimension parameter with regard amongst others to certain
characteristics of the unknown structural function we investigate its
fully data-driven choice based on a combination of model selection and
Lepski’s method inspired by \cite{GoldenshlugerLepski2011}. For the
resulting fully data-driven thresholded least squares estimator a
non-asymptotic oracle risk bound is derived by considering  either an
\iid sample or by dismissing the independence assumption. In both
cases the derived risk bounds coincide up to a constant  assuming
sufficiently weak dependence characterised by a fast decay of the
mixing coefficients. Employing the risk bounds the minimax optimality
up to constant of the estimator is established over a variety of
classes of structural functions. 
\end{abstract} 
{\footnotesize
\begin{tabbing} 
\noindent \emph{Keywords:} \=Non-parametric regression, instrumental variable, dependence, mixing, minimax theory, adaptive. \\[.2ex] 
\noindent\emph{JEL codes:} \=C13, C14, C30, C36.
\end{tabbing}}
%
\input{_1Intro}
\input{_2Methodo}
\input{_3Oracle}
\input{_4Minimax}
\paragraph{Acknowledgements.}
This work was supported by the IAP research network no.\ P7/06 of the
Belgian Government (Belgian Science Policy), and by the contract "Projet d'Actions de Recherche Concert\'ees" (ARC) 11/16-039 of the "Communaut\'e française de Belgique", granted by the "Acad\'emie universitaire Louvain".  
\setcounter{subsection}{0}
\renewcommand{\thesubsection}{\Alph{subsection}}
\renewcommand{\theprop}{\Alph{subsection}.\arabic{prop}}
\numberwithin{equation}{subsection}  
\numberwithin{prop}{subsection}
\appendix
\section*{Appendix: Proofs}\label{app:proofs}
\subsection{Notations}\label{app:not}
\input{_proof_notations} 
\subsection{Preliminary results}\label{app:pre}
\input{_proof_preliminary} 
\input{_proof_3Oracle} 
\input{_proof_4Minimax}

\newlength{\bibitemsep}\setlength{\bibitemsep}{.2\baselineskip plus .05\baselineskip minus .05\baselineskip}
\newlength{\bibparskip}\setlength{\bibparskip}{0pt}
\let\oldthebibliography\thebibliography
\renewcommand\thebibliography[1]{%
  \oldthebibliography{#1}%
  \setlength{\parskip}{\bibitemsep}%
  \setlength{\itemsep}{\bibparskip}%
}
\bibliography{NIV_dep.bib} 
\end{document}

%% file: _1Intro.tex
%
%
%
\section{Introduction}\label{s:in}
In non-parametric instrumental regression the relationship between a response $\iY$ and an endogenous explanatory variable $\iZ$ 
is characterised by 
\begin{subequations}
\begin{equation}
\label{eq:model}\iY=\So(\iZ)+\iNo\,\quad\text{with}\,\quad\Ex(\iNo|\iZ)\ne 0
\end{equation} where the error term $\iNo$ and $\iZ$ are not stochastically mean-independent and $\So$ is called structural function. To account for the lack of mean-independence an additional exogenous random variable $\iV$, an instrument, is assumed, that is
\begin{equation}
\label{eq:model2}\Ex(\iNo|\iV)=0. 
\end{equation}
\end{subequations}
In this paper we are interested in a fully data-driven estimation of
the structural function $\So$ based on an identically distributed
(\id) sample of $(\iY,\iZ,\iV)$ consisting either of independent or
weakly dependent observations. 
Considering a thresholded least-squares estimator based on a dimension
reduction with data-driven selection of the dimension parameter  we
show that 
the resulting fully data-driven estimator can attain optimal rates of
convergence in a minimax sense.

Typical examples of models satisfying (\ref{eq:model}–\ref{eq:model2})
are error-in-variable models, simultaneous equations or treatment
models with endogenous selection. The natural generalisation
(\ref{eq:model}–\ref{eq:model2})  of a standard parametric model
(e.g. \cite{Amemiya1974}) to the non-parametric situation has been
introduced by \cite{Florens2003} and \cite{NeweyPowell2003}, while its
identification has been studied e.g. in
\cite{CarrascoFlorensRenault2007},
\cite{DarollesFanFlorensRenault2011} and
\cite{FlorensJohannesVanBellegem2011}. Applications and extensions of
this approach include non-parametric tests of exogeneity
(\cite{BlundellHorowitz2007}), quantile regression models
(\cite{HorowitzLee2007}), semi-parametric modelling
(\cite{FlorensJohannesVanBellegem2012}), or quasi-Bayesian  approaches
(\cite{FlorensSimoni2012}), to name but a few. There exists a vast
literature on the non-parametric estimation of the structural function
based on an \iid sample of $(\iY,\iZ,\iV)$. For example,
\cite{AiChen2003}, \cite{BlundellChenKristensen2007} or
\cite{NeweyPowell2003} consider sieve minimum distance estimators,
\cite{DarollesFanFlorensRenault2011},
\cite{FlorensJohannesVanBellegem2011} or \cite{GagliardiniScaillet2012} study penalised least squares
estimators, \cite{DunkerFlorensHohageJohannesMammen2014} propose an
iteratively regularised Gauß–Newton methods,  while  iteratively
regularised least squares estimators are analysed in
\cite{CarrascoFlorensRenault2007} and \cite{JohannesVanBellegemVanhems2013}. A least squares estimator based on
dimension reduction  and threshold techniques has been considered by
\cite{JohannesSchwarz2010} and \cite{BreunigJohannes2015} which
borrows ideas from the inverse problem community
(c.f. \cite{EfromovichKoltchinskii2001} or
\cite{HoffmannReiss2008}). \cite{HallHorowitz2005},
\cite{ChenReiss2011} and \cite{JohannesSchwarz2010} prove lower bounds
for the mean integrated squared error (MISE) and propose estimators
which can attain optimal rates in a minimax sense. On the other hand
lower bounds and  minimax-optimal estimation of the value of a linear
functional of the structural function has been shown in \cite{BreunigJohannes2015}.

It is worth noting that all the proposed estimation procedures rely on
the choice of at least one tuning parameter, which in turn, crucially
influences the attainable accuracy of the constructed estimator. 
In general, this choice requires knowledge of characteristics of the
structural function, such as the number of its derivatives, which are
not known in practice. From an empirical point of view data-driven
estimation procedures have been studied, for example, by
\cite{FeveFlorens2014}, and \cite{Horowitz2014}.
Considering an \iid sample a fully data-driven estimation procedure for linear functionals of the
structural function which can attain minimax-rates up to a logarithmic
deterioration has been proposed by \cite{BreunigJohannes2015}. 
On the other hand side, based on an \iid sample data-driven estimators of the structural function
which can attain lower bounds for the MISE are studied by  \cite{LoubesMarteau2009} or
\cite{JohannesSchwarz2010}. However, a straightforward application of
their results is not obvious to us since they assume  a partial
knowledge of the associated conditional expectation of
$\iZ$ given $\iV$, that is,  the eigenfunctions are known in advance, but the eigenvalues have to be estimated.
In this paper we do not impose an a priori knowledge of the eigenbasis, and
hence the estimators considered in \cite{LoubesMarteau2009} and
\cite{JohannesSchwarz2010} are no more accessible to us. Instead, we
consider  a thresholded least squares estimator as presented in
\cite{JohannesSchwarz2010}. 


Let us briefly sketch our fully data-driven estimation approach here. For the moment being, suppose that the structural function can be
represented as $\So=\sum_{j=1}^{m}\fSo_j\basZ_j$ using only $m$
pre-specified basis functions $\{\basZ_j\}_{j=1}^m$, and that only
the coefficients $\{\fSo_j\}_{j=1}^m$ with respect to these functions are
unknown. In this situation, rewriting (\ref{eq:model}--\ref{eq:model2})
as a  multivariate linear conditional moment equation the estimation
 of the $m$ coefficients of $\So$ 
is a classical textbook problem in econometrics
(cf. \cite{PaganUllah1999}). A popular approach consists in replacing
the conditional moment equation by an unconditional one, that is,
$\Ex[\iY\basV_l(\iV)]
=\sum^m_{j=1}\fSo_j\Ex[\basZ_j(\iZ)\basV_l(\iV)]$, $l = 1,\dotsc,m$
given $m$ functions $\{\basV_l\}_{l=1}^m$. 
Notice that once the functions $\set{\basV_l}_{l=1}^m$ are chosen, all
the unknown quantities in the unconditional moment equations can be
estimated by simply substituting empirical versions for the
theoretical expectation. Moreover, a least squares solution of the
estimated equation leads to a consistent and asymptotically normally
distributed estimator of the coefficients vector of $\So$ 
under mild assumptions. The choice of the functions
$\set{\basV_l}_{l=1}^m$ directly influences the asymptotic variance of
the estimator and thus the question of optimal instruments minimising
the asymptotic variance arises (cf. \cite{Newey1990}). 
However, in many situations an infinite number of functions
$\set{\basZ_j}_{j=1}^\infty$ and associated coefficients
$\set{\fSo_j}_{j=1}^\infty$ is needed to represent the structural
function $\So$, 
but we could still
consider the finite dimensional least squares estimator described
above for each dimension parameter $m\in\Nz$.  In this situation the
dimension $m$ plays the role of a smoothing parameter and we may hope
that the estimator of the structural function $\So$ is also consistent
as $m$ tends to infinity at a suitable rate. Unfortunately, this is not true in general. Let $\DiSo:=\sum_{j=1}^\Di \fDiSo_j \basZ_j$ denote a least squares solution of the reduced unconditional moment equations, that
is, the vector of coefficients $(\fDiSo_j)_{j=1}^\Di$ minimises the quantity $\sum_{l=1}^\Di\{\Ex[\iY \basV_l(\iV)] -\sum_{j=1}^\Di \ga_j\Ex[\basZ_j(\iZ)\basV_l(\iV)]\}^2$ over all
$(\ga_j)_{j=1}^\Di$. Under an additional assumption (defined below) on the basis
$\{\basV_j\}_{j\geq1}$  it is shown in \cite{JohannesSchwarz2010} that
$\DiSo$ converges to the true structural function as $\Di$ tends to
infinity. Moreover, requiring  a suitable chosen dimension parameter
$\Di$ a least squares estimator $\hDiSo$ of $\So$ based on a dimension reduction together with an
additional thresholding can attain minimax-optimal rates of
convergence in terms of the MISE.  In this paper we make use of  a method to select the dimension parameter in
a fully data-driven way, that is, neither depending on the structural
function nor on the underlying joint distribution of $\iZ$ and
$\iV$. Inspired by the work of \cite{GoldenshlugerLepski2011} the procedure
combines a model selection approach (cf. \cite{BarronBirgeMassart1999}
and its detailed discussion in \cite{Massart07}) and Lepski’s
method (cf. \cite{Lepskij1990}). 

The main contribution of this paper is the derivation of a
non-asymptotic oracle bound of the MISE for the resulting fully
data-driven thresholded least squares estimator by considering either
an \iid sample or by dismissing the independence assumption.
Employing these bounds the minimax optimality up to constant  of the estimator is established in terms
of the MISE over a variety of classes of structural functions and
conditional expectations. The estimator which depends only on the data
adapts thus automatically to the unknown characteristics of the structural function.

The paper is organised as follows: in Section \ref{s:mth} we introduce
our basic model assumptions and notations, introduce the 
thresholded least squares estimator $\hDiSo$ as proposed in
\cite{JohannesSchwarz2010} and present the data-driven method to
select the tuning parameter $\hDi$. We prove in Section \ref{s:ora} an oracle upper bound of the MISE for the resulting fully
data-driven  estimator $\hDiSo[\hDi]$
assuming first that  the \id sample of $(\iY,\iZ,\iV)$ consists of
independent observations  and second that the
sample is drawn from a strictly stationary process.  We briefly review
elementary dependence notions and present standard coupling
arguments. The risk bounds are non-asymptotic and depend as usual on
the structural function and the conditional expectation. Employing these risk bounds 
we show in Section \ref{s:mm} that within the general framework as presented in
\cite{JohannesSchwarz2010} the fully
data-driven  estimator $\hDiSo[\hDi]$ can attain up to a constant  the lower bound of
the maximal MISE over a variety of classes of
structural functions and conditional expectations. In particular we
provide sufficient conditions on the dependence structure such that
the fully data-driven estimator based on the dependent observations
can still attain the minimax-rates for independent data.


%% file: _2Methodo.tex
%
%
%
\section{Assumptions and methodology}\label{s:mth}
\paragraph{Basic model assumptions}
For ease of presentation we consider a scalar regressor $\iZ$
  and a scalar instrument $\iV$.
  However, all the results below can be extended to the multivariate
  case in a straightforward way. It is convenient to rewrite the model
  (\ref{eq:model}--\ref{eq:model2}) in terms of an operator between
  Hilbert spaces. Let us first introduce the Hilbert spaces
  $\HiZ:=\set{\So:\Rz\to\Rz\, |\, \VnormZ{\So}^2:=
    \Ex[\So^2(\iZ)]<\infty }$ and
  $\HiV:=\set{\Im:\Rz\to\Rz \,|\, \VnormV{\Im}^2:=\Ex[\Im^2(\iV)]<\infty
  }$ endowed with the usual inner products $\skalarZ$ and $\skalarV$, respectively.
  For the sake of simplicity and ease of understanding,
  we follow and refer the reader to \cite{HallHorowitz2005} for a
  discussion of the assumption that $\iZ$
  and $\iV$
  are marginally uniformly distributed on the interval
  $[0,1]$.
  Obviously, in this situation both Hilbert spaces $\HiZ$
  and $\HiV$
  are isomorphic to $L^2:=L^2[0,1]$
  endowed with the usual norm $\norm[L^2]$
  and inner product $\skalar[L^2]$.
  The conditional expectation of $\iZ$
  given $\iV$,
  however, defines a linear operator $\Op\So:=\Ex[\So(\iZ)|\iV]$,
  $\So\in\HiZ$
  mapping  $\HiZ$
  into $\HiV$.
  Taking the conditional expectation with respect to the instrument
  $\iV$
  on both sides in \eqref{eq:model} we obtain from \eqref{eq:model2}
  that:
  \begin{equation}\label{de:mth:ip}
    \Im:=\Ex(\iY|\iV)=\Ex(\So(\iZ)|\iV)=:\Op\So
  \end{equation}
  where the function $\Im$
  belongs to $\HiV$. Estimation of
    the structural function $\So$
    is thus linked to the inversion of  $\Op$   and it is therefore called an inverse
    problem. Here und subsequently,  we
    suppose implicitly that the operator $\Op$  is compact, which is the case under fairly mild assumptions
    (c.f. \cite{CarrascoFlorensRenault2006}). Consequently, unlike in
    a multivariate linear instrumental regression model, a continuous
    generalised inverse of $\Op$
    does not exist as long as the range of the operator $\Op$
    is an infinite dimensional subspace of $\HiV$.
    This corresponds to the set-up of ill-posed inverse problems with
    the additional difficulty that $\Op$
    is unknown and has to be estimated. In what follows, it is always
    assumed that there exists a unique solution $\So\in\HiZ$
    of equation \eqref{de:mth:ip}, in other words, that $\Im$
    belongs to the range of $\Op$,
    and that T is injective. For a detailed discussion in the context
    of inverse problems see Chapter 2.1 in
    \cite{EnglHankeNeubauer2000}, while in the special case of a
    non-parametric instrumental regression we refer to
    \cite{CarrascoFlorensRenault2006}.
  Considering $\ceE(\iZ,\iV):=\Ex[\iNo|\iZ,\iV]$
  we decompose throughout the paper the error term
  $\iNo=\iE+\ceE(\iZ,\iV)$
  where $\iE$
  is centred due to the mean independence of $\iNo$
  given the instrument $\iV$
  as supposed in \eqref{eq:model2}. Moreover, we assume that $\iE$
  and $(\iZ,\iV)$
  are independent of each other.  Denoting by $\VnormInf{h}$ and $\Vnorm[\iZ,\iV]{h}:=(\Ex h^2(\iZ,\iV))^{1/2}$,
  respectively,   the usual uniform norm and $L^2$-norm of a real
  valued function $h$ the next assumption completes and formalises our conditions on the
  regressor $\iZ$, the instrument $\iV$ and the random variable $\iE$.
\begin{ass}\label{a:mth:rv} The joint distribution of $(\iZ,\iV)$
  admits a bounded density $p_{\iZ,\iV}$, i.e.,
  $\VnormInf{p_{\iZ,\iV}}<\infty$, while both $\iZ$ and $\iV$ are marginally
  uniformly distributed on the interval $[0,1]$. The conditional mean
  function $\ceE(\iZ,\iV):=\Ex[\iNo|\iZ,\iV]$ is uniformly bound, that
  is, $\VnormInf{\ceE}<\infty$ and, thus $\Vnorm[\iZ,\iV]{\ceE}<\infty$.
 The random variables 
  $\{\iE_i:=\iNo_i-\ceE(\iZ_i,\iV_i)\}_{i=1}^n$ form an \iid $n$-sample of  $\iE:=\iNo-\ceE(\iZ,\iV)$ satisfying   $\Ex\iE^{12}<\infty$ and
  $\vE^2:=\Ex\iE^2>0$,  which is independent of $\{(\iZ_i,\iV_i)\}_{i=1}^n$. 
\end{ass}
\paragraph{Matrix and operator notations}
We base our estimation procedure on
  the expansion of the structural function $\So$ and the conditional
  expectation operator $\Op$ in an orthonormal basis   of $\HiZ$
  and $\HiV$, respectively.  The selection of an adequate basis in non-parametric
  instrumental regression, and inverse problems in particular, is
  discussed in various publications, (c.f. 
  \cite{EfromovichKoltchinskii2001} or \cite{BreunigJohannes2015}, and
  references within). We may emphasise that, the basis in $\HiZ$ is
  determined by the presumed information on the structural
  function and is not necessarily an eigenbasis for the unknown
  operator.   However, the statistical  choice of a basis from
  a family of bases (c.f. \cite{BirgeMassart1997})  is complicated, and its discussion is far beyond the scope
  of this paper. Therefore, we assume here and subsequently that $\basis{\basZ}$
  and $\basis{\basV}$ denotes an adequate orthonormal basis  of $\HiZ$
  and $\HiV$,
  respectively, which do not in general correspond to the
  eigenfunctions of the operator $\Op$ defined in \eqref{de:mth:ip}.
The next assumption summarises our minimal conditions  on those basis.
\begin{ass}\label{a:mth:bs}There exists a finite constant $\maxnormsup^2\geq1$ such that the
basis $\basis{\basZ}$ and $\basis{\basV}$ satisfy
$\VnormInf{\sum_{j=1}^{\Di}\basZ_j^2}\leq\Di\maxnormsup^2$ and $\VnormInf{\sum_{j=1}^{\Di}\basV_j^2}\leq\Di\maxnormsup^2$, for any $\Di\in\Nz$.
\end{ass}%
According to Lemma 6 of \cite{BirgeMassart1997} Assumption
  \ref{a:mth:bs}  is exactly equivalent to following
  property: there exists a positive constant $\maxnormsup$
  such that for any $h$
  belongs to the subspace $\Dz_\Di$,
  spanned by the first $\Di$
  basis functions, holds
  $\VnormInf{h}\leq \maxnormsup\sqrt{\Di} \VnormZ{h}$.
  Typical example are bounded basis, such as the trigonometric basis,
  or basis satisfying the assertion, that there exists a positive
  constant $C_\infty$
  such that for any $(c_1,\dotsc,c_\Di)\in\Rz^\Di$,
  $\VnormInf{\sum_{j=1}^\Di c_j\basZ_j}\leq
  C_\infty\sqrt{\Di}|c|_\infty$ where
  $|c|_\infty=\max_{1\leq j\leq\Di}c_j$.
  \cite{BirgeMassart1997} have shown that the last property is
  satisfied for piece-wise polynomials, splines and wavelets. 

Given the orthonormal
  basis $\basis{\basZ}$
  and $\basis{\basV}$
  of $\HiZ$
  and $\HiV$,
  respectively, we consider for all $\So\in\HiZ$
  and $\Im\in\HiV$
  the development $\So=\sum_{j=1}^\infty \fou{\So}_j\basZ_j$
  and $\Im=\sum_{j=1}^\infty \fou{\Im}_j\basV_j$
  where with a slight abuse of notation the sequences
  $(\fou{\So}_j)_{j\geq1}$
  and $(\fou{\Im}_j)_{j\geq1}$
  with generic elements $\fou{\So}_j:=\VskalarZ{\So,\basZ_j}$
  and $\fou{\Im}_j:=\VskalarV{\Im,\basV_j}$
  are square-summable, that is,
  $\VnormZ{\So}^2=\sum_{j=1}^\infty\fou{\So}^2_j<\infty$
  and $\VnormV{\Im}^2=\sum_{j=1}^\infty\fou{\Im}^2_j<\infty$.
  We will refer to any sequence as a whole by omitting its index as
  for example in «the sequence $\fou{\So}$».
  Furthermore, for $\Di\geq1$
  let $\fou{\So}_{\uDi}:=(\fou{\So}_j,\dotsc,\fou{\So}_\Di)^t$
  (resp. $\fou{\Im}_{\uDi}$)
  where $x^t$
  is the transpose of $x$.
  Let us further denote by $\DiHiZ$
  and $\DiHiV$
  the subspace of $\HiZ$
  and $\HiV$
  spanned by the basis functions $\{\basZ_j\}_{j=1}^\Di$
  and $\{\basV_j\}_{j=1}^\Di$,
  respectively. 
  Obviously, the norm of $\So\in\DiHiZ$
  equals the euclidean norm of its coefficient vector
  $\fou{\So}_{\uDi}$,
  that is,
  $\VnormZ{\So}=(\fou{\So}_{\uDi}^t\fou{\So}_{\uDi})^{1/2}=:\Vnorm{\fou{\So}_{\uDi}}$.
  Clearly, if $(\iY,\iZ,\iV)$
  obeys the model equations (\ref{eq:model}--\ref{eq:model2}) then
  introducing the infinite dimensional random vector
  $\fou{\basV(\iV)}$
  with generic elements $\fou{\basV(\iV)}_j=\basV_j(\iV)$
  the identity $\DifIm:=\Ex (Y\fou{\basV(\iV)}_{\uDi})$
  holds true. Consider in addition the infinite dimensional random
  vector $\fou{\basZ(\iZ)}$
  with generic elements $\fou{\basZ(\iZ)}_j=\basZ_j(\iZ)$.
  We define the $\Di\times\Di$
  dimensional matrix
  $\DifOp:=\Ex (\fou{\basV(\iV)}_{\uDi}\fou{\basZ(\iZ)}_{\uDi}^t)$
  with generic elements $\VskalarV{\basV_l,\Op\basZ_j}$  which is throughout the paper assumed to be non singular for all
  $\Di\geq1$
  (or, at least for sufficiently large $\Di$),
  so that  $\DifOp^{-1}$
  always exists with finite  spectral
  norm $\VnormS{\DifOp^{-1}}:=\sup_{\Vnorm{v}\leq 1}\Vnorm{\DifOp^{-1}v}<\infty$. Note that it is a non-trivial problem to determine
  under what precise conditions such an assumption holds (see
  e.g. \cite{EfromovichKoltchinskii2001} and references therein). 
  We
  consider the approximation $\DiSo\in\DiHiZ$
  of $\So$
  given by $\fou{\DiSo}_{\uDi}=\DifOp^{-1}\DifIm$
  and $\fou{\DiSo}_{j}=0$
  for all $j>\Di$.
  Although, it does generally not correspond to the orthogonal
  projection of $\So$
  onto the subspace $\DiHiZ$
  and the approximation error
  $\biasnivSo:=\sup_{k\geq \Di}\VnormZ{\DiSo-\So}^2$
  does generally not converge to zero as $\Di\to\infty$.
  Here and subsequently, however, we restrict ourselves to cases of
  structural functions and conditional expectation operators which
  ensure the convergence. Obviously, this is a minimal regularity
  condition for us since we aim to estimate the approximation
  $\DiSo$. 
\paragraph{Thresholded least squares estimator}
  In this paper, we follow \cite{JohannesSchwarz2010} and consider a
  least squares solution of a reduced set of unconditional moment
  equations which takes its inspiration from the linear Galerkin
  approach used in the inverse problem community
  (c.f. \cite{EfromovichKoltchinskii2001} or
  \cite{HoffmannReiss2008}).  To be precise, let
  $\set{(\iY_i,\iZ_i,\iV_i)}_{i=1}^n$
  be an identically distributed sample of $(\iY,\iZ,\iV)$
  obeying (\ref{eq:model}--\ref{eq:model2}). Since
  $\DifOp=\Ex\fou{\basV(\iV)}_{\uDi}\fou{\basZ(\iZ)}_{\uDi}^t$
  and $\DifIm=\Ex\iY\fou{\basV(\iV)}_{\uDi}$
  are written as expectations we can construct estimators using their
  empirical counterparts, that is,
  $\DihfOp:=n^{-1}\sum_{i=1}^n
  \fou{\basV(\iV_i)}_{\uDi}\fou{\basZ(\iZ_i)}_{\uDi}^t$ and
  $\DihfIm:=n^{-1}\sum_{i=1}^n \iY_i\fou{\basV(\iV_i)}_{\uDi}$.
  Let $\Ind{\{\VnormS{\DihfOp^{-1}}^2\leq n\}}$
  denote the indicator function which takes the value one if $\DihfOp$
  is non singular with squared spectral norm $\VnormS{\DihfOp^{-1}}^2$
  bounded by $n$.
  The estimator $\hDiSo\in\DiHiZ$
  of the structural function $\So$ is then defined by
  \begin{equation}\label{de:mth:est}
    \fou{\hDiSo}_{\uDi}:=\hfou{\Op}_{\uDi}^{-1}\hfou{g}_{\uDi} \Ind{\{\VnormS{\hfou{\Op}_{\uDi}^{-1}}\leq n\}}
  \end{equation}
  where the dimension parameter $\Di=\Di(n)$
  has to tend to infinity as the sample size $n$ increases.
\paragraph{Data-driven dimension selection}
Our selection method combines model selection (c.f. \cite{BarronBirgeMassart1999} and its discussion in \cite{Massart07}) and Lepskij's method (c.f. \cite{Lepskij1990})  borrowing
ideas from  \cite{GoldenshlugerLepski2011}.  We select the dimension parameter as minimiser of a penalised contrast function   which we formalise next.  
Given a positive sequence $\ga:=(\ga_m)_{m\geq 1}$ denote%
\begin{multline}\label{de:delta}
  \Delta_m(\ga):=\max\limits_{1\leq k\leq m}\ga_k, \quad\Lambda_m(\ga):=\max\limits_{1\leq k\leq m}\frac{\log(\ga_k\vee(k+2))}{\log(k+2)} \quad\mbox{and}\\\hfill  \delta_m(\ga):=m\;\Delta_m(\ga)\;\Lambda_m(\ga).
\end{multline}
Thereby, we define $\hdelta_m:=\delta_m(\ga)$ with
  $\ga=(\VnormS{\fhOp_{\um}^{-1}}^2)_{m\geq1}$. For $n\geq1$, a
  positive sequence $\ga:=(\ga_m)_{m\geq1}$ and $\alpha_n:=n^{1-1/\log(2+\log n)}(1+\log n)^{-1}$  denote
  \begin{equation}\label{de:Mn} \Mn(\ga):=\min\set{2\leq m\leq \DiMa:
      m^2\, \ga_m> \alpha_n}-1
  \end{equation}
  where we set $\Mn(\ga):=\DiMa$  if the minimum is taken over an empty set
and  $\gauss{x}$ denotes as usual the integer part of $x$. Thereby, the dimension parameter is selected
among a collection of admissible values $\{1,\dotsc,\Mh\}$ with  random integer $\Mh=\Mn(\ga)$ and  $\ga=(\VnormS{\fhOp_{\um}^{-1}}^2)_{m\geq1}$. Taking its inspiration from
\cite{ComteJohannes2012} the stochastic sequence of
penalties $(\hpen)_{1\leq m\leq\Mh}$  is defined by 
\begin{equation}\label{de:hpen} \hpen:=11\; \kappa\;\hvB^2\;
  \hdelta_m\; n^{-1}\quad\mbox{with}\quad
  \hvB^2:=2\big(\sum_{i=1}^n Y^2_i+\max_{1\leq k\leq m}\VnormZ{\hDiSo[k]}^2\big)
\end{equation}
where $\kappa$ is a positive constant to be chosen below.   The random  integer $\Mh$  and the stochastic penalties $(\hpen)_{1\leq m\leq\Mh}$ are used to define the sequence
 of contrasts $(\hcontr)_{1\leq m\leq\Mh}$ by
\begin{equation}\label{de:hcontr}
\hcontr:=\max_{\Di\leq k\leq \Mh}\set{\VnormZ{\hDiSo[k]-\hDiSo}^2 -\hpen[k] }.
\end{equation}
Setting $\argmin\nolimits_{m\in A}\{\ga_m\}:=\min\{m: \ga_m\leq \ga_{m'},\,\forall m'\in A\}$ for a sequence $(\ga_m)_{m\geq 1}$ with minimal value in $A\subset \mathbb N$ we select the dimension parameter
\begin{equation}\label{de:hDi}
\hDi:=\argmin_{1\leq \Di\leq \Mh}\set{\hcontr+\hpen}.
\end{equation}
The  estimator of $\So$ is now given by $\hSo_{\whm}$ and below we
derive an upper bound for its risk
$\Ex\VnormZ{\hDiSo[\hDi]-\So}^2$.  By construction the choice of the dimension
parameter and hence the estimator  $\hSo_{\whm}$ do rely
neither  on the structural function and the conditional
expectation operator nor on their regularity assumptions which we formalise
in  Section \ref{s:mm}.


%% file: _3Oracle.tex
%
%
\section{Non asymptotic oracle  risk bound}\label{s:ora}
\subsection{Independent observations}\label{s:ora:iid}
In this section we derive an upper bound for the MISE
  of the thresholded least squares  estimator $\hSo_{\whm}$
  with data-driven choice $\whm$
  of the dimension parameter.  We first suppose that the identically
  distributed $n$-sample
  $\{(\iY_i,\iZ_i,\iV_i)\}_{i=1}^n$
  consists of independent random variables.  In a second step we
  dismiss below the independence assumption by imposing that
  $\{(\iZ_i,\iV_i)\}_{i=1}^n$
  are weakly dependent.  
 The next assumption summarises our conditions on  the operator, the solution and its approximation.
\begin{ass}\label{a:iid:ora}\begin{ListeN}[\setlength{\labelwidth}{2ex}\setListe{0ex}{1ex}{4ex}{0ex}\renewcommand{\theListeN}{(\alph{ListeN})}] 
\item\label{a:iid:ora:a}  The matrix $\DifOp$ is non singular for all $\Di\geq 1$ such
  that $\DifOp^{-1}$ always exists.
\item\label{a:iid:ora:b} The function $\ceE$ as in Assumption
  \ref{a:mth:rv}, the  structural function $\So$ and its approximation $\DiSo\in\DiHiZ$
  given by $\fDiSo_{\uDi}=\DifOp^{-1}\DifIm$ satisfy
$\Vnorm[\iZ,\iV]{\ceE}^2\vee\VnormZ{\So}^2\vee\sup_{m\geq1}\VnormZ{\DiSo}^2\leq\DiSoTwo<\infty$
and $\VnormInf{\ceE}+\VnormInf{\So}+\sup_{m\geq1}\VnormInf{\So-\DiSo}\leq\DiSoInf<\infty$.
\end{ListeN}  
\end{ass}%
The formulation of the upper risk bound
  relies on theoretical counterparts to the random quantities $\hM$
  and $\hpen$
  which amongst other we define now referring only to the structural
  function $\So$
  and the operator $\Op$.
  Keep in mind the notation given in \eqref{de:delta} and
  \eqref{de:Mn}. For $m,n\geq1$
  and $\ga:=(\VnormS{\DifOp^{-1}}^2)_{\Di\geq1}$
  define $\Delta^\Op_{\Di}:=\Delta_{\Di}(\ga)$,
  $\Lambda^\Op_{\Di}:=\Lambda_{\Di}(\ga)$
  and $\delta^\Op_\Di:=m\Delta^\Op_\Di\Lambda_\Di^\Op$,
  set $\Mut:=\Mn(4\ga)$
  and $\Mot:=\Mn(\ga/4)$
  where $\Mut\leq \Mot$
  by construction. We require in addition that the sequence
  $(\Mot)_{n\geq 1}$
  satisfies $\log(n)(\Mot+1)^2\Delta_{\Mot+1}^\Op=o(n)$
  as $n\to\infty$. In Section \ref{s:mm:iid} below we provide an Illustration considering different
configurations for the decay of the sequence
$(\VnormS{\DifOp^{-1}}^2)_{\Di\geq1}$ where this 
condition is automatically satisfied.
\begin{theo}\label{t:iid:ora}  
Assume an i.i.d. $n$-sample of $(\iY,\iZ,\iV)$ obeying
(\ref{eq:model}--\ref{eq:model2}). Let Assumption \ref{a:mth:rv},
 \ref{a:mth:bs} and \ref{a:iid:ora} be satisfied. Set $\kappa=144$ in the
definition \eqref{de:hpen} of the penalty $\hpen$. If
$\log(n)(\Mot+1)^2\Delta_{\Mot+1}^\Op=o(n)$ as $n\to\infty$, then there exists a
constant $\SoRaC$ given as in \eqref{de:iid:ora:Sigma} in the
Appendix \ref{a:ora:iid}, which depends amongst others  on $\maxnormsup$, $\DiSoInf$ and $\vE$, and a numerical constant $C$ such that for all
$n\geq 1$ 
\begin{equation*}
\Ex\big(\VnormZ{\hDiSo[\hDi]-\So}^2\big)\leq C\;\maxnormsup^2 (1+\vE^2+\DiSoTwo)\{\min_{1\leq \Di\leq
      \Mut}\{[\biasnivSo\vee \delta_{\Di}^\Op n^{-1}]\}  +n^{-1}\;\SoRaC\}.
\end{equation*}
\end{theo}
Let us briefly comment on the last result. We shall emphasise that the
derived upper bound holds for all $n\geq 1$ true and thus is non-asymptotic. The bound consists of two terms, a 
remainder term $n^{-1}\;\SoRaC$ which is negligible with
respect to the first rhs  term $\min_{1\leq \Di\leq\Mut}\{[\biasnivSo\vee \delta_{\Di}^\Op n^{-1}]\}$.
The dependence of the factor  $\SoRaC$ in the remainder term on the
unknown structural function $\So$ (and the conditional expectation operator $\Op$) is explicitly given
in its definition \eqref{de:iid:ora:Sigma}. This
dependence is rather complicated but allows us  still  to derive in the
next section an uniform bound of $\SoRaC$ over certain classes of
structural functions and conditional expectation operators. On the
other hand side, identifying for $1\leq \Di\leq \Mut$, $\biasnivSo$ as
upper bound of the squared-bias and
$\delta_{\Di}^\Op n^{-1}$ as upper bound of the variance of the thresholded least squares 
estimator $\hDiSo$ the dominating term  $\min_{1\leq m\leq\Mut}\{[\biasnivSo\vee \delta_{\Di}^\Op n^{-1}]\}$ mimics a
squared-bias-variance trade-off. Let us further introduce
\begin{equation}\label{de:aDi}
\aDi:= \argmin_{1\leq \Di\leq\Mut}\{[\biasnivSo\vee \delta_{\Di}^\Op
n^{-1}]\} \quad\text{and}\quad \aRa:=[\biasnivSo[\aDi]\vee \delta_{\aDi}^\Op
n^{-1}].
\end{equation}
Obviously, the estimator $\hDiSo[\aDi]$ minimises  within
the family $\{\hDiSo[1],\dotsc,\hDiSo[\Mut]\}$ of estimators the upper bound for the risk. The
dimension parameter $\aDi$ and, hence the estimator $\hDiSo[\aDi]$
depend, however, on the unknown structural function and
conditional expectation operator. The estimator $\hDiSo[\aDi]$ is
therefore not feasible, and called an oracle.  We shall emphasise that
due to Theorem \ref{t:iid:ora}  the risk of the data-driven estimator
$\hDiSo[\hDi]$ is  bounded up to a constant by the risk $\aRa$ of
the oracle within the family
$\{\hDiSo[1],\dotsc,\hDiSo[\Mut]\}$. Moreover, we will show in Section
\ref{s:mm} below that $\aRa$ is the minimax-optimal rate
for a wide range of classes of structural functions and conditional
expectation operators which in turn establishes minimax optimality of
the data-driven estimator.
\subsection{Dependent observations}\label{s:ora:dep}
In this section we dismiss the independence assumption and assume
  weakly dependent observations. More precisely,
  $(\iZ_1,\iV_1),\dotsc,(\iZ_n,\iV_n)$
  are drawn from a strictly stationary process
  $\{(\iZ_i,\iV_i)\}_{i\in\Zz}$.
  Keep in mind that a process is called strictly stationary if its
  finite dimensional distributions do not change when shifted in
  time. 
  We suppose that the observations $\{(Y_i,\iZ_i,\iV_i)\}_{i=1}^n$
  still form an identically distributed sample from $(\iY,\iZ,\iV)$
  obeying the model (\ref{eq:model}--\ref{eq:model2}).  Our aim is the
  non-parametric estimation of the structural function $\So$
  under some mixing conditions on the dependence of the process
  $\set{(\iZ_i,\iV_i)}_{i\in\Zz}$.
  Let us begin with a brief review of a classical measure of
  dependence, leading to the notion of a stationary absolutely regular
  process.

  Let $(\Omega,\sA,P)$
  be a probability space. Given two sub-$\sigma$-fields
  $\sU$
  and $\sV$
  of $\sA$
  we introduce next the definition and properties of the absolutely
  regular mixing (or $\beta$-mixing)
  coefficient $\beta(\sU,\sV)$.
  The coefficient was introduced by \cite{KolmogorovRozanov1960} and
  is defined by
  \begin{equation*}
    \beta(\sU,\sV)
    :=\tfrac{1}{2}\sup\set{\sum_{i}\sum_j\left|P(U_i)P(V_i)-P(U_i\cap V_i)\right|}
  \end{equation*}
  where the supremum is taken over all finite partitions
  $(U_i)_{i\in I}$
  and $(V_j)_{j\in J}$,
  which are respectively $\sU$
  and $\sV$
  measurable. Obviously, $\beta(\sU,\sV)\leq 1$.
  As usual, if $\Ob$
  and $\Ob^\prime$
  are two  random variables defined on  $(\Omega,\sA,P)$, we denote by
  $\beta(\Ob,\Ob^\prime)$
  the mixing coefficient $\beta(\sigma(\Ob),\sigma(\Ob^\prime))$,
  where $\sigma(\Ob)$
  and $\sigma(\Ob^\prime)$
  are, respectively, the $\sigma$-fields
  generated by $\Ob$ and $\Ob^\prime$.
 
  We assume in the sequel that there exists a sequence of independent
  random variables with uniform distribution on $[0,1]$
  independent of the strictly stationary process
  $\{(\iZ_i,\iV_i)\}_{i\in\Zz}$.
  Employing Lemma 5.1 in \cite{Viennet1997} we construct by induction
  a process $\{(\cou{\iZ}_i,\cou{\iV}_i)\}_{i\geq1}$
  satisfying the following properties.  Given an integer $q$
  we introduce disjoint even and odd blocks of indices, i.e., for any
  $l\geq1$,
  $\cI^e_{l}:=\{2(l-1)q+1,\dotsc,(2l-1)q\}$
  and $\cI^o_{l}:=\{(2l-1)q+1,\dotsc,2lq\}$,
  respectively, of size $q$.
  Let us further partition into blocks the random processes
  $\{(\iZ_i,\iV_i)\}_{i\geq 1}=\{(E_l,O_l)\}_{l\geq1}$
  and
  $\{(\cou{\iZ}_i,\cou{\iV}_i)\}_{i\geq1}=\{(\couE_l,\couO_l)\}_{l\geq1}$
  where
  \begin{align*}
    E_l=(\iZ_i,\iV_i)_{i\in\cI^e_{l}},\quad
    \couE_l=(\cou{\iZ}_i,\cou{\iV}_i)_{i\in\cI^e_{l}},\quad
    O_l=(\iZ_i,\iV_i)_{i\in\cI^o_{l}}, \quad \couO_l=(\cou{\iZ}_i,\cou{\iV}_i)_{i\in\cI^o_{l}}.
  \end{align*}
  If we set further $\sF_l^-:=\sigma((\iZ_j,\iV_j),j\leq l)$
  and $\sF_l^+:=\sigma((\iZ_j,\iV_j),j\geq l)$,
  then the sequence $(\beta_k)_{k\geq 0}$
  of $\beta$-mixing
  coefficient defined by $\beta_0:=1$
  and $\beta_k:=\beta(\sF_0^-,\sF_k^+)$,
  $k\geq 1$,
  is monotonically non-increasing and satisfies trivially
  $\beta_k\geq\beta((\iZ_0,\iV_0),(\iZ_k,\iV_k))$ for any $k\geq1$.
  \noindent Based on the construction presented in \cite{Viennet1997},
  the sequence $(\cou{\iZ}_i,\cou{\iV}_i)_{i\geq1}$
  can be chosen such that for any integer $l\geq1$:
  \begin{enumerate}[label={\textbf{(P\arabic*)}},ref={\textbf{(P\arabic*)}}]\addtocounter{enumi}{0}
  \item\label{dd:as:cou1} $\couE_l$,
    $E_l$, $\couO_l$ and $O_l$ are identically distributed,
  \item\label{dd:as:cou2} $P(E_l\ne \couE_l)\leq \beta_{q+1}$,
    and $P(O_l\ne \couO_l)\leq \beta_{q+1}$.
  \item\label{dd:as:cou3} The variables $(\couE_1,\dotsc,\couE_l)$
    are \iid and so $(\couO_1,\dotsc,\couO_l)$.
  \end{enumerate}%
  We shall emphasise that the random vectors $\couE_1,\dotsc,\couE_l$
  are \iid but the components within each vector are generally not
  independent.
The next result requires the following assumption which has been used, for example, in \cite{Bosq1998}.
\begin{ass}\label{a:dep:df}
For any integer $k$ the joint distribution $P_{\iZ_0,\iV_0,\iZ_k,\iV_k}$ of
$(\iZ_0,\iV_0)$ and $(\iZ_k,\iV_k)$ admits a density $p_{\iZ_0,\iV_0,\iZ_k,\iV_k}$ which is square
integrable and satisfies\\
$\pdZWInf:=\sup_{k\geq1}\Vnorm[\iZ,\iV\times\iZ,\iV]{p_{(\iZ_0,\iV_0),(\iZ_k,\iV_k)}-p_{\iZ,\iV}\otimes p_{\iZ,\iV}}<\infty$.
\end{ass}%
\begin{theo}\label{t:dep:ora}  
Assume a sample
$\{(\iY_i,\iZ_i,\iV_i)\}_{i=1}^n$ obeying
(\ref{eq:model}--\ref{eq:model2}) where
$\{(\iZ_i,\iV_i)\}_{i=1}^n$ is drawn from a stationary absolutely
regular process with mixing coefficients $(\beta_k)_{k\geq0}$ satisfying
$\gB:=\sum_{k=0}^\infty(k+1)^2\beta_k<\infty$ and given $k\geq 1$ set
$\gB_{k}:=\sum_{j=k}^\infty\beta_k\leq \gB$. Let the Assumptions
\ref{a:mth:rv}--\ref{a:dep:df} be satisfied.
Considering the oracle dimension $\aDi$ as in \eqref{de:aDi} let
$k_n:=\gauss{(\DiSoInf/\vE)^2\pdZWInf\aDi}$ and  $\penDep\in [6+8(\DiSoInf/\vE)^2\gB_{k_n},8(1 +(\DiSoInf/\vE)^2 \gB)]$.  Set $\kappa=288\penDep$ in the
definition \eqref{de:hpen} of the penalty $\hpen$. If
$\log(n)(\Mot+1)^2\Delta_{\Mot+1}^\Op=o(n)$
 as $n\to\infty$, then there exists a
constant $\SoRaC$ given as in \eqref{de:dep:ora:Sigma} in the
Appendix \ref{a:ora:dep}, which depends amongst others  on $\maxnormsup$, $\DiSoInf$,
$\vE$ and $\gB$, and a numerical constant $C$ such that for all
$1\leq q\leq n$  
\begin{multline*}
\Ex\big(\VnormZ{\hDiSo[\hDi]-\So}^2\big)\leq
C\;\big\{[\biasnivSo[\aDi]\vee n^{-1}\delta_{\aDi}^\Op] + n^{-1}[\SoRaC\vee n^3\exp(-n^{1/6}q^{-1}/100) \vee  n^4 q^{-1}\beta_{q+1}]\big\} \\\times\maxnormsup^2 (1+\vE^2+\DiSoTwo)(1+(\DiSoInf/\vE)^2 \gB).
\end{multline*}
\end{theo}
We shall emphasise that the last assertion provides  again a non-asymptotic
  risk bound for the estimator $\hDiSo[\whm]$   with dimension $\hDi$
  as in \eqref{de:hDi}. Note that, the quantity $\penDep$
  used to construct the penalty $\hpen$
  in the last theorem still depends on the mixing coefficients
  $(\beta_k)_{k\geq0}$
  which are generally unknown.  However, the condition
  $\gB=\sum_{k=0}^\infty(k+1)^2\beta_k<\infty$
  implies $\sum_{k=k_n}^\infty\beta_k\leq (k_n+1)^{-2}\gB$
  and hence, $6+8(\DiSoInf/\vE)^2\gB_{k_n}\leq 7$
  whenever  $k_n=\gauss{(\DiSoInf/\vE)^2\pdZWInf\aDi}\geq
  \sqrt{8(\DiSoInf/\vE)^2\gB_{k_n}}$. Thereby, if $\aDi\to\infty$
  as $n\to\infty$,
  then there exists an integer $n_o$ such that $\penDep=7\in
  [6+8(\DiSoInf/\vE)^2\gB_{k_n},8(1 +(\DiSoInf/\vE)^2 \gB)]$ for all
  $n\geq n_o$.  The next assertion is thus an immediate consequence of Theorem
  \ref{t:dep:ora}, and hence its proof is omitted.
\begin{coro}\label{c:dep:ora}  Let the assumptions of Theorem
  \ref{t:dep:ora} be satisfied. Suppose that  $\aDi\to\infty$ as
  $n\to\infty$ and that there exists an  unbounded sequence of integers
$(q_n)_{n\geq1}$ and a finite constant $L$ satisfying
\begin{equation}\label{c:dep:ora:L}
\sup_{n\geq1}n^3\exp(-n^{1/6}q^{-1}/100)\leq L\quad\text{ and  }\quad\sup_{n\geq1} n^4q_n^{-1}\beta_{q_n+1}\leq
L.
\end{equation}
 If we set $\kappa=2016$ in the definition \eqref{de:hpen} of the
penalty $\hpen$, then there exist a numerical constant $C>0$ and an
integer $n_o$ such that for all $n\geq n_o$
\begin{multline*}
\Ex\big(\VnormZ{\hDiSo[\hDi]-\So}^2\big)\leq
C\;\big\{[\biasnivSo[\aDi]\vee n^{-1}\delta_{\aDi}^\Op] + n^{-1}[\SoRaC\vee L]\big\} \\\times\maxnormsup^2 (1+\vE^2+\DiSoTwo)(1+(\DiSoInf/\vE)^2 \gB).
\end{multline*}
\end{coro}
Note that the penalty $\hpen$
  used in the last assertion depends indeed only on known quantities
  and, hence the estimator $\hDiSo[\whm]$
  with dimension $\hDi$
  as in \eqref{de:hDi} is fully data-driven. It is further interesting
  to compare its upper risk bound given in Corollary \ref{c:dep:ora}
  with the upper bound derived in Theorem \ref{t:iid:ora} assuming
  independent observations. Both upper bounds coincide up to the
  multiplicative constants, thereby the discussion below Theorem
  \ref{t:iid:ora} applies also here. It remains to underline that \eqref{c:dep:ora:L} in
  Corollary \ref{c:dep:ora} imposes a sufficiently
  fast decay of the sequence of the  mixing coefficients
  $(\beta_k)_{k\geq1}$.  Is it interesting to note that an arithmetically decaying sequence
of mixing coefficients $(\beta_k)_{k\geq1}$ satisfies
\eqref{c:dep:ora:L}. To be precise, consider a sequence of integers
$(q_n)_{n\geq1}$ satisfying $q_n\sim n^{r}$, i.e., $(n^{-r}q_n)_{n\geq1}$
is bounded away both from zero and infinity,  and assume
additionally $\beta_k\sim k^{-s}$. In this situation, the condition
\eqref{c:dep:ora:L} is satisfied whenever $4-r<rs$ and $1/6>r$. In other words,
if the sequence  of mixing coefficients $(\beta_k)_{k\geq1}$ is
sufficiently fast decaying, that is $s > 4 (6 +\theta)-1$ for some
$\theta>0$, then the  condition \eqref{c:dep:ora:L} holds true
taking, for example, a sequence  $q_n\sim n^{1/(6+\theta)}$.


%% file: _4Minimax.tex
%
%
%
\section{Minimax optimality of the data-driven estimator}\label{s:mm}
\subsection{Assumptions and notations} We shall access in this section
the accuracy of the  estimator $\hDiSo[\whm]$
  with dimension $\hDi$ selected  as in \eqref{de:hDi} by its maximal
  integrated mean squared error over a class $\Soc$ of structural
  functions, that is,
  $\sup_{\So\in\Soc}\Ex\VnormZ{\hDiSo[\whm]-\So}^2$. The class $\Soc$
  reflects prior information on the structural function, e.g., its
  level of smoothness. It will be determined by means of
  a weighted norm in $\HiZ$ and, hence  will be constructed flexibly enough to
  characterise, in particular, differentiable functions. Given the
  orthonormal basis $\basis{\basZ}$ in $\HiZ$ and a strictly positive
  sequence of weights $\ga=(\ga_j)_{j\geq1}$ we define for  $
  h\in\HiZ$ the weighted norm
  $\VnormW[\ga]{h}:=(\sum_{j\in\Nz}\ga_j^{-1}\fou{h}_j^2)^{1/2}$.
  Furthermore, we denote by $\Soc_{\ga}$ and $\Soc_{\ga}^\Sor$ for a constant
  $\Sor>0$, respectively, the completion of $\HiZ$ with respect to $\normW[\ga]$ and
  the ellipsoid $\Soc_{\ga}^\Sor:=\{h\in\So_\ga:\VnormW[\ga]{h}^2\leq\Sor^2\}$.
Observe that $\Soc_{\ga}$ is a subspace of $\HiZ$ for any
non-increasing weight sequence $\ga$. Here and subsequently, we assume
that there exist a monotonically non-increasing and strictly positive
sequence of weights $\Sow:=(\Sow_j)_{j\geq 1}$ tending to zero and a
constant $\Sor>0$ such that the structural $\So$ belongs to the ellipsoid $\Socwr$ which captures all the prior information
about the unknown structural function $\So$. Additionally we specify
the mapping properties of the conditional expectation operator $\Op$
and more precisely, we will impose  a restriction on the decay of the
sequence $(\VnormS{\DifOp^{-1}}^2)_{\Di\geq 1}$ which essentially
determines $\delta^\Op$ used in the upper bounds given in Theorem
\ref{t:iid:ora} and \ref{t:dep:ora}. Denoting by $\Opc$ the set of all
operator mapping $\HiZ$ and $\HiV$ and given a
strictly positive   sequence of weights $\Opw=(\Opw_j)_{j\geq 1}$ and a constant $\Opd\geq 1$ we define the subset $\Opcwd$ of $\cT$ by \begin{equation}\label{de:lc}
\Opcwd:=\set{T\in\cT: \Opd^{-2}\VnormW[\Opw]{f}^2\leq\VnormV{\Op f}^2\leq\Opd^2\VnormW[\Opw]{f}^2,\,\quad \forall f\in\HiZ}.
\end{equation}
We notice that  each $\Op\in\Opcwd$ is injective with  $\Opd^{-2}\leq
\Opw_j\VnormV{\Op\basZ_j}^2\leq\Opd^2$ for all $j\in\Nz$. Moreover,  the
sequence $\gs:=(\gs_j)_{j\geq1}$ of singular values of $\Op$ satisfies $\Opd^{-2}\leq
\Opw_j\gs^2_j\leq\Opd^2$, too. We shall emphasise, if
$\fou{\Diag{\Opw}}_{\uDi}$ denotes the $\Di$-dimensional diagonal
matrix with diagonal entries $(\Opw_j)_{1\leq j\leq\Di}$ then  for all $\Op\in\Opcwd$ holds
$\VnormS{\fou{\Op}_{\uDi}\fou{\Diag{\Opw}}_{\uDi}^{1/2}}\leq
\Opd$ which in turn implies $\Opw_{(\Di)}:=\max_{1\leq
  j\leq\Di}\Opw_j=\VnormS{\fou{\Diag{\Opw}}_{\uDi}^{1/2}}^2\leq\Opd^2\VnormS{\DifOp^{-1}}^2$
for all $\Di\in\Nz$.  Notice that the link condition \eqref{de:lc}
involves only the  basis $\set{\basZ_l}_{l\geq1}$ in $\HiZ$. In what
follows, we introduce an alternative but stronger condition, which
extends the link condition \eqref{de:lc}.  We denote by $\OpcwdD$
for some $\OpD\geq\Opd$ the subset of $\Opcwd$ given by 
\begin{equation}\label{de:elc}
\OpcwdD=\set{\Op\in\Opcwd:
  \sup_{\Di\in\Nz}\VnormS{\fou{\Diag{\Opw}}_{\uDi}^{-1/2}\fou{\Op}_{\uDi}^{-1}}\leq\OpD}.
\end{equation}
Obviously, for all $\Op\in\OpcwdD$ we have $\VnormS{\DifOp^{-1}}^2\leq \VnormS{\fou{\Diag{\Opw}}_{\uDi}^{1/2}}^2\OpD^2=
\Opw_{(\Di)}\OpD^2$ and thus $\OpD^{-2}\leq\Opd^{-2}\leq \Opw_{(\Di)}^{-1}\VnormS{\DifOp^{-1}}^2\leq\OpD^2$
for all $\Di\in\Nz$. In other words,   the sequence $\Opw$
characterises the decay of the sequence
$(\VnormS{\DifOp^{-1}}^2)_{\Di\geq 1}$ for each $\Op\in\OpcwdD$.  It
is important to note, that the extended link
condition \eqref{de:elc} guaranties further the convergence  of the theoretical
approximation $\DiSo\in\DiHiZ$ given by $\DifDiSo:=\DifOp^{-1}\DifIm$ to the
structural function $\So$, that is, $\biasnivSo[\Di]=o(1)$ as
$\Di\to\infty$. 
Moreover, assuming in addition  $\So\in\Socwr$ 
 the approximation error satisfies $\Sow_{\Di}^{-1}\biasnivSo[\Di]\leq 4
\OpD^{4}\Sor^2$  due to  Lemma \ref{app:pre:lb} in the Appendix \ref{app:pre}.
All results of this section are derived under regularity conditions on
the structural function $\So$ and the conditional expectation operator
$\Op$ described by the sequences $\Sow$ and $\Opw$, respectively. 
 The next assumption summarises our conditions on these sequences. An illustration is provided below by assuming a \lq\lq
regular decay\rq\rq\ of these sequences.

\begin{ass}\label{a:mm:seq}
\begin{ListeN}[\setlength{\labelwidth}{2ex}\setListe{0ex}{1ex}{4ex}{0ex}\renewcommand{\theListeN}{(\alph{ListeN})}] 
\item\label{a:mm:seq:a}
Let  $\Opw:=(\Opw_j)_{j\geq1}$ be a strictly positive, finite,
monotonically non-decreasing sequences of weights with   $\Opw_1=1$. 
\item\label{a:mm:seq:b}
Let $\Sow:=(\Sow_j)_{j\geq1}$ be strictly positive, monotonically non-increasing sequence of weights with  limit zero, $\Sow_1=1$ and  
  $\VnormInf{\sum_{j\geq1}\Sow_j\basZ_j^2}\leq \SowBasZsup^2$ for some finite constant
  $\SowBasZsup\geq1$.
\end{ListeN}
\end{ass}
Note that under Assumption \ref{a:mm:seq} \ref{a:mm:seq:a}  
for each
$\Op\in\OpcwdD$ the matrix $\DifOp[\uk]$ is non-singular with
$\OpD^{-2}\leq\Opw^{-1} \VnormS{\DifOp[\uk]^{-1}}^{2}\leq\OpD^2$ for
all $k\in\Nz$, and hence the Assumption \ref{a:iid:ora}
\ref{a:iid:ora:a} holds true. On the other hand side, Assumption \ref{a:mm:seq} \ref{a:mm:seq:b} holds in case of
a bounded basis $\basis{\basZ}$ for any summable weight sequence $\Sow$, that is,
$\sum_{j\geq1}\Sow_j<\infty$. More generally, under Assumption
\ref{a:mth:bs} the additional assumption $\sum_{j\geq1}j\Sow_j<\infty$
is sufficient to ensure Assumption \ref{a:mm:seq}
\ref{a:mm:seq:b}. Furthermore, under Assumption \ref{a:mm:seq}
\ref{a:mm:seq:b} the elements of $\Socwr$ are bounded uniformly, that
is, $\VnormInf{\phi}^2\leq
\VnormInf{\sum_{j\geq1}\Sow_j\basZ_j^2}\VnormW[\Sow]{\phi}^2\leq
\SowBasZsup^2\Sor^2$ for all $\phi\in\Socwr$. The last estimate  is used in Lemma \ref{app:pre:lb} in the
Appendix \ref{app:pre}  to show that for all $\So\in\Socwr$ and
$\Op\in\OpcwdD$ the approximation $\DiSo$
satisfies $\VnormInf{\So-\DiSo}\leq 2 \SowBasZsup\OpD^2\Sor$ and
$\VnormZ{\DiSo}^2\leq 4\OpD^4\Sor^2$. Thereby, setting
$\DiSowTwo:=\Vnorm[\iZ,\iV]{\ceE}^2\vee4\OpD^4\Sor^2$ and $\DiSowTwo:=\VnormInf{\ceE}+(1+2\OpD^2 )\SowBasZsup\Sor$
the Assumption \ref{a:iid:ora} \ref{a:iid:ora:b} holds with
$\DiSoTwo:=\DiSowTwo$ and $\DiSoInf:=\DiSowInf$ uniformly for all
$\So\in\Socwr$ and $\Op\in\OpcwdD$.
\subsection{Independent observations}\label{s:mm:iid}
A careful inspection of the
proof of Theorem \ref{t:iid:ora} shows that the constant $\SoRaC$
given as in \eqref{de:iid:ora:Sigma} can be bounded uniformly by a
constant $\SowRaC$ as in \eqref{de:iid:mm:Sigma} for all
$\So\in\Socwr$.  Keep in mind the notation given in \eqref{de:delta} and
  \eqref{de:Mn}. Let us introduce in analogy to $\Mut$, $\Mot$, $\delta_m^\Op$ and $\Delta_m^\Op$  the quantities
  $\Mutw:= \Mn(4\OpD^2\Opw)$,  $\Motw:= \Mn(\Opw/(4\OpD^2))$,
  $\delta_m^\Opw:=\delta_m(\Opw)$ and
  $\Delta_m^\Opw:=\Delta_m(\Opw)$. Under Assumption \ref{a:mm:seq} it
  is easily seen that for each $\Op\in\OpcwdD$ we
  have  $(1+2\log\OpD)^{-1}\OpD^{-2}\leq
  \delta_m^\Op/\delta_m^\Opw\leq(1+2\log\OpD)\OpD^2$ for all $m\geq1$ and
  $\Mutw\leq\Mut\leq\Mot\leq\Motw$ for all $n\geq1$. If we require in
  addition that $(\log n)(\Motw+1)^2\Delta_{\Motw+1}^\Opw=o(n)$ as
  $n\to\infty$, then it holds immediately   $(\log n)(\Mot+1)^2\Delta_{\Mot+1}^{\Op}=o(n)$ as
  $n\to\infty$. Moreover, the condition is automatically satisfied in
  both cases considered in the Illustration below.
\begin{theo}\label{t:iid:mm}Assume an i.i.d. $n$-sample of $(\iY,\iZ,\iV)$ obeying
(\ref{eq:model}--\ref{eq:model2}). Let Assumption \ref{a:mth:rv}, \ref{a:mth:bs} and 
\ref{a:mm:seq}  be satisfied. Set $\kappa=144$ in the
definition \eqref{de:hpen} of the penalty $\hpen$. If
$\Op\in\OpcwdD$, 
$\log(n)(\Motw+1)^2\Delta_{\Motw+1}^\Opw=o(n)$ as $n\to\infty$, then there exists a
constant $\SowRaC$ given as in \eqref{de:iid:mm:Sigma} in the Appendix
\ref{a:mm:iid}, and a numerical constant $C$ such that  for all
$n\geq 1$ 
\begin{equation*}
\sup_{\So\in\Socwr}\Ex\big(\VnormZ{\hDiSo[\hDi]-\So}^2\big)\leq
C\;\maxnormsup^2\OpD^4 (\Sor^2+\vE^2+\DiSowTwo)\{\min_{1\leq \Di\leq
  \Mutw}\{[\Sow_{\Di}\vee \delta_{\Di}^\Opw n^{-1}]\}  +n^{-1}\;\SowRaC\}.
\end{equation*}
\end{theo}
We shall compare the last assertion with the lower bound of the maximal
  risk over the classes $\Socwr$ and $\OpcwdD$  given, for example, in
  \cite{JohannesSchwarz2010} or \cite{ChenReiss2011}. Given  sequences as in Assumption \ref{a:mm:seq} let us define
  \begin{equation}\label{de:aDi:Sow}
\taDi:=  \argmin_{1\leq\Di\leq\Mutw}\{[\Sow_\Di\vee n^{-1}\delta_{\Di}^\Opw
  ]\}\quad\text{ and }\quad\taRa:=[\Sow_{\taDi}\vee
  n^{-1}\delta_{\taDi}^\Opw]
\end{equation}
as well as $\oDi:= \argmin_{\Di\geq1}\{[\Sow_\Di\vee
n^{-1}\sum_{j=1}^{\Di}\Opw_j]\}$ and  $\oRa:=[\Sow_{\oDi}\vee
n^{-1}\sum_{j=1}^{\oDi}\Opw_j]$. Assuming a sufficiently rich class $\cP_{\iNo}$
of error distributions $P_\iNo$ (c.f.  \cite{JohannesSchwarz2010} or \cite{ChenReiss2011} for a precise
definition) there exists a constant $C$ such that for all
$\Op\in\OpcwdD$ we have 
\begin{equation}\label{r:iid:lb}
\inf_{\tSo}\sup_{P_\iNo\in\cP_{\iNo}}\sup_{\So\in\Socwr}
\Ex\big(\VnormZ{\tSo-\So}^2\big)\geq C \oRa,\quad\text{for
all }n\geq1,
\end{equation}
where the infimum is taken over all possible estimators $\tSo$ of $\So$.
Obviously, the fully data-driven estimator
  $\hDiSo[\hDi]$  given in \eqref{de:mth:est} attains the lower bound up to a constant
  if and only if $\taRa$  is of the same order as $\oRa$ which leads immediately to the following corollary.
\begin{coro}\label{c:iid:mm}  Let the Assumptions of Theorem
  \ref{t:iid:mm}   be satisfied. If $\sup_{n\geq1}\{\taRa/\oRa\}<\infty$, then
$\sup_{\So\in\Socwr}\Ex\big(\VnormZ{\hDiSo[\hDi]-\So}^2\big)=O(\oRa)$,
as $n\to\infty$.
\end{coro}
We shall emphasise that the last assertion establishes the minimax optimality of the fully
data-driven estimator $\hDiSo[\hDi]$ over the classes $\Socwr$ and
$\OpcwdD$. Therefore, the estimator is called adaptive. However,
minimax optimality is only attained if the rates $\oRa$
  and $\taRa$ are  of the same order. This is, for example, the case
  if the following two conditions hold simultaneously true:  (i) $\oDi\leq \Mutw$ and (ii)
$\delta_{\Di}^\Opw\leq C \sum_{j=1}^{\Di}\Opw_j$. Considering the
Illustration below in case  (P-P) (i) and
(ii) are satisfied, while in case (P-E) (ii) does not hold
true. However, in case (P-E)  no loss in terms of the rate occur since the
squared bias term  dominates the variance term, for a detailed
discussion in a deconvolution context, we refer to
\cite{ButuceaTsybakov2007,ButuceaTsybakov2007a}.  
\begin{illu}  We  illustrate briefly the last results considering the following two
  configurations  for the sequences $\Sow$ and $\Opw$ which are   usually studied in the literature (c.f. \cite{HallHorowitz2005},
  \cite{ChenReiss2011}, \cite{JohannesSchwarz2010} or
  \cite{BreunigJohannes2015}). Let
\begin{ListeN}[\setlength{\labelwidth}{2ex}\setListe{0ex}{1ex}{4ex}{0ex}\renewcommand{\theListeN}{(\alph{ListeN})}] 
\item[{\rm(P-P)}]\label{a:mm:ill:a} $\Sow_j=j^{-2p}$ and $\Opw_j=j^{2a}$, $j\geq1$,
  with $p>1$ and $a>1/2$;
\item[{\rm(P-E)}]\label{a:mm:ill:c} $\Sow_j=j^{-2p}$ and
  $\Opw_j=\exp(j^{2a}-1)$, $j\geq1$,  with $p>1$, $a>0$;
\end{ListeN}
then Assumption \ref{a:mm:seq} is satisfied in both cases. Writing for
two strictly positive sequences $(\ga_n)_{n\geq1}$ and
$(\gb_n)_{n\geq1}$ that $\ga_n\sim\gb_n$, if $(\ga_n/\gb_n)_{n\geq1}$
is bounded away from $0$ and infinity, we have   
\begin{ListeN}[\setlength{\labelwidth}{2ex}\setListe{0ex}{1ex}{4ex}{0ex}\renewcommand{\theListeN}{(\alph{ListeN})}] 
\item[{\rm(P-P)}]\label{a:mm:ill:a1} $\oDi\sim n^{1/(2p+2a+1)}$ and
  $\taRa\sim\oRa\sim n^{-2p/(2p+2a+1)}$;
\item[{\rm(P-E)}]\label{a:mm:ill:c1} $\oDi\sim(\log
  n-\frac{2p+(2a-1)_+}{2a}\log(\log n))^{1/(2a)}$ and $\taRa\sim\oRa\sim (\log
  n)^{-p/a}$.
\end{ListeN}
An increasing value of the parameter $a$ leads in both cases to a
slower rate $\oRa$, and hence it is called degree of ill-posedness;
cf. \cite{Natterer1984}.
\end{illu}
\subsection{Dependent observations}\label{s:mm:dep}
We dismiss again the independence assumption and assume weakly
dependent observations as introduced in Section
\ref{s:ora:dep}. Moreover, keeping in mind the case of independent observations  we replace Assumption \ref{a:iid:ora} by
Assumption \ref{a:mm:seq} which allows us to derive in
\eqref{de:dep:mm:Sigma} a constant $\SowRaC$ uniformly over the
classes $\Socwr$  depending amongst others  on the quantities
$\DiSowTwo$, $\DiSowInf$ and $\vE$.    
\begin{theo}\label{t:dep:mm}  
Assume a sample
$\{(\iY_i,\iZ_i,\iV_i)\}_{i=1}^n$ obeying
(\ref{eq:model}--\ref{eq:model2}) where
$\{(\iZ_i,\iV_i)\}_{i=1}^n$ is drawn from a stationary absolutely
regular process with mixing coefficients $(\beta_k)_{k\geq0}$ satisfying
$\gB:=\sum_{k=0}^\infty(k+1)^2\beta_k<\infty$ and given $k\geq 1$ set
$\gB_{k}:=\sum_{j=k}^\infty\beta_k\leq \gB$. Let the Assumptions
\ref{a:mth:rv}, \ref{a:mth:bs},  \ref{a:dep:df} and \ref{a:mm:seq} be satisfied.
Considering the dimension $\taDi$ as in \eqref{de:aDi:Sow} let
$k_n:=\gauss{(\DiSowInf/\vE)^2\pdZWInf\taDi}$ and  $\penDep\in [6+8(\DiSowInf/\vE)^2\gB_{k_n},8(1 +(\DiSowInf/\vE)^2 \gB)]$.  Set $\kappa=288\penDep$ in the
definition \eqref{de:hpen} of the penalty $\hpen$. If
$\Op\in\OpcwdD$, $\log(n)(\Motw+1)^2\Delta_{\Motw+1}^\Opw=o(n)$ as $n\to\infty$, then there exists a
constant $\SowRaC$ given as in \eqref{de:dep:mm:Sigma} in the
Appendix \ref{a:mm:dep}, which depends amongst others  on $\maxnormsup$,
 $\DiSowInf$, $\vE$ and $\gB$, and a  numerical constant $C$ such that for all
$1\leq q\leq n$ 
\begin{multline*}
\sup_{\So\in\Socwr}\Ex\big(\VnormZ{\hDiSo[\hDi]-\So}^2\big)\leq
C\;\big\{\taRa
+ n^{-1}[\SowRaC\vee  n^3\exp(-n^{1/6}q^{-1}/100) \vee n^4q^{-1}\beta_{q_n+1}]\big\} \\\times\maxnormsup^2 \OpD^4(\Sor^2+\vE^2+\DiSowTwo)(1+(\DiSowInf/\vE)^2 \gB).
\end{multline*}
\end{theo}
We shall emphasise that the last assertion provides  in
  analogy to Theorem \ref{t:dep:ora}  a non-asymptotic
  risk bound for the estimator $\hDiSo[\whm]$   with dimension $\hDi$
  as in \eqref{de:hDi} where the quantity $\penDep$
  used to construct the penalty $\hpen$ still depends on the mixing coefficients
  $(\beta_k)_{k\geq0}$. As Corollary \ref{c:dep:ora} in Section
  \ref{s:ora:dep} follows directly from Theorem \ref{t:dep:ora} the
  next assertion is an immediate consequence of Theorem 
  \ref{t:dep:mm}, and hence its proof is omitted.
\begin{coro}\label{c:dep:mm}  Let the assumptions of Theorem
  \ref{t:dep:mm} be satisfied. Suppose that  $\taDi\to\infty$ as
  $n\to\infty$ and that there exists an  unbounded sequence of integers
$(q_n)_{n\geq1}$ and a finite constant $L$ satisfying \eqref{c:dep:ora:L}.
If we set $\kappa=2016$ in the definition \eqref{de:hpen} of the
penalty $\hpen$, then there exist a numerical constant $C$ and an
integer $n_o$ such that for all $n\geq n_o$
\begin{multline*}
\sup_{\So\in\Socwr}\Ex\big(\VnormZ{\hDiSo[\hDi]-\So}^2\big)\leq
C\;\big\{\taRa
+ n^{-1}[\SowRaC\vee L]\big\} \\\times\maxnormsup^2 \OpD^4(\Sor^2+\vE^2+\DiSowTwo)(1+(\DiSowInf/\vE)^2 \gB).
\end{multline*}
\end{coro}
Let us briefly comment on the last result. The additional condition
\eqref{c:dep:ora:L} is, for example, satisfied if  the mixing coefficients $\beta$ 
have an arithmetic decay as pointed out below Corollary
\ref{c:dep:ora}. Comparing   Corollary
\ref{c:dep:mm} and Theorem \ref{t:iid:mm} we see that both upper
bounds coincide up to the multiplicative constants. Keep in mind
that   exploiting Theorem \ref{t:iid:mm} in case of independent
observations Corollary \ref{c:iid:mm}
establishes  minimax optimality of the estimator  $\hDiSo[\whm]$ with
dimension $\hDi$  as in \eqref{de:hDi} whenever the rates $\taRa$ and $\oRa$ coincide.
Exactly in the same manner from  Corollary
\ref{c:dep:mm}  follows the  minimax optimality of the estimator  $\hDiSo[\whm]$ 
 for weakly mixing observations  provided the rates
$\taRa$ and $\oRa$ coincide. In particular, considering the
Illustration in Section \ref{s:mm:iid} the estimator
$\hDiSo[\whm]$ attains the minimax rates in the mildly and severely ill-posed case (P-P) and (P-E),
respectively, without having in advance the knowledge of the case. It
remains to underline that the penalty   $\hpen$
  used in Corollary \ref{c:dep:mm} depends again only on known quantities
  and, hence the estimator $\hDiSo[\whm]$
  with dimension $\hDi$
  as in \eqref{de:hDi} is fully data-driven, and thus, adaptive. 

%% file: _proof_notations.tex
%
%
We begin by defining and recalling notations to be used in all
proofs. Given $m\geq 1$,
$\DiHiZ$
and $\DiHiV$
denote the subspace of $\HiZ$
and $\HiV$
spanned by the functions $\{\basZ_j\}_{j=1}^{\Di}$
and $\{\basV_j\}_{j=1}^{\Di}$,
respectively. $\DiProjZ$
and $\cDiProjZ$
(resp. $\DiProjV$
and $\cDiProjV$)
denote the orthogonal projections on $\DiHiZ$
and its orthogonal complement $\DiHiZ^\perp$,
respectively. If $K$
is an operator mapping $\HiZ$
to $\HiV$
and if we restrict $\DiProjV K \DiProjZ$
to an operator from $\DiHiZ$
to $\DiHiV$,
then it can be represented by a matrix $[K]_{\um}$
with generic entries $\VskalarV{\basV_{j},K\basZ_{l}}=:[K]_{j,l}$
for $1\leq j,l\leq m$.
The spectral norm of $[K]_{\um} $
is denoted by $\VnormS{[K]_{\um}}$
and the inverse matrix of $[K]_{\um}$
by $[K]_{\um}^{-1}$.
For $m\geq 1$,
$\Id_{\um}$
denotes the $m$-dimensional
identity matrix and for all $x\in\Rz^m$
we denote by $x^tx=:\Vnorm{x}^2$
its the euclidean norm.  Furthermore, keeping in mind the notations
given in \eqref{de:delta} and \eqref{de:Mn} we use for all $m\geq1$
and $n\geq1$
\begin{multline}\label{a:not:de:no}
\Delta_m^\Op= \Delta_m((\VnormS{\fOp_{\um}^{-1}}^2)_{m\geq1}) ,\,
\Lambda_m^\Op=\Lambda_m((\VnormS{\fOp_{\um}^{-1}}^2)_{m\geq1}),\,\delta_m^\Op=
m\Delta_m^\Op\Lambda_m^\Op,\\
\hDelta_m= \Delta_m((\VnormS{\fhOp_{\um}^{-1}}^2)_{m\geq1}),\quad \hLambda_m=\Lambda_m((\VnormS{\fhOp_{\um}^{-1}}^2)_{m\geq1}),\quad\hdelta_m= m\hDelta_m\hLambda_m,\\
\Mh= \Mn\big((\VnormS{\DihfOp^{-1}}^2)_{m\geq1}\big),\quad \Mut= \Mn(4(\VnormS{\DifOp^{-1}}^2)_{m\geq1}),\quad \Mot= \Mn(\tfrac{1}{4}(\VnormS{\DifOp^{-1}}^2)_{m\geq1}),\\
\pen=\kappa\sigma_m^2m\Delta_m^\Op\Lambda_m^\Op n^{-1}\quad\mbox{and}\quad \hpen=11\kappa\hsigma_m^2m\hDelta_m\hLambda_mn^{-1}.\hfill
\end{multline}
Recall that
$\DihfOp=\frac{1}{n}\sum_{i=1}^n\fou{\basV(\iV_i)}_{\uDi}\fou{\basZ(\iZ_i)}_{\uDi}^t$
and $\DihfIm=\frac{1}{n}\sum_{i=1}^n\iY_i\fou{\basV(\iZ_i)}_{\uDi}$
where  $\DifOp=\Ex\fou{\basV(\iV)}_{\uDi}\fou{\basZ(\iZ)}^t_{\uDi}$
and $\DifIm=\Ex \iY\fou{\basV(\iV)}_{\uDi}$.  Given
$\DiSo:=\sum_{j=1}^{\Di}\fou{\DiSo}_{j}\basZ_j\in\DiHiZ$,
$\Di\geq1$, with $\fou{\DiSo}_{\uDi}=\DifOp^{-1}\DifIm$ which is
well-defined since $\DifOp$ is non singular. Let $\iE:=\iNo-\ceE(\iZ,\iV)$
  with $\ceE(\iZ,\iV):=\Ex[\iNo|\iZ,\iV]$ where $\{\iE_i\}_{i=1}^n$  forms an \iid sample
  independent of $\{(\iZ_i,\iV_i)\}_{i=1}^n$. 
Given $\DiSoTwo=\Vnorm[\iZ,\iV]{\ceE}^2\vee\VnormZ{\So}^2\vee\sup_{m\geq1}\VnormZ{\DiSo}^2$
we note that $\vY^2:=\Ex \iY^2\leq\vE^2+2\DiSoTwo$ and $\vB^2= 2\{\vY^2+\max_{1\leq k\leq
  m}\VnormZ{\DiSo[k]}^2\}\leq 2\{\vE^2+3\DiSoTwo\}$ where $\vB^2\geq
\Ex(\iY-\DiSo(\iZ))^2$ and
$\vE^2=\inf_{\Vnorm[\iZ,\iV]{h}<\infty}\Ex(\iNo-h(\iZ,\iV))^2\leq\vY^2\wedge\Ex(\iY-\DiSo(\iZ))^2$.
Furthermore,
$\Ex|\iY-\DiSo(\iZ)|^{2k}\leq 2^{2k-1}\{\Ex(\iE)^{2k}+(\DiSoInf)^{2k}\}$ with $\DiSoInf:=\VnormInf{\ceE}+\VnormInf{\So}+\sup_{m\geq1}\VnormInf{\So-\DiSo}$. Define
the random matrix $[\Xi]_{\um}:= \DihfOp - \DifOp$
and random vectors $\fou{B}_{\uDi}$, $\fou{S}_{\uDi}$ and
$\fou{V}_{\uDi}:=\fou{B}_{\uDi}+\fou{S}_{\uDi}$ given by
their components
\begin{equation*}
  \fou{B}_j:=\frac{1}{n}\sum_{i=1}^n \iE_i \basV_j(\iV_i),\;
  \fou{S}_j:=\frac{1}{n}\sum_{i=1}^n \basV_j(\iV_i)\{\ceE(\iZ_i,\iV_i)+ \So(\iZ_i) -   
\DiSo(\iZ_i)\},\;1\leq j\leq \Di,
\end{equation*} 
where $\DihfIm- \DihfOp\fou{\DiSo}_{\uDi}=\fou{V}_{\uDi}$. Note that
$\Ex\fou{V}_{\uDi}=0$, indeed it holds $\Ex
\fou{B}_{\uDi}=0$ due to the mean independence, i.e.,
$0=\Ex(\iNo|\iV)=\Ex(\iE|\iV)+
\Ex(\iNo|\iV)=\Ex(\iE)$,
$\Ex (\ceE(\iZ,\iV))=\Ex (\Ex(\ceE(\iZ,\iV)|\iV))=\Ex
(\Ex(\iNo|\iV))=0$ and $\Ex\fou{S}_{\uDi}= \fou{\Op\So}_{\uDi} -
\fou{\Op\DiSo}_{\uDi}=0$.  Define further
$\hvY^2:=n^{-1}\sum_{i=1}^nY_i^2$, $\hvB^2= 74\{\hvY^2+\max_{1\leq k\leq
  m}\VnormZ{\hDiSo[k]}^2\}$,  the events
\begin{multline}\label{a:not:de:ev}
 \hOpset:=\{ \VnormS{\DihfOp^{-1}}\leq \sqrt n \},\quad  \Xiset:= \{4\VnormS{[\Xi]_{\um}}\VnormS{\DifOp^{-1}}\leq 1\},\\
\Aset:=\{\vY^2\leq {2}\hvY^2\leq {3}\vY^2\},\quad 
\Bset:=\{\VnormS{\DifOp[\uk]^{-1}}\VnormS{\fou{\Xi}_\uk}\leq 1/4,\forall 1\leq k\leq (\Mot+1)\},\\
\Cset:=\{\Vnorm{\DifOp[\uk]^{-1}\fou{V}_{\uk}}^2\leq
\tfrac{1}{8}(\Vnorm{\DifOp[\uk]^{-1}\DifIm[\uk]}^2 + \vY^2)
,\forall 1\leq k\leq \Mot\},\hfill\\
\Eset:=\set{\pen\leq\hpen\leq 99\pen; \quad\forall 1\leq \Di\leq \Mot }\cap\set{\Mut\leq \Mh\leq \Mot},\hfill
\end{multline}
and their complements $\hOpset^c$, $\Xiset^c$, $\Aset^c$, $\Bset^c$, $\Cset^c$, and $\Eset^c$, respectively. Furthermore,   we will denote by $C$ universal numerical constants  and
by $C(\cdot)$ constants depending only on the arguments. In both cases, the values of the constants may change from line to line.


%% file: _proof_preliminary.tex
%
%
This  section gathers preliminary results.  Given independent
observations $\{(\iZ_i,\iV_i)\}_{i=1}^n$ the first assertion
provides our key arguments in order to control the deviations of the
data-driven selection procedure.  Both inequalities are due to
\cite{Talagrand1996}, the formulation of the first part  can be found for
example in \cite{KleinRio2005}, while the second part is  based on
equation (5.13) in Corollary 2 in \cite{BirgeMassart1995} and stated
in this form for example in  \cite{ComteMerlevede2002}.
\begin{lem}(Talagrand's inequalities)\label{l:talagrand} Let
  $X_1,\dotsc,X_n$ be independent random variables and let
  $\overline{\nu_t}=n^{-1}\sum_{i=1}^n\left[\nu_t(X_i)-\Ex\left(\nu_t(X_i)\right)
  \right]$ for $\nu_t$ belonging to a countable class
  $\{\nu_t,t\in\cT\}$ of measurable functions. Then, there exists a
  numerical constant  $C>0$ such that
\begin{align}
	&\Ex\vectp{\sup_{t\in\cT}|\overline{\nu_t}|^2-6H^2}\leq C \left[\frac{v}{n}\exp\left(\frac{-nH^2}{6v}\right)+\frac{h^2}{n^2}\exp\left(\frac{- n H}{100h}\right) \right],\label{l:talagrand:eq1} \\
	&\ProbaMeasure\big(\sup_{t\in\cT}|\overline{\nu_t}|\geq2H+\lambda\big)\leq3\exp\bigg[-\frac{n}{100}\bigg(\frac{\lambda^2}{v}\wedge\frac{\lambda}{h}\bigg)\bigg],\label{l:talagrand:eq2}
\end{align}
for any $\lambda>0$, where
\begin{equation*}
	\sup_{t\in\cT}\sup_{x\in\cZ}|\nu_t(x)|\leq h,\qquad \Ex[\sup_{t\in\cT}|\overline{\nu_t}|]\leq H,\qquad \sup_{t\in\cT}\frac{1}{n}\sum_{i=1}^n \Var(\nu_t(X_i))\leq v.
\end{equation*}
\end{lem}%
\noindent Lemma \ref{l:mixing:V0} -- \ref{l:mixing:V2}  gather preliminary results if $\{(\iZ_i,\iV_i)\}_{i\in\Zz}$
is  a stationary absolutely
regular process with mixing coefficients
  $(\beta_k)_{k\geq1}$.
\begin{lem}\label{l:mixing:V0}
Under Assumption \ref{a:mth:bs} if $\{(\iZ_i,\iV_i)\}_{i\in\Zz}$ is  a stationary absolutely
regular process with mixing coefficients
  $(\beta_k)_{k\geq0}$,  then 
\begin{equation*}
  \sum_{j,l=1}^m\Var(\sum_{i=1}^q\basZ_j(\iZ_i)\basV_j(\iV_i))\leq qm^2\maxnormsup^4[1+4\sum_{k=1}^{q-1}\beta_k].
\end{equation*}
\end{lem}
\begin{proof}[\noindent\textcolor{darkred}{\sc Proof of Lemma \ref{l:mixing:V0}}]
Due to  Lemma 4.1 in \cite{AsinJohannes2016} which is a direct consequence of Theorem 2.1 in
\cite{Viennet1997} there exists a sequence $(b_k)_{k\geq1}$ of
measurable functions  $b_k:\Rz\to[0,1]$ with $\Ex
b_k(\iZ_0,\iV_0)=\beta((\iZ_0,\iV_0),(\iZ_k,\iV_k))\leq \beta_k$ such that for any
measurable function $h$ with $\Ex h^2(\iZ_0,\iV_0)<\infty$ and any
integer $q$ holds
\begin{equation*}
  \Var(\sum_{i=1}^q h(\iZ_i,\iV_i))\leq q \Ex\big\{h^2(\iZ_0,\iV_0)(1+4\sum_{k=1}^{q-1}b_k(\iZ_0,\iV_0)\}.
\end{equation*}
Setting $ h(\iZ,\iV)=\basZ_j(\iZ)\basV_l(\iV)$ the last assertion
together with Assumption \ref{a:mth:bs}  implies
\begin{multline*}
 \sum_{j,l=1}^{\Di}\Var(\sum_{i=1}^q \basZ_j(\iZ_i)\basV_l(\iV_i))\leq q \sum_{j,l=1}^{\Di}\Ex\big\{ \basZ_j^2(\iZ_i)\basV_l^2(\iV_i)(1+4\sum_{k=1}^{q-1}b_k(\iZ_0,\iV_0)\big\}\\
\leq q \Di^2\maxnormsup^4\Ex\big\{(1+4\sum_{k=1}^{q-1}b_k(\iZ_0,\iV_0)\big\}
\leq q \Di^2\maxnormsup^4\big\{1+4\sum_{k=1}^{q-1}\beta_k\big\}
\end{multline*}
which shows the assertion, and thus completes the proof.
\end{proof}
\noindent The proof of the next assertion follows along
the lines of the proof of  Theorem 2.1 and Lemma 4.1 in
\cite{Viennet1997} and  we omit the details.
\begin{lem}\label{l:mixing:V}Let $\{(\iZ_i,\iV_i)\}_{i\in\Zz}$ be  a
stationary absolutely
regular process with mixing coefficients
  $(\beta_k)_{k\geq0}$ satisfying $\gB:= 2\sum_{k=0}^\infty
(k+1)\beta_k<\infty$. Then 
\begin{equation*}
  \Var(\sum_{i=1}^nh(\iZ_i,\iV_i))\leq4n\big(\Ex h^2(\iZ_0,\iV_0)\big)^{1/2}\VnormInf{h}\gB^{1/2}.
\end{equation*}
\end{lem}
\noindent The next Lemma is a direct consequence of Theorem 2.2  in
\cite{Viennet1997} and  we omit its proof.
\begin{lem}\label{l:mixing:V2}Let $\{(\iZ_i,\iV_i)\}_{i\in\Zz}$ be  a
stationary absolutely
regular process with mixing coefficients
  $(\beta_k)_{k\geq0}$ satisfying $\gB:= \sum_{k=0}^\infty
(k+1)^2\beta_k<\infty$. Then there exists a numerical constant $C>0$
such that 
\begin{equation*}
  \Ex|\sum_{i=1}^n\{h(\iZ_i,\iV_i)-\Ex h(\iZ_i,\iV_i)\}|^4\leq C n^2 \VnormInf{h}^{p}\gB
\end{equation*}
\end{lem}
\noindent The next assertion is due to \cite{AsinJohannes2016} Lemma 4.10, and we omit its proof. 
\begin{lem}\label{l:mixing:AJ}Let $\{(\iZ_i,\iV_i)\}_{i\in\Zz}$ be  a stationary absolutely
regular process with mixing coefficients
  $(\beta_k)_{k\geq0}$. Under the
  Assumptions   \ref{a:mth:bs} and \ref {a:dep:df}
  we have for any  $q\geq 1$ and $K\in\{0,\dotsc,q-1\}$
\begin{multline}\label{l:mixing:AJ:e1}
 \sum_{j}^\Di  \Var(\sum_{i=1}^q h(\iZ_i,\iV_i)\basV_j(\iV_i))\\\leq  q m \{\maxnormsup^2\VnormZ{h}^2 + 2\VnormInf{h}^2 [ \gamma K/\sqrt{m} + 2\maxnormsup^2 \sum_{k=K+1}^{q-1}\beta_k]\}.
\end{multline}
\end{lem}
\noindent In the remaining part of this section we gather in Lemma \ref{app:pre:l4} ––
\ref{app:pre:l6} preliminary results
linking the different notations introduced in the last section.
\begin{lem}\label{app:pre:l4}For all $n,m\geq 1$ we have 
\[\set{\frac{1}{4}< \frac{\VnormS{\DihfOp^{-1}}^2}{\VnormS{\DifOp^{-1}}^2}\leq 4,\forall\,1\leq m\leq (\Mot+1)}\subset \set{\Mut\leq \Mh\leq \Mot}.\]
\end{lem}
\begin{proof}[\noindent\textcolor{darkred}{\sc Proof  of Lemma
    \ref{app:pre:l4}.}] Let $\htau_m=\VnormS{\DihfOp^{-1}}^{-2}$ and recall that $1\leq \Mh\leq \DiMa$ with
\begin{equation*}
\set{\Mh = M}=\left\{\begin{matrix}
&&\set{\frac{\htau_{M+1}}{(M+1)^2}< \alpha_n^{-1}},& M=1,\\
\set{\min\limits_{2\leq m\leq M}\frac{\htau_{m}}{m^2}\geq  \alpha_n^{-1}}&\bigcap&\set{\frac{\htau_{M+1}}{(M+1)^2} < \alpha_n^{-1}},&1< M< \DiMa,\\
\set{\min\limits_{2\leq m\leq M} \frac{\htau_{m}}{m^2}\geq  \alpha_n^{-1}},&&&M=\DiMa.
\end{matrix}\right.
\end{equation*}
For $\tau_m=\VnormS{\DifOp^{-1}}^{-2}$ we proof below the following two assertions 
\begin{align}\label{app:pre:l4:e1}
\set{\Mh < \Mut}&\subset \set{\min\limits_{1\leq m\leq \Mut}:\frac{\htau_{m}}{\tau_{m}}< \frac{1}{4}},\\\label{app:pre:l4:e2}
\set{\Mh > \Mot}&\subset \set{\max_{1\leq m\leq (\Mot+1)}\frac{\htau_{m}}{\tau_{m}}\geq 4}.
\end{align}
Obviously, the assertion of Lemma \ref{app:pre:l4} follows now  by combination of \eqref{app:pre:l4:e1} and \eqref{app:pre:l4:e2}.\\
Consider \eqref{app:pre:l4:e1} which is trivial in case $\Mut=1$. If
$\Mut>1$ we have  $\min\limits_{1\leq m\leq \Mut} \frac{\tau_m}{m^2}\geq \frac{4}{\alpha_n}$. By exploiting the last estimate  we obtain 
\begin{multline*}
\set{\Mh < \DiMa}\cap\set{\Mh < \Mut}=\bigcup_{M=1}^{\Mut-1}\set{\Mh =M}
\\\hfill\subset\bigcup_{M=1}^{\Mut-1}\set{\frac{\htau_{M+1}}{(M+1)^2}<\alpha_n^{-1}}
=
\set{\min_{2\leq m\leq \Mut}\frac{\htau_{m}}{m^2}< \alpha_n^{-1}}\subset  \set{\min_{1\leq m\leq
  \Mut}\frac{\htau_{m}}{\tau_{m}}< 1/4}
\end{multline*}
while  trivially $\set{\Mh = \DiMa}\cap\set{\Mh < \Mut}=\emptyset$, which proves \eqref{app:pre:l4:e1} because $\Mut\leq \DiMa$.\\
Consider \eqref{app:pre:l4:e2} which is trivial in case $\Mot=\DiMa$. If $\Mot<\DiMa$, then 
$\frac{\tau_{\Mot+1}}{(\Mot+1)^2}< \alpha_n^{-1}$, and hence
\begin{multline*}
  \set{\Mh >1}\cap\set{\Mh> \Mot}=\bigcup_{M=\Mot+1}^{\DiMa}\set{\Mh =M}\\
\subset\bigcup_{M=\Mot+1}^{\DiMa}\set{\min_{2\leq m\leq M}
    \frac{\htau_{m}}{m^2}\geq \alpha_n^{-1}}=
      \set{\min_{2\leq m\leq (\Mot+1)}\frac{\htau_{m}}{m^2}\geq \alpha_n^{-1}}
\subset\set{\frac{\htau_{\Mot+1}}{\tau_{\Mot+1}}\geq 4}
\end{multline*}
while  trivially $\{\Mh = 1\}\cap\{\Mh > \Mot\}=\emptyset$ which shows \eqref{app:pre:l4:e2} and completes the proof.
\end{proof}
\begin{lem}\label{app:pre:l5}  
Let $\Aset$, $\Bset$ and $\Cset$ as in \eqref{a:not:de:ev}.   For all $n\geq1$ it  holds true that \\ \centerline{$\Aset\cap\Bset\cap\Cset\subset\{ \pen[k]\leq\hpen[k]\leq 99\pen[k],1\leq k\leq \Mot\}\cap\{\Mut\leq \Mh \leq  \Mot\}.$} 
\end{lem}
\begin{proof}[\noindent\textcolor{darkred}{\sc Proof of Lemma \ref{app:pre:l5}.}] 
Let $(\Mot+1)\geq k\geq 1$.  If
$\VnormS{\DifOp[\uk]^{-1}}\VnormS{\fou{\Xi}_\uk}\leq 1/4$, i.e. on the
event $\Bset$,   it follows by the usual Neumann series
argument that $ \VnormS{(\Id_{\uk}
  +\fou{\Xi}_\uk\DifOp[\uk]^{-1})^{-1}-\Id_{\uk}}\leq 1/3$. Thus,
using the identity
$\DihfOp[\uk]^{-1}=\DifOp[\uk]^{-1}-\DihfOp[\uk]^{-1}\{(\Id_{\uk}
  +\fou{\Xi}_\uk\DifOp[\uk]^{-1})^{-1} - \Id_{\uk}\}$ we conclude
\begin{multline}\label{app:pre:l5:e1}
 2 \VnormS{\DifOp[\uk]^{-1}}\leq 3  \VnormS{\DihfOp[\uk]^{-1}}\leq 4 \VnormS{\DifOp[\uk]^{-1}}\quad\mbox{and}\hfill\\
 2 \Vnorm{\DifOp[\uk]^{-1}x}\leq 3\Vnorm{\DihfOp[\uk]^{-1} x} \leq 4 \Vnorm{\DifOp[\uk]^{-1} x},\quad\mbox{for all } x\in\Rz^k.
\end{multline}
Thereby, since
$\DihfOp[\uk]^{-1}(\fou{V}_{\uk})=\DihfOp[\uk]^{-1}\DihfIm[\uk]-\DifOp[\uk]^{-1}\DifIm[\uk]$
we conclude
\begin{gather*}
 \Vnorm{\DifOp[\uk]^{-1}\DifIm[\uk]}^2 \leq
 (32/9)\Vnorm{\DifOp[\uk]^{-1}\fou{V}_{\uk}}^2+2\Vnorm{\DihfOp[\uk]^{-1}\DihfIm[\uk]}^2,\\
 \Vnorm{\DihfOp[\uk]^{-1}\DihfIm[\uk]}^2 \leq (32/9)\Vnorm{\DifOp[\uk]^{-1}\fou{V}_{\uk}}^2+2\Vnorm{\DifOp[\uk]^{-1}\DifIm[\uk]}^2.
\end{gather*}
Consequently, on $\Cset$ where $\Vnorm{\DifOp[\uk]^{-1}\fou{V}_{\uk}}^2\leq
\frac{1}{8}(\Vnorm{\DifOp[\uk]^{-1}\DifIm[\uk]}^2 +
\vY^2)$ it follows  that
\begin{gather*}
(5/9) (\Vnorm{\DifOp[\uk]^{-1}\DifIm[\uk]}^2+\vY^2)
\leq\vY^2+2 \Vnorm{\DihfOp[\uk]^{-1}\DihfIm[\uk]}^2,\\
\Vnorm{\DihfOp[\uk]^{-1}\DihfIm[\uk]}^2\leq (22/9)\Vnorm{\DifOp[\uk]^{-1}\DifIm[\uk]}^2+(4/9)\vY^2.
\end{gather*}
and thus on $\Aset$, i.e.,  $\vY^2\leq {2}\hvY^2\leq
{3}\vY^2$ we have
\begin{gather*}
(5/9) (\Vnorm{\DifOp[\uk]^{-1}\DifIm[\uk]}^2+\vY^2)
\leq(3/2)\hvY^2+2 \Vnorm{\DihfOp[\uk]^{-1}\DihfIm[\uk]}^2,\\
\Vnorm{\DihfOp[\uk]^{-1}\DihfIm[\uk]}^2+\hvY^2\leq (22/9)\Vnorm{\DifOp[\uk]^{-1}\DifIm[\uk]}^2+(10/9)\vY^2.
\end{gather*}
Combining the last two inequalities we conclude for all $1\leq k\leq\Mot$
\begin{equation*}
(5/18) (\Vnorm{\DifOp[\uk]^{-1}\DifIm[\uk]}^2+\vY^2)\leq (\Vnorm{\DihfOp[\uk]^{-1}\DihfIm[\uk]}^2+\hvY^2)
\leq(22/9)(\Vnorm{\DifOp[\uk]^{-1}\DifIm[\uk]}^2+\vY^2).
\end{equation*}
Since on the event $\Aset\cap\Bset\cap\Cset$ the last estimates and \eqref{app:pre:l5:e1} hold for all $1\leq k\leq \Mot$ it follows
\[\Aset\cap\Bset\cap\Cset\subset\set{5\vB^2\leq 18\hvB^2\leq
    44 \vB^2  \mbox{ and }  4\Delta_m^\Op \leq 9\hDelta_m\leq 16 \Delta_m^\Op , \,\forall1\leq m\leq
  \Mot}.\]
From $\hLambda_m=\max_{1\leq k\leq m}\frac{\log(\VnormS{\DihfOp[\uk]^{-1}}^2\vee(k+2))}{\log(k+2)}$ it is easily seen that  $(4/9) \leq \hDelta_m/ \Delta_m^\Op\leq (16/9)$ implies $1/2\leq (1+\log(9/4))^{-1}\leq \hLambda_m/\Lambda_m^\Op \leq
(1+\log(16/9))\leq 2$. Taking into account the last estimates and the definitions $\pen=\kappa\vB^2m\Delta_m^\Op\Lambda_m^\Op n^{-1}$ and $\hpen=11\kappa\hvB^2m\hDelta_m\hLambda_mn^{-1}$ we obtain
\begin{equation}\label{app:pre:l5:e3}
\Aset\cap\Bset\cap\Cset\subset\set{\pen \leq \hpen\leq 99\pen, \,\forall1\leq m\leq
  \Mot}.
\end{equation}
On the other hand, by exploiting successively \eqref{app:pre:l5:e1} and Lemma \ref{app:pre:l4}  we  have
\begin{equation}\label{app:pre:l5:e4}
\Aset\cap\Bset\cap\Cset\subset\set{\frac{4}{9}\leq \frac{\VnormS{\DihfOp^{-1}}^2}{\VnormS{\DifOp^{-1}}^2}\leq \frac{9}{4}, \,\forall1\leq m\leq (\Mot+1)}\subset \set{\Mut\leq
  \Mh\leq \Mot}.
\end{equation}
From \eqref{app:pre:l5:e3} and \eqref{app:pre:l5:e4} follows  the assertion of the lemma, which completes the proof.
\end{proof}
\begin{lem}\label{app:pre:l6}  
For all $m,n\geq1$ with $ \sqrt{n}\geqslant (4/3) \VnormS{\DifOp^{-1}}$  we have $\Xiset \subset\hOpset$.
\end{lem}
\begin{proof}[\noindent\textcolor{darkred}{\sc Proof of Lemma \ref{app:pre:l6}.}] 
We observe that  $\VnormS{\DihfOp^{-1}} \leqslant (4/3)\VnormS{\DifOp^{-1}}$ due
to the usual Neumann series argument, if  
$\VnormS{\DifOp^{-1}}\VnormS{\fou{\Xi}_{\uDi}}\leq 1/4$, and consequently $\Xiset \subset\hOpset$ whenever $ \sqrt{n}\geqslant (4/3) \VnormS{[\Op]^{-1}_{\um}}$, which proves the lemma.
\end{proof}
\begin{lem}\label{app:pre:lb} 
Let $\Im=\Op\So$ and for each $\Di\in\Nz$ define $\DiSo\in\DiHiZ$ with $\fDiSo_{\uDi}:=\fOp_{\uDi}^{-1}\fIm_{\uDi}$. Given sequences $\Sow$ and $\Opw$ satisfying Assumption \ref{a:mm:seq}, for each $\So\in\Socwr$ and $\Op\in\OpcwdD$ we obtain
\begin{align}\label{app:pre:lb:e1}
	&\sup_{\Di\geq1}\Sow_{\Di}^{-1}\VnormZ{\So-\DiSo}^2 \leq 4
          \OpD^2\Opd^2\Sor^2,\quad\VnormW{\So-\DiSo}^2 \leq 4 \OpD^2\Opd^2\Sor^2,\quad \VnormZ{\DiSo}^2\leq 4\OpD^2\Opd^2\Sor^2\\\label{app:pre:lb:e2}
&\VnormInf{\So-\DiSo}^2 \leq 4\SowBasZsup^2\OpD^2\Opd^2\Sor^2,\quad\VnormInf{\So}^2 \leq \SowBasZsup^2\Sor^2.
\end{align}
\end{lem}
\begin{proof}[\noindent\textcolor{darkred}{\sc Proof of Lemma
    \ref{app:pre:lb}.}]The proof of \eqref{app:pre:lb:e1} can be found
  in \cite{JohannesSchwarz2010}. Exploiting 
  $\VnormInf{h}^2\leq\VnormInf{\sum_{j\geq1}\Sow_j\basZ_j^2}\VnormW{h}^2\leq\SowBasZsup^2\VnormW{h}^2$
  and \eqref{app:pre:lb:e1} we obtain \eqref{app:pre:lb:e2}, which
  completes the proof.
\end{proof}


%% file: _proof_3Oracle.tex
%
%
\subsection{Proof of Theorem \ref{t:iid:ora}}\label{a:ora:iid}%
\noindent We assume throughout this section that
$\{(\iY_i,\iZ_i,\iV_i)\}_{i=1}^n$ is an independent and
identically distributed sample of the random vector  $(\iY,\iZ,\iV)$ obeying the
model equations (\ref{eq:model}--\ref{eq:model2}).
We shall prove below the Propositions \ref{p:iid:Eset} and
\ref{p:iid:Rest} which are used in the proof of Theorem \ref{t:iid:ora}.
In the proof the propositions we refer to  three technical Lemma
(\ref{l:iid:sets} -- \ref{l:iid:tal2}) which are shown in the end of
this section. Moreover, we make use of  functions
  $\bPsi,\bPhia,\bPhib,\bPhic,\bPhid,\bPhie:\Rz_+\to\Rz$
  defined by
  \begin{multline}\label{de:iid:ora}
    \bPsi(x)=\sum\nolimits_{m\geq 1} x \VnormS{\DifOp^{-1}}^2 \exp(- m
    \Lambda_\Di^\Op/(6x) ),\\
    \bPhia(x)= x n \exp(- \DiMa\log(n)/(6x)),\\
    \bPhib(x)=n^{7/6}x^2\exp(-n^{1/6}/(100x)),\\
    \bPhic(x)= n^3\exp(-n(\Delta^\Op_{\Mot})^{-1}/(25600x^2)),\\
    \bPhid(x)=n^3\exp(- n(\Delta_{\Mot+1}^\Op)^{-1}/(6400x))\hfill\\
    \bPhie(x)= x n \exp(-\Mot\log(n)/(6x) ).\hfill
  \end{multline}
  We shall emphasise that each function in \eqref{de:iid:ora}  is non decreasing in $x$
  and for all $x>0$,  $\bPsi(x)<\infty$, $\bPhia(x)=o(1)$
  and $\bPhib(x)=o(1)$  as $n\to\infty$. Moreover, if
  $\log(n)(\Mot+1)^2\Delta_{\Mot+1}^\Op=o(n)$ as $n\to\infty$  then 
there exists an integer $n_o:=n_o(\Op,\maxnormsup)$ depending on $\Op$
and $\maxnormsup$ only   such that 
  \begin{equation}\label{de:iid:no}
1\geq \sup_{n\geq n_o} \big\{1024\maxnormsup^4(1+\DiSoInf/\vE\big)^2(\Mot+1)^2\Delta_{\Mot+1}^\Op
n^{-1}\big\},
\end{equation}
 and we have also for all $x>0$,  $\bPhic(x)=o(1)$, $\bPhid(x)=o(1)$
  and $\bPhie(x)=o(1)$  as $n\to\infty$. Consequently, under
  Assumption \ref{a:mth:rv} and \ref{a:mth:bs}  there exists 
a finite constant $\SoRaC$  such that for all $n\geq1$,
\begin{multline}\label{de:iid:ora:Sigma}
\SoRaC\geq
\big\{n_o^2\bigvee n^3\exp(-n^{1/6}/50)\bigvee \bPsi\big(1+\DiSoInf/\vE\big)\bigvee
\bPhia\big(1+\DiSoInf/\vE\big)\bigvee
\bPhib\big(1+\DiSoInf/\vE\big)\\\hfill\bigvee 
\bPhie\big(1+\DiSoInf/\vE\big)\bigvee 
\bPhic(1+\DiSoInf/\vE)\bigvee
\bPhid(\VnormInf{p_{\iZ,\iV}})\\
\bigvee\Ex(\iE/\vE)^{8}\bigvee(\DiSoInf/\vE)^{8}\bigvee (\maxnormsup/\vE)^2\Ex(\iE/\vE)^{12} \big\}.
\end{multline}
\begin{proof}[\noindent\textcolor{darkred}{\sc Proof of Theorem \ref{t:iid:ora}}] 
 We start the
  proof with the observation that  $(\pen[1],\dotsc,\pen[\Mh])$ is by construction a non-decreasing
  sub-sequence. Therefore, we can apply  Lemma 2.1 in \cite{ComteJohannes2012}
  which in turn implies  for
  all $1\leq \Di\leq \Mh$ that
\begin{equation}\label{t:iid:ora:key:arg}
\VnormZ{\hDiSo[\hDi]-\So}^2\leq 85 [\biasnivSo\vee\hpen] +42 \max_{\Di\leq k\leq \Mh}\vectp{ \VnormZ{\hDiSo[k]-\DiSo[k]}^2 - \hpen[k]/6}
\end{equation}
where $(x)_+:=\max(x,0)$.
Having the last bound in mind we decompose the risk
  $\Ex\VnormZ{\hDiSo[\hDi]-\So}^2$ with respect to
the event $\Eset$  defined in \eqref{a:not:de:ev} on which  the  quantities $\hpen$ and $\Mh$ are close to their theoretical
counterparts $\pen$,  $\Mut$ and $\Mot$  defined in
\eqref{a:not:de:no}. To be precise, we consider the elementary identity
  \begin{equation}\label{ri:dec}
    \Ex\VnormZ{\hDiSo[\hDi]-\So}^2= \Ex\left(\Ind{\Eset}\VnormZ{\hDiSo[\hDi]-\So}^2\right)+\Ex\left(\Ind{\Eset^c}\VnormZ{\hDiSo[\hDi]-\So}^2\right)
  \end{equation}
  where we bound the two right hand side (rhs) terms separately.  The
  second rhs term  we bound with help of Proposition
  \ref{p:iid:Rest}, which leads to
  \begin{equation}\label{ri:dec:e1}
    \Ex\VnormZ{\hDiSo[\hDi]-\So}^2\leq \Ex\left(\Ind{\Eset}\VnormZ{\hDiSo[\hDi]-\So}^2\right)+C\;n^{-1}\;
\maxnormsup^2(1+\vE^2+\DiSoTwo)\SoRaC.
  \end{equation}
Consider the first rhs term. On the event $\Eset$
  the upper bound given in \eqref{t:iid:ora:key:arg} implies
  \begin{equation*}
    \VnormZ{\hDiSo[\hDi]-\So}^2\Ind{\Eset}\leq 582 \min_{1\leq \Di\leq \Mut}\{[\biasnivSo\vee\pen]\} +42 \max_{1\leq k\leq \Mot}\vectp{ \VnormZ{\hDiSo[k]-\DiSo[k]}^2 - \pen[k]/6}.
  \end{equation*}
Keeping in mind that $\pen[k]=144\maxnormsup^2\vB[k]^2\delta_k^\Op n^{-1}$  with $\delta_k^\Op=k\Lambda_k^\Op\Delta_{k}^\Op$ and $\vB[k]^2\leq2(\vE^2+3\DiSoTwo)$  we derive in
  Proposition \ref{p:iid:Eset} below an upper bound for the
  expectation of the second rhs term, the remainder term,
  in the last display. Thereby, we obtain
 \begin{equation*}
    \Ex\big(\Ind{\Eset}\VnormZ{\hDiSo[\hDi]-\So}^2\big)\leq 
    C\;\maxnormsup^2 (1+\vE^2+\DiSoTwo)\{\min_{1\leq \Di\leq
      \Mut}\{[\biasnivSo\vee n^{-1}\delta_{\Di}^\Op]\}  +n^{-1}\;\SoRaC\}.
  \end{equation*}
Replacing in  \eqref{ri:dec:e1} the first rhs by the last upper bound
we obtain the assertion of the theorem, which completes the proof.
 \end{proof}%
\begin{prop}\label{p:iid:Eset} Under the assumptions of Theorem \ref{t:iid:ora}
  there exists a numerical constant $C$ such that for all $n\geq1$ 
\begin{equation*}\Ex\set{\max_{1\leq k\leq
      \Mot}\vectp{\VnormZ{\hDiSo-\DiSo}^2 -24\maxnormsup^2\Di n^{-1}\vB^2\Lambda_\Di^\Op\Delta_{\Di}^\Op}}\leq
C n^{-1}\maxnormsup^2(1+\vE^2+\DiSoTwo)\SoRaC.
\end{equation*}
\end{prop}
\begin{proof}[\noindent\textcolor{darkred}{\sc Proof of Proposition \ref{p:iid:Eset}}]  We start the
  proof with the observation that $\VnormZ{\hDiSo-\DiSo}^2\Ind{\hOpset}\Ind{\Xiset}\leq 2\VnormS{\DifOp^{-1}}^2\Vnorm{\fou{V}_{\uDi}}^2$ and
$\VnormZ{\hDiSo-\DiSo}^2\Ind{\hOpset}\Ind{\Xiset^c}\leq
n\Vnorm{\fou{V}_{\uDi}}^2\Ind{\Xiset^c}$, and hence
\begin{equation*}
\VnormZ{\hDiSo-\DiSo}^2\leq 2\VnormS{\DifOp^{-1}}^2 \Vnorm{\fou{V}_{\uDi}}^2 + n\Vnorm{\fou{V}_{\uDi}}^2\Ind{\Xiset^c}+\VnormZ{\DiSo}^2\Ind{\hOpset^c}.
\end{equation*}
Since  $(\Delta_m^\Op)_{m\geq1}$ as in
\eqref{a:not:de:no} satisfies
$\Delta_{m}^\Op\geq\VnormS{\DifOp^{-1}}^2$ and
$\Vnorm{\fou{V}_{\uDi}}^2\Ind{\Xiset^c}\leq \Vnorm{\fou{V}_{\underline{\Mot}}}^2\sum_{m=1}^{\Mot}\Ind{\Xiset^c}$ we obtain
 \begin{multline}\label{p:p:iid:Eset:e1}
\Ex\set{\max_{1\leq \Di\leq\Mot}\vectp{\VnormZ{\hDiSo-\DiSo}^2 -24\maxnormsup^2\Di n^{-1}\vB^2\Lambda_\Di^\Op\Delta_{\Di}^\Op}}\\
\hfill\leq 2\Ex\set{ \max_{1\leq\Di\leq\Mot} \VnormS{\DifOp^{-1}}^2\vectp{\Vnorm{\fou{V}_{\uDi}}^2
    -12\maxnormsup^2\Di n^{-1}\vB^2\Lambda_\Di^\Op}}\\\hfill +
 \Ex\set{n\vectp{\Vnorm{\fou{V}_{\underline{\Mot}}}^2-12\maxnormsup^2\Mot  n^{-1}\vB[\Mot]^2\log(n)}}
\\+12\maxnormsup^2\Mot \vB[\Mot]^2 \log(n)P(\bigcup_{k=1}^{\Mot}\Xiset[k]^c)
 +\max_{1\leq \Di\leq\Mot}\VnormZ{\DiSo}^2P(\bigcup_{k=1}^{\Mot}\hOpset[k]^c)
\end{multline}
where we bound separately each of the four rhs
terms. In order to bound  the first and second rhs term we employ \eqref{l:iid:tal:conc} in
Lemma \ref{l:iid:tal} with $K=\Mot$ and sequence $\ga=(\ga_m)_{m\geq1}$ given by $\ga_m=\VnormS{\DifOp^{-1}}^2$ and
$\ga_m=n\Ind{\{m=\Mot\}}$, respectively. Keeping in mind that in both cases
$\ga_{(K)}K^2\leq n^{3/2}$  there exists a numerical constant $C>0$
such
that 
 \begin{multline*}
\Ex\set{\max_{1\leq \Di\leq\Mot}\vectp{\VnormZ{\hDiSo-\DiSo}^2 -24\maxnormsup^2\Di n^{-1}\vB^2\Lambda_\Di^\Op\Delta_{\Di}^\Op}}\\
\leq C n^{-1}\maxnormsup^2\big\{\vE^2\bPsi(1+\DiSoInf/\vE)+\vE^2\bPhib(1+\DiSoInf/\vE)+
\vE^2\bPhie(1+\DiSoInf/\vE)+\Ex(\iE/\vE)^{12}\}\\
+6\maxnormsup^2\Mot \vB[\Mot]^2\log(n) P(\bigcup_{k=1}^{\Mot}\Xiset[k]^c)
 +\max_{1\leq \Di\leq\Mot}\VnormZ{\DiSo}^2P(\bigcup_{k=1}^{\Mot}\hOpset[k]^c)
\end{multline*}
with $\bPsi$, $\bPhib$ and $\bPhie$ as in \eqref{de:iid:ora}, i.e., $\bPsi(x)=\sum_{m\geq 1} x \VnormS{\DifOp^{-1}}^2 \exp(- m
\Lambda_\Di^\Op/(6x) )$, $\bPhib(x)=n^{7/6}x^2\exp(-n^{1/6}/(100x))$ and $\bPhie(x)= x n \exp(-
\Mot\log(n)/(6x) )$, $x>0$. 
Exploiting that $\vB^2\leq 2(\vE^2+3\DiSoTwo)$, $\Mot\log(n)\leq
n$   and $\max_{1\leq \Di\leq\Mot}\VnormZ{\DiSo}^2\leq\DiSoTwo$,
replacing the probability $P(\bigcup_{k=1}^{\Mot}\hOpset[k]^c)$ and $P(\bigcup_{k=1}^{\Mot}\Xiset[k]^c)$ by its upper bound  
given in \eqref{l:iid:sets:hO} and  \eqref{l:iid:sets:B} in Lemma
\ref{l:iid:sets}, respectively, and  employing the definition
of $\SoRaC$ as in \eqref{de:iid:ora:Sigma} we obtain the result of the proposition, 
 which completes the proof.
\end{proof}
\begin{prop}\label{p:iid:Rest} Under the assumptions of Theorem \ref{t:iid:ora}  there exists a numerical constant $C$ such that for all $n\geq1$
 \begin{equation*}
\Ex\big(\VnormZ{\hDiSo[\hDi]-\So}^2\1_{\Eset^c}\big)\leq 
C\;n^{-1}\;
\maxnormsup^2(1+\vE^2+\DiSoTwo)\SoRaC.
\end{equation*}
\end{prop}
\begin{proof}[\noindent\textcolor{darkred}{\sc Proof of Proposition \ref{p:iid:Rest}}]  We start the
  proof with the observation that  
$\VnormZ{\hDiSo-\DiSo}^2\Ind{\hOpset}\leq
n\Vnorm{\fou{V}_{\uDi}}^2\leq
n\Vnorm{\fou{V}_{\umaxDi}}^2$
for all $1\leq \Di\leq \maxDi:=\DiMa$, and hence
$\VnormZ{\hDiSo-\So}^2\Ind{\hOpset} \leq3n\Vnorm{\fou{V}_{\umaxDi}}^2+6\DiSoTwo$
where $\DiSoTwo\geq\VnormZ{\So}^2\vee\sup_{m\geq1}\VnormZ{\DiSo}^2$  which  together with $\hDi\leq \maxDi$ implies 
\begin{multline}\label{p:p:iid:Rest:e1}
\Ex\big(\VnormZ{\hDiSo[\hDi]-\So}^2\1_{\Eset^c}\big)\leq 
3\Ex\set{n\vectp{\Vnorm{\fou{V}_{\umaxDi}}^2-12\maxnormsup^2
    \vB[\maxDi]^2\maxDi \log(n)n^{-1}}} \\+\{36\maxnormsup^2 \maxDi \vB[\maxDi]^2 \log(n) + 6\DiSoTwo\}P(\Eset^c)
\end{multline}
where we bound separately the two rhs
terms. In order to bound  the first rhs term we employ \eqref{l:iid:tal:conc} in
Lemma \ref{l:iid:tal} with sequence $\ga=(\ga_m)_{m\geq1}$ given by
$\ga_m=n\Ind{\{m=K\}}$ and $K=\maxDi$ where $K^2\ga_{(K)}\leq n^{3/2}$.  Thereby, there exists a
numerical constant $C>0$ such that
\begin{multline*}
\Ex\big(\VnormZ{\hDiSo[\hDi]-\So}^2\1_{\Eset^c}\big)\leq 
Cn^{-1}\maxnormsup^2\big\{\vE^2\bPhia\big(1+\DiSoInf/\vE\big)+\vE^2\bPhib\big(1+\DiSoInf/\vE\big)
        +\Ex(\iE/\vE)^{12}\big\}\\+\{36\maxnormsup^2 \maxDi \vB[\maxDi]^2
        \log(n) + 6\DiSoTwo\}P(\Eset^c)
\end{multline*}
with $\bPhia$ and $\bPhib$ as in \eqref{de:iid:ora}, i.e.,
$\bPhia(x)= x n \exp(- \DiMa\log(n)/(6x))$ and $\bPhib(x):=n^{7/6}x^2\exp(-n^{1/6}/(100x))$, $x>0$. 
Exploiting further the definition of $\SoRaC$ as in
\eqref{de:iid:ora:Sigma} and that $\vB[\maxDi]^2\leq
2\{\vE^2+3\DiSoTwo\}$ and
$M\log(n)\leq n$ the result of the proposition follows now  by
replacing the probability $P(\Eset^c)$ by its upper bound  
given in \eqref{l:iid:sets:E} in Lemma \ref{l:iid:sets}, which completes the proof.
\end{proof}
\begin{lem}\label{l:iid:sets}
Under the assumptions of Theorem \ref{t:iid:ora} there exists a numerical constant $C$ such that for all $n\geq1$
\begin{align}
&\ProbaMeasure\big(\Aset^c)=\ProbaMeasure\big(\{1/2\leq\hsigma^2_Y/\sigma_Y^2\leq
  3/2\}^c\big)\leq C\; \SoRaC\;n^{-2},
 \label{l:iid:sets:A}\\
&\ProbaMeasure\big(\Bset^c\big)=\ProbaMeasure\big(\bigcup_{\Di=1}^{\Mot+1}\Xiset^c\big)\leq C\;
 \SoRaC\; n^{-2}, \label{l:iid:sets:B}\\
&\ProbaMeasure\big(\Cset^c\big)\leq C\; \SoRaC\; n^{-2}, \label{l:iid:sets:C}\\
&\ProbaMeasure\big(\Eset^c\big)\leq C\;\SoRaC\; n^{-2}, \label{l:iid:sets:E}\\
&\ProbaMeasure\big(\bigcup_{\Di=1}^{\Mot}\hOpset\big)\leq  C\;
 \SoRaC\;n^{-2}. \label{l:iid:sets:hO}
\end{align}
\end{lem}
\begin{proof}[\noindent\textcolor{darkred}{\sc Proof of Lemma \ref{l:iid:sets}}]Consider
  \eqref{l:iid:sets:A}. Since
  $\iY_1^2/\vY^2-1,\dotsc,\iY_n^2/\vY^2-1 $ are independent and
  and centred random variables with $\Ex\big|Y_i^2/\vY^2-1\big|^{4}\leq C \vY^{-8}\Ex|\iY|^{8}$ it follows from Theorem 2.10 in
\cite{Petrov1995}  that $\Ex\big|n^{-1}\sum_{i=1}^n\iY_i^2/\vY^2-1\big|^{4}\leq C
n^{-2} \vY^{-8}\Ex|\iY|^{8}$ where $\vY\geq \vE$ and
$\Ex|\iY|^{8}\leq C(\Ex(\iE)^{8}+(\DiSoInf)^{8})$ with
$\DiSoInf\geq\VnormInf{\ceE}\vee\VnormInf{\So}$. Employing Markov's inequality  and the last bounds
we obtain  $P\big(|n^{-1}\sum_{i=1}^n\iY_i^2/\vY^2-1|>1/2\big)\leq C n^{-2}
 (\Ex(\iE/\vE)^{8}+(\DiSoInf/\vE)^{8})$. Thereby, the assertion
\eqref{l:iid:sets:A} follows from the last bound by employing the
definition of $\SoRaC$ given in \eqref{de:iid:ora:Sigma} and  by
exploiting that  $\{1/2\leq\hvY^2/\vY^2\leq3/2\}^c\subset
\{|n^{-1}\sum_{i=1}^n\iY_i^2/\vY^2-1|>1/2\}$. Consider 
\eqref{l:iid:sets:B}--\eqref{l:iid:sets:E}. Let  $\ga$ be a sequence given by
$\ga_m=\VnormS{\DifOp^{-1}}^2$ where
$\ga_{(m)}=\Delta^T_{m}$ and $n_o$ an integer satisfying
\ref{de:iid:no}, that is,  $n\geq 1024\maxnormsup^4(1+\DiSoInf/\vE)^2(\Mot+1)^2\Delta_{\Mot+1}^\Op$ for
 all $n> n_o$. We distinguish in the following the cases $n\leq n_o$ and
$n > n_o$. Consider 
\eqref{l:iid:sets:B}. Obviously, we have
$\ProbaMeasure\big(\Bset^c\big)\leq n^{-2}n_o^2$ for all $1\leq n\leq n_o$. On the other hand, given $n\geq n_o$ and, hence $n\geq
256\maxnormsup^4(\Mot+1)^2\Delta_{\Mot+1}^\Op=4c^{-2}\maxnormsup^4 K^2\ga_{(K)}$ with sequence $\ga=(\VnormS{\DifOp^{-1}}^2)_{m\geq1}$, integer $K=\Mot+1$ and constant
$c=1/8$ we obtain from  \eqref{l:iid:tal2:prob} in Lemma
\ref{l:iid:tal2} for all $1\leq m\leq (\Mot+1)$
\begin{equation*}
  \ProbaMeasure(\Xiset^c)=\ProbaMeasure(\VnormS{\DifOp^{-1}}^2\VnormS{\fou{\Xi}_{\Di}}^2\geq1/16)
  \leq3\exp\bigg[\frac{-n}{6400\VnormInf{p_{\iZ,\iV}}\Delta_{\Mot+1}^\Op}\vee\frac{-n^{1/2}}{50}\bigg]
\end{equation*}
and hence,  given  $\bPhid$ as in \eqref{de:iid:ora} and
$\Mot+1\leq n$  it follows 
\begin{equation*}
\ProbaMeasure\big(\Bset^c\big)\leq (\Mot+1)\max_{1\leq m\leq
  (\Mot+1)}\ProbaMeasure(\Xiset^c)\leq 3\;
n^{-2}\;\bPhid(\VnormInf{p_{\iZ,\iV}})\vee \{n^3\exp(-n^{1/2}/50)\}.
\end{equation*}
 By combination of the
two cases and employing the definition of
$\SoRaC$ given in \eqref{de:iid:ora:Sigma} we obtain \eqref{l:iid:sets:B}. The proof of
\eqref{l:iid:sets:C} follows a long the lines of the proof of
\eqref{l:iid:sets:B} using \eqref{l:iid:tal:prob} in Lemma
\ref{l:iid:tal} rather than  \eqref{l:iid:tal2:prob} in Lemma
\ref{l:iid:tal2}. Precisely, if $1\leq n\leq n_o$ we have $P(\Cset^c)\leq
n^{-2}n_o^2$, while given $n> n_o$ and, hence 
$n \geq 1024\maxnormsup^2 (1+\DiSoInf/\vE)^2 (\Mot+1) \Delta_{\Mot+1}^\Op =
4c^{-2}\maxnormsup^2 (1+\DiSoInf/\vE)^2 K\ga_{(K)}$  with sequence $\ga_m=\VnormS{\DifOp^{-1}}^2$, integer
$K=\Mot$ and constant $c=1/16$ from  \eqref{l:iid:tal:prob} in Lemma
\ref{l:iid:tal} we obtain for all $1\leq m\leq \Mot$
\begin{multline*}
\ProbaMeasure(\Vnorm{\DifOp^{-1}\fou{V}_{\uDi}}^2>\tfrac{1}{8}(\VnormZ{\DiSo}^2+
\vY^2))\leq\ProbaMeasure(\ga_{m}\Vnorm{\fou{V}_{\uDi}}^2>16
c^2 \{2\VnormZ{\DiSo}^2+ 2\vY^2\})\\
\hfill\leq 3\exp\bigg[\frac{-n}{25600(1+\DiSoInf/\vE)^2\Delta^\Op_{\Mot}}\vee\frac{-n^{1/6}}{50}\bigg]
+ 32 (\maxnormsup^2/\vE^2) \Ex(\iE/\vE)^{12}\;n^{-3}.
\end{multline*}
Exploiting the definition of $\bPhic$ given in \eqref{de:iid:ora}
implies $\ProbaMeasure\big(\Cset^c\big)\leq 3 
\{n^3\exp(-n^{1/6}/50)\}\vee\bPhic(1+\DiSoInf/\vE)n^{-2}+ 32(\maxnormsup^2/\vE^2) \Ex(\iE/\vE)^{12}n^{-2} $
The assertion \eqref{l:iid:sets:C} follows employing the definition of
$\SoRaC$ given in \eqref{de:iid:ora:Sigma}. Consider
\eqref{l:iid:sets:E}. Due to Lemma \ref{app:pre:l5} it holds
$\ProbaMeasure\big(\Eset^c\big)\leq \ProbaMeasure\big(\Aset^c\big)
+\ProbaMeasure\big(\Bset^c\big)
+\ProbaMeasure\big(\Cset^c\big)$. Therefore, the assertion
\eqref{l:iid:sets:E} follows from
\eqref{l:iid:sets:A}--\eqref{l:iid:sets:C}. Consider
\eqref{l:iid:sets:hO}. We distinguish again the two cases $n\leq n_o$
and $n>n_o$, where $\ProbaMeasure\big(\bigcup_{m=1}^{\Mot}\hOpset^c\big)\leq n^{-2}n_o^2$
for all $1\leq n\leq n_o$. On the other hand, for all $n>n_o$ we have 
$n\geq (16/9)\Delta_{\Mot+1}^\Op\geq
(16/9)\VnormS{\DifOp^{-1}}^2$ for all $1\leq m\leq \Mot$, and hence from Lemma
\ref{app:pre:l6} follows $\bigcup_{m=1}^{\Mot}\hOpset^c\subset\bigcup_{m=1}^{\Mot}\Xiset^c\subset\Bset^c$ for all
$n>n_o$. Thereby, \eqref{l:iid:sets:B} implies \eqref{l:iid:sets:hO}  for all
$n>n_o$. By combination of the two cases we obtain \eqref{l:iid:sets:hO}, which completes the proof.
\end{proof}
\begin{lem}\label{l:iid:tal} Given  a
  non negative sequence  $\ga:=(\ga_{\Di})_{\Di\in\Nz}$ let
  $\Lambda_\Di^\ga:=\Lambda_\Di(\ga)$ as in \eqref{de:delta},
  $\ga_{(K)}:=\max_{1\leq m\leq K}\ga_m$, for any $x>0$, $\Phi_n(x):=n^{7/6}x^2\exp(-n^{1/6}/(100x))$ and $\Psi_\ga(x):=\sum_{m\geq 1} x^2 \ga_m \exp(- m \Lambda_\Di^\ga/(6x^2) )<\infty$, which by
  construction always exists. If    $\ga_{(K)}K^2\leq n^{3/2}$ and $\DiSoInf\geq\sup_{\Di\geq
      1}\VnormInf{\ceE+\So-\DiSo}<\infty$  then  there exists a
  numerical constant $C$ such that 
\begin{multline}\label{l:iid:tal:conc}
\Ex\vectp{\max_{1\leq\Di\leq K}\ga_\Di[\Vnorm{\fou{V}_{\uDi}}^2-12\maxnormsup^2\sigma_{\Di}^2\Di
  \Lambda_\Di^\ga n^{-1}]}\\
\leq Cn^{-1} \maxnormsup^2 \big\{\vE^2\Psi_\ga\big(1+\DiSoInf/\vE\big)+\vE^2\Phi_n\big(1+\DiSoInf/\vE\big)
        +\Ex(\iE/\vE)^{12}\big\}
\end{multline}
Moreover, if  $n \geq 4 c^{-2}(1+\DiSoInf/\vE)^2\maxnormsup^2  K \ga_{(K)}$ for $c>0$ then for all $1\leq\Di\leq K$ holds
\begin{multline}\label{l:iid:tal:prob}
\ProbaMeasure\big(\ga_m\Vnorm{\fou{V}_{\uDi}}^2\geq 16c^2\{2\vY^2+2\VnormZ{\DiSo}^2\}\big)\\
\hfill\leq 3\exp\bigg[\frac{-nc^2}{100(1+\DiSoInf/\vE)^2\ga_{(K)}}\vee\frac{-n^{1/6}}{50}\bigg]\hfill\\
+ (8c^2)^{-1} (\maxnormsup^2/\vE^2) \Ex(\iE/\vE)^{12}\;n^{-3}.
      \end{multline}
\end{lem}
\begin{proof}[\noindent\textcolor{darkred}{\sc Proof of Lemma \ref{l:iid:tal}}]
We intend to apply  Talagrand's inequalities given in Lemma \ref{l:talagrand}  employing the identity
$\Vnorm{\fou{V}_{\uDi}}^2=\sup_{t\in\Bz_{\Di}}|\overline{\nu_t}|^2$
where $\Bz_{\Di}:=\{t\in\DiHiZ:\VnormZ{t}\leq1\}$
and
$\nu_t(\iE,\iZ,\iV)=\sum_{j=1}^\Di(\iE+\ceE(\iZ,\iV)+\So(\iZ)-\DiSo(\iZ))\fou{t}_j\basV_j(\iV)$
where $\iNo=\iE+\ceE(\iZ,\iV)$ and $\iE$ and $(\iZ,\iV)$ are
independent. A direct application, however, is not possible since
$\iE$ and hence, $\nu_t$ are generally not uniformly bounded in $\iE$.  Therefore, let
us introduce $\iE^b:=\iE\Ind{\set{|\iE|\leq\vE n^{1/3}}}- \Ex\iE\Ind{\set{|\iE|\leq\vE n^{1/3}}}$ and
      $\iE^u:=\iE-\iE^b$. Setting
      $\nu_t^b(\iE,\iZ,\iV):=\nu_t(\iE^b,\iZ,\iV)$ and
      $\nu_t^u:=\nu_t-\nu_t^b=\sum_{j=1}^\Di\iE^u\fou{t}_j\basV_j(\iV)$
      we have obviously $
      \overline{\nu_t}=\overline{\nu_t}^b+\overline{\nu_t}^u$. Considering
      first the assertion \eqref{l:iid:tal:conc} it follows
\begin{multline}\label{p:l:iid:tal:conc:e1}
\Ex\vectp{
  \max_{1\leq\Di\leq K}\set{\ga_\Di[\Vnorm{\fou{V}_{\uDi}}^2
    -12\maxnormsup^2\vB^2 \Di \Lambda_\Di^\ga n^{-1}]}}\\
\leq2\Ex\vectp{\max_{1\leq \Di\leq K}\{\ga_\Di[\sup_{t\in\Bz_\Di}|{\overline{\nu_t}^b}|^2-
  6\maxnormsup^2\vB^2\Di \Lambda_\Di^\ga n^{-1}]\}}+2\ga_{(K)}\Ex(\sup_{t\in\Bz_{K}}|{\overline{\nu_t}^u}|^2)
\end{multline}
where we bound separately each rhs term. Consider the second rhs
term. Keeping in mind that
$\Ex|\iE|^2\Ind{\set{|\iE|^2>\eta^2}}\leq\eta^{-10}\Ex(|\iE|^{12})$ for any $\eta>0$  and $\vE^2=\Ex|\iE|^2$
by employing successively the independence of the
sample, Assumption  \ref{a:mth:bs} and $\ga_{(K)}K\leq n^{3/2}$ we obtain
\begin{equation}\label{p:l:iid:tal:conc:e2}
\ga_{(K)}\Ex\sup_{t\in\Bz_{K}}|\overline{\nu_t}^u|^2
\leq n^{-1}\maxnormsup^2\ga_{(K)}K\Ex[|\iE|^2\Ind{\{|\iE|>\vE n^{1/3}\}}]\leq\maxnormsup^2n^{-1}\Ex(\iE/\vE)^{12}.
\end{equation}
The first rhs term of \eqref{p:l:iid:tal:conc:e1} we bound
employing Talagrand's inequality \eqref{l:talagrand:eq1}
given in Lemma \ref{l:talagrand}. To this end, we need to compute the
quantities $h$, $H$ and $v$ verifying the three required
inequalities. Employing $|\iE^b|\leq 2\vE n^{1/3}$ and Assumption
 \ref{a:mth:bs}  
we obtain
\begin{multline}\label{p:l:iid:tal:conc:h}
\sup_{t\in\Bz_\Di}\VnormInf{\nu_t^b}                               
= \sup_{\iE,\iZ,\iV}\big|(\iE^b+\ceE(\iZ,\iV)+\So(\iZ)-\DiSo(\iZ))^2\sum_{j=1}^\Di\basV_j^2(\iV)\big|^{1/2}\\
\leq\{\vE n^{1/3}+ \VnormInf{\ceE+\So-\DiSo}\}\maxnormsup m^{1/2}=:h.
\end{multline}
Employing in addition the independence of the
sample, the independence between $\iE$ and $(\iZ,\iV)$ implying
$\Ex(\iE^b+\ceE(\iZ,\iV)+\So(\iZ)-\DiSo(\iZ))^2\leq
\Ex(\iE)^2+\Ex(\ceE(\iZ,\iV)+\So(\iZ)-\DiSo(\iZ))^2=\Ex(\iY-\DiSo(\iZ))^2\leq\vB^2$,
and $\Lambda_{\Di}^\ga\geq1$  the quantity $H$ is given by
\begin{multline}\label{p:l:iid:tal:conc:H}
\Ex\sup_{t\in\Bz_\Di}|\overline{\nu_t}^b|^2
\leq
n^{-1}\Ex(\iE^b+\ceE(\iZ,\iV)+\So(\iZ)-\DiSo(\iZ))^2\sum_{j=1}^\Di\basV_j^2(\iV)\leq\maxnormsup^2\Ex(\iY-\DiSo(\iZ))^2\Di
n^{-1}
\\\leq
\maxnormsup^22(\Ex(Y)^2+\VnormZ{\DiSo}^2)\Di n^{-1}\leq
\maxnormsup^2\vB^2\Di\Lambda_\Di^\ga n^{-1}=:H^2.
\end{multline}
It remains to calculate the third quantity $v$. Using successively the
independence of $\iE$ and $(\iZ,\iV)$ and the uniform distribution of
$\iV$  we obtain
\begin{multline}\label{p:l:iid:tal:conc:v}
\sup_{t\in\Bz_\Di}n^{-1}\sum_{i=1}^n\Var(\nu_t(\iE^b_i,\iZ_i,\iV_i)\leq\sup_{t\in\Bz_\Di}\Ex\nu_t^2(\iE^b,\iZ,\iV)\\
=\sup_{t\in\Bz_\Di}\Ex(\iE^b)^2\Ex(\sum_{j=1}^\Di\basV_j(\iV)\fou{t}_j)^2+\sup_{t\in\Bz_\Di}\Ex([\ceE(\iZ,\iV)+\So(\iZ)-\DiSo(\iZ)]\sum_{j=1}^\Di\basV_j(\iV)\fou{t}_j)^2\hfill\\
\leq\vE^2+\VnormInf{\ceE+\So-\DiSo}^2=:v.
\end{multline}
Evaluating \eqref{l:talagrand:eq1} of  Lemma \ref{l:talagrand} with
$h$, $H$, $v$ given by \eqref{p:l:iid:tal:conc:h},
\eqref{p:l:iid:tal:conc:H} and \eqref{p:l:iid:tal:conc:v},
respectively, and exploiting $\maxnormsup^2\vB^2\geq\vE^2$ and
$\VnormInf{\ceE+\So-\DiSo}\leq\DiSoInf$  it follows 
\begin{multline*}
\Ex\vectp{\max_{1\leq \Di\leq K}\{\ga_\Di[\sup_{t\in\Bz_\Di}|{\overline{\nu_t^b}}|^2-
  6\maxnormsup^2\Di n^{-1}\vB^2\Lambda_\Di^\ga]\}}\\
\leq Cn^{-1}\sum_{\Di=1}^{K}\ga_\Di\big\{ \vE^2 (1+\DiSoInf/\vE)^2\exp\Big(-\frac{\Di\Lambda_\Di^a}{6(1+\DiSoInf/\vE)^2}\Big)\\
\hfill+\maxnormsup^2\vE^2( 1 + (\DiSoInf/\vE))^2 \Di
n^{-2+2/3}\exp\big(-n^{1/6}/[100(1+\DiSoInf/\vE)]\big)\big\}
\end{multline*}
Since $\ga_{(K)}K^2\leq n^{3/2}$ and exploiting the definition of $\Psi_\ga$ and
$\Phi_n$  we conclude
\begin{multline*}
\Ex\vectp{\max_{1\leq \Di\leq K}\{\ga_\Di[\sup_{t\in\Bz_\Di}|{\overline{\nu_t^b}}|^2-
  6\maxnormsup^2\Di n^{-1}\sigma^2_{\Di}\Lambda_\Di^\ga]\}} \\
	\leq Cn^{-1} \{\vE^2 \Psi_\ga(1+\DiSoInf/\vE)+  \maxnormsup^2\vE^2\Phi_n(1+\DiSoInf/\vE)\}.
\end{multline*}
We obtain the assertion \eqref{l:iid:tal:conc} by replacing
in  \eqref{p:l:iid:tal:conc:e1}  the last bound and
\eqref{p:l:iid:tal:conc:e2}. Consider now
\eqref{l:iid:tal:prob}. From 
$\Vnorm{\fou{V}_{\uDi}}\leq\sup_{t\in\Bz_{\Di}}|\overline{\nu_t}^b|+\sup_{t\in\Bz_{\Di}}|\overline{\nu_t}^u|$,
$\sup_{t\in\Bz_{\Di}}|\overline{\nu_t}^u|\leq
\sup_{t\in\Bz_{K}}|\overline{\nu_t}^u|$ and $\ga_{\Di}\leq
\ga_{(K)}$  follows for all $1\leq\Di\leq K$
\begin{equation}\label{p:l:iid:tal:prob:e1}
\ProbaMeasure\big(\Vnorm{\fou{V}_{\uDi}}\geq
4c\ga_m^{-1/2}\big)
\leq\ProbaMeasure\big(\sup_{t\in\cB_{\Di}}|\overline{\nu_t}^b| \geq 2c\ga_m^{-1/2}\big)+\ProbaMeasure\big(\sup_{t\in\cB_{K}}|\overline{\nu_t}^u|\geq 2c\ga_{(K)}^{-1/2}\big)
\end{equation}
where we bound separately each rhs term. Consider the second rhs
term. Applying successively Markov's inequality,  $\ga_{(K)}K\leq n$ and
\eqref{p:l:iid:tal:conc:e2}  we obtain
\begin{multline}\label{p:l:iid:tal:prob:e2}
\ProbaMeasure\big(\sup_{t\in\cB_{K}}|\overline{\nu_t}^u|\geq2c\ga_{(K)}^{-1/2}\big)\leq
(2c)^{-2} \ga_{(K)}
\Ex\sup_{t\in\cB_{K}}|\overline{\nu_t}^u|^2\leq
(2c)^{-2}\maxnormsup^2 n^{-3} \Ex(\iE/\vE)^{12}
\end{multline}
The first rhs term of \eqref{p:l:iid:tal:prob:e1} we bound
employing Talagrand's inequality \eqref{l:talagrand:eq2} given in
Lemma \ref{l:talagrand} with
$h$, $H$, $v$ as in by \eqref{p:l:iid:tal:conc:h}--\eqref{p:l:iid:tal:conc:v}, respectively.
Thereby, for all $\lambda>0$ we have
\begin{multline*}
\ProbaMeasure\big(\sup_{t\in\cB_{\Di}}|\overline{\nu_t}^b| \geq 2 \{2\vY^2+2\VnormZ{\DiSo}^2\}^{1/2}
 m^{1/2}\maxnormsup n^{-1/2}+\lambda\big)
\\
\leq3\exp\bigg[\frac{-n\lambda^2}{100(\vE+\DiSoInf)^2}\vee\frac{-n^{2/3}\lambda}{100(\vE+\DiSoInf)\maxnormsup m^{1/2} }\bigg].
\end{multline*}
Since $n \geq 4 c^{-2}
(1+\DiSoInf/\vE)^2\maxnormsup^2  K \ga_{(K)}\geq 4 c^{-2}
\maxnormsup^2 \Di \ga_{\Di}$, letting $\lambda:=c\{2\vY^2+2\VnormZ{\DiSo}^2\}^{1/2}\ga_{\Di}^{-1/2}$ and
using $\{2\vY^2+2\VnormZ{\DiSo}^2\}^{1/2}\geq \vE$, $\ga_m\leq
\ga_{(K)}$ and  $n^{1/2}c \geq  2 (1+\DiSoInf/\vE)\maxnormsup  K^{1/2} \ga_{(K)}^{1/2}$ we obtain
\begin{multline*}
\ProbaMeasure\big(\sup_{t\in\cB_{\Di}}|\overline{\nu_t}^b| \geq 2c \{2\vY^2+2\VnormZ{\DiSo}^2\}^{1/2}\ga_m^{-1/2}\big)\\\hfill
\leq3\exp\bigg[\frac{-nc^2}{100(1+\DiSoInf/\vE)^2\ga_{(K)}}\vee\frac{-n^{1/6}}{50}\bigg]
\end{multline*}
We obtain the assertion \eqref{l:iid:tal:prob} by replacing
in  \eqref{p:l:iid:tal:prob:e1}  the last bound and
\eqref{p:l:iid:tal:prob:e2},  which completes the proof.
\end{proof}
\begin{lem}\label{l:iid:tal2} Let $\ga$ be  a
  non negative sequence  and $\ga_{(K)}:=\max_{1\leq m\leq
 K}\ga_m$. If $n \geq 4 c^{-2}
\maxnormsup^4 K^2 \ga_{(K)}$ for $c>0$ then for all $1\leq\Di\leq K$ holds
  \begin{equation}\label{l:iid:tal2:prob}
    \ProbaMeasure\bigg(\ga_{\Di}\VnormS{\fou{\Xi}_{\uDi}}^2\geq 4c^2\bigg)
\leq3\exp\bigg[\frac{-nc^2}{100\VnormInf{p_{\iZ,\iV}}\ga_{(K)}}\vee\frac{-n^{1/2}}{50}\bigg]
  \end{equation}
  where $p_{Z,W}$
  denotes the joint density of $Z$
  and $W$.
\end{lem}
\begin{proof}[\noindent\textcolor{darkred}{\sc Proof of Lemma \ref{l:iid:tal2}}]
  We are going to apply Talagrand's inequality \eqref{l:talagrand:eq2}
  in Lemma \ref{l:talagrand} using
  $\sup_{t\in\cB_{\Di^2}}|\overline{\nu_t}(x)|^2=\sum_{j,l=1}^{\Di}\fou{\Xi}_{j,l}^2\geq
  \VnormS{\fou{\Xi}_{\uDi}}^2$ where  $\nu_t(\iZ,\iV)=\sum_{j,l=1}^{\Di}\fou{t}_{j,l}\basZ_j(\iZ)\basV_l(\iV)$.
  Therefore, we compute next the quantities $h$,
  $H$
  and $v$
  verifying the three required inequalities. 
  Exploiting the independence and identical distribution of the sample
  and Assumption \ref{a:mth:bs} we obtain
  \begin{gather}\label{p:l:iid:tal2:H}
    \Ex[\sup_{t\in\cB_{\Di^2}}|\overline{\nu_t}|^2]
    \leq\frac{1}{n}\sum_{j,l=1}^{\Di}\Ex\big(\basZ^2_j(\iZ_1)\basV^2_l(\iV_1)\big)\leq\frac{\Di^2}{n}\maxnormsup^4=: H^2,\hfill\\
    \label{p:l:iid:tal2:h}
    \sup_{t\in\cB_{\Di^2}}\VnormInf{\nu_t}^2=\sup_{z,w}\sum_{j,l=1}^{\Di}\basZ^2_j(z)\basV^2_l(w)\leq\Di^2\maxnormsup^4=:h^2, \\
    \label{p:l:iid:tal2:v}
    \sup_{t\in\cB_{\Di^2}}\frac{1}{n}\sum_{i=1}^n
    \Var\big(\nu_t(\iZ_i,\iV_i)\big)\leq
    \VnormInf{p_{\iZ,\iV}}\sup_{t\in\cB_{\Di^2}}
    \Vnorm{\fou{t}_{\uDi}}^2=\VnormInf{p_{\iZ,\iV}}=:v.
  \end{gather}
  Evaluating \eqref{l:talagrand:eq2} of Lemma \ref{l:talagrand} with
  $h$,
  $H$,
  $v$
  given by \eqref{p:l:iid:tal2:H}--\eqref{p:l:iid:tal2:v},
  respectively,  for any $\lambda>0$ we have
  \begin{equation*}
    \ProbaMeasure\big(\sup_{t\in\cB_{\Di^2}}|\overline{\nu_t}|\geq
    2m\maxnormsup^2n^{-1/2}+\lambda\big) \leq3\exp\bigg[\frac{-n\lambda^2}{100\VnormInf{p_{\iZ,\iV}}}\vee\frac{-n\lambda}{100\Di\maxnormsup^2}\bigg].
  \end{equation*}
  Since $ n\geq 4c^{-2} K^2 \ga_{(K)}\maxnormsup^4 \geq
  4c^{-2} m^2 \ga_m\maxnormsup^4$, $1\leq m\leq K$, letting
  $\lambda:=c\ga_m^{-1/2}$ and using $\ga_m\leq\ga_{(K)}$ and $n^{1/2}c\geq 2 K\ga_{(K)}^{1/2}\maxnormsup^2$ we obtain
  \begin{multline*}
    \ProbaMeasure\bigg(\sup_{t\in\cB_{\Di^2}}|\overline{\nu_t}|\geq
    2c\ga_{\Di}^{-1/2}\bigg)\leq3\exp\bigg[\frac{-nc^2}{100\VnormInf{p_{\iZ,\iV}}\ga_{(K)}}\vee\frac{-nc}{100\maxnormsup^2K\ga_{(K)}^{1/2}}\bigg]\\
\leq 3\exp\bigg[\frac{-nc^2}{100\VnormInf{p_{\iZ,\iV}}\ga_{(K)}}\vee\frac{-n^{1/2}}{50}\bigg].
  \end{multline*}
A combination
of the last bound  and  $\sup_{t\in\cB_{\Di^2}}|\overline{\nu_t}(x)|\geq
\VnormS{\fou{\Xi}_{\uDi}}$ implies  the assertion, which completes
the proof.
\end{proof}
\subsection{Proof of Theorem \ref{t:dep:ora}}\label{a:ora:dep}
Throughout this section we suppose that  $\set{(\iZ_i,\iV_i)}_{i\in\Zz}$
  is a stationary absolutely
regular process with mixing coefficients  $(\beta_k)_{k\geq0}$. The
sample $\set{(\iY_i,\iZ_i,\iV_i)}_{i=1}^n$ still  obeys the model
(\ref{eq:model}--\ref{eq:model2}) and the Assumption \ref{a:mth:rv}, i.e.,
 $\{\iE_i:=\iNo_i-\ceE(\iZ_i,\iV_i)\}_{i=1}^n$    forms an \iid sample   independent of $\{(\iZ_i,\iV_i)\}_{i=1}^n$. We shall prove below the Propositions \ref{p:dep:Eset} and
\ref{p:dep:Rest} which are used in the proof of Theorem \ref{t:dep:ora}.
In the proof the propositions we refer to three technical Lemma
(\ref{l:dep:sets} -- \ref{l:dep:tal2}) which are shown in the end of this section.
Moreover, we make use of  functions
  $\bPsi,\bPhia,\bPhib,\bPhic,\bPhid,\bPhie:\Rz_+\to\Rz$
  defined by 
  \begin{multline}\label{de:dep:ora}
    \bPsi(x)=\Psi(x)=\sum\nolimits_{m\geq 1} x m^{1/2} \VnormS{\DifOp^{-1}}^2 \exp(- m^{1/2}
    \Lambda_\Di^\Op/(48x) ),\\
    \bPhia(x)= x n \exp(- \gauss{n^{1/8}}\log(n)/(48x)),\\
    \bPhib(x)=n^{7/6}x^2\exp(-n^{1/6}/(200x)),\\
    \bPhic(x)= n^3\exp(-n(\Mot)^{-1/2}(\Delta^\Op_{\Mot})^{-1}/(204800x)),\\
    \bPhid(x)=n^3\exp(-n(\Mot+1)^{-1}(\Delta_{\Mot+1}^\Op)^{-1}/(51200x))\hfill\\
    \bPhie(x)= x n \exp(-(\Mot)^{1/2}\log(n)/(48x) ).\hfill
  \end{multline}
  We shall emphasise that the functions are non decreasing in $x$
  and for all $x>0$,
  $\bPsi(x)<\infty$, $\bPhia(x)=o(1)$
  and $\bPhib(x)=o(1)$
  as $n\to\infty$. Moreover, if
  $\log(n)(\Mot+1)^2\Delta_{\Mot+1}^\Op=o(n)$  as $n\to\infty$ then there exists an integer $n_o$   such that 
  \begin{equation}\label{de:dep:no}
1\geq \sup_{n\geq n_o} \big\{1024\maxnormsup^4(6+8(\DiSoInf/\vE)^2\gB)(\Mot+1)^2\Delta_{\Mot+1}^\Op
n^{-1}\big\}.
\end{equation}
If in addition $q_n(\Mot+1)(\Delta^\Op_{\Mot+1})^{1/2}(\log n)=o(n^{2/3})$ then
we have also for all $x>0$,  $\bPhic(x)=o(1)$, $\bPhid(x)=o(1)$
  and $\bPhie(x)=o(1)$  as $n\to\infty$. Consequently, under
  Assumption \ref{a:mth:rv} and \ref{a:mth:bs}  there exists 
a finite constant $\SoRaC$  such that for all $n\geq1$,
\begin{multline}\label{de:dep:ora:Sigma}
\SoRaC\geq\big\{n_o^2\bigvee
\bPsi\big(1+(\DiSoInf/\vE)^2\maxnormsup\gB^{1/2}\big)\bigvee\bPhia\big(1+(\DiSoInf/\vE)^2\maxnormsup\gB^{1/2}\big)\bigvee\bPhib\big(1+\DiSoInf/\vE\big)\\\bigvee\bPhic(1+(\DiSoInf/\vE)^2\maxnormsup\gB^{1/2})\bigvee\bPhid(\VnormInf{p_{\iZ,\iV}}\gB^{1/2}\maxnormsup^2)\bigvee\bPhie\big(1+(\DiSoInf/\vE)^2\maxnormsup\gB^{1/2}\big)\\
\bigvee\Ex(\iE/\vE)^8
\bigvee(\DiSoInf/\vE)^8\gB
\bigvee(\maxnormsup/\vE)^2 \Ex(\iE/\vE)^{12}\big\}.
\end{multline}
\begin{proof}[\noindent\textcolor{darkred}{\sc Proof of Theorem \ref{t:dep:ora}}]
 The proof follows line by line the proof of Theorem
 \ref{t:iid:ora}. By using Proposition  \ref{p:dep:Rest} rather than Proposition 
 \ref{p:iid:Rest} we obtain  similar to \eqref{ri:dec:e1} for all $n\geq1$
  \begin{multline}\label{ri:dec:e1:dep}
    \Ex\VnormZ{\hDiSo[\hDi]-\So}^2\leq \Ex\left(\Ind{\Eset}\VnormZ{\hDiSo[\hDi]-\So}^2\right)\\+C\;n^{-1}\;\maxnormsup^2 \{\vE^2+\DiSoTwo\} (1+(\DiSoInf/\vE)^2\gB)[\SoRaC\vee n^3\exp(-n^{1/6}q^{-1}/100) \vee  n^4 q^{-1}\beta_{q+1}].
  \end{multline}
Consider the first rhs term. On the event $\Eset$ defined in
\eqref{a:not:de:ev},  on which  the  quantities $\hpen$ and $\Mh$ are close to their theoretical
counterparts $\pen$,  $\Mut$ and $\Mot$  defined in
\eqref{a:not:de:no},  the upper bound given in \eqref{t:iid:ora:key:arg} implies
  \begin{equation*}
    \VnormZ{\hDiSo[\hDi]-\So}^2\Ind{\Eset}\leq 582 [\biasnivSo[\aDi]\vee\pen[\aDi]] +42 \max_{\aDi\leq k\leq \Mot}\vectp{ \VnormZ{\hDiSo[k]-\DiSo[k]}^2 - \pen[k]/6}.
  \end{equation*}
Keeping in mind that
$\pen[k]=288\penDep\maxnormsup^2\vB[k]^2\delta_k^\Op n^{-1}$  with
$\delta_k^\Op=k\Lambda_k^\Op\Delta_{k}^\Op$, $\penDep\leq 8(1 +(\DiSoInf/\vE)^2 \gB)$ and $\vB[k]^2\leq2(\vE^2+3\DiSoTwo)$  we derive in
  Proposition \ref{p:dep:Eset} below an upper bound for the
  expectation of the second rhs term, the remainder term,
  in the last display. Thereby, we obtain
 \begin{multline*}
    \Ex\big(\Ind{\Eset}\VnormZ{\hDiSo[\hDi]-\So}^2\big)\leq 
    C\;\big\{[\biasnivSo[\aDi]\vee n^{-1}\delta_{\aDi}^\Op] +
n^{-1}[\SoRaC\vee n^3\exp(-n^{1/6}q^{-1}/100) \vee  n^4 q^{-1}\beta_{q+1}]\big\} \\\times\maxnormsup^2 (1+\vE^2+\DiSoTwo)(1+(\DiSoInf/\vE)^2 \gB).
  \end{multline*}
Replacing in  \eqref{ri:dec:e1:dep} the first rhs by the last upper bound
we obtain the assertion of the theorem, which completes the proof.
 \end{proof}
\begin{prop}\label{p:dep:Eset} Under the assumptions of Theorem \ref{t:dep:ora}
  there exists a numerical constant $C$ such that for all $1\leq q\leq n$ 
\begin{multline*}\Ex\set{\max_{\aDi\leq k\leq
      \Mot}\vectp{\VnormZ{\hDiSo-\DiSo}^2 -48\maxnormsup^2\vB^2\penDep\Di
      \Lambda_\Di^\Op\Delta_{\Di}^\Op n^{-1}}}\\\leq
C\;n^{-1}\;[\SoRaC\vee n^3\exp(-n^{1/6}q^{-1}/100) \vee  n^4 q^{-1}\beta_{q+1}]\;\maxnormsup^2 \{\vE^2+\DiSoTwo\} (1+(\DiSoInf/\vE)^2\gB).
\end{multline*}
\end{prop}
\begin{proof}[\noindent\textcolor{darkred}{\sc Proof of Proposition \ref{p:dep:Eset}}]  The proof
  follows along the lines of the proof of Proposition
  \ref{p:iid:Eset}. Similarly to \eqref{p:p:iid:Eset:e1} we have
 \begin{multline*}
\Ex\set{\max_{\aDi\leq \Di\leq\Mot}\vectp{\VnormZ{\hDiSo-\DiSo}^2
    -48\maxnormsup^2 \penDep\vB^2\Di\Lambda_\Di^\Op\Delta_{\Di}^\Op n^{-1}}}\\
\hfill\leq 2\Ex\set{ \max_{\aDi\leq\Di\leq\Mot} \VnormS{\DifOp^{-1}}^2\vectp{\Vnorm{\fou{V}_{\uDi}}^2
    -24 \maxnormsup^2 \penDep\vB^2\Di\Lambda_\Di^\Op n^{-1}}}\\\hfill +
 \Ex\set{n\vectp{\Vnorm{\fou{V}_{\underline{\Mot}}}^2-24\maxnormsup^2\vB[\Mot]^2\penDep \Mot \log(n) n^{-1}}}
\\+24\maxnormsup^2\vB[\Mot]^2\penDep\Mot  \log(n)P(\bigcup_{k=1}^{\Mot}\Xiset[k]^c)
 +\max_{\aDi\leq \Di\leq\Mot}\VnormZ{\DiSo}^2P(\bigcup_{k=1}^{\Mot}\hOpset[k]^c)
\end{multline*}
where we bound separately each of the four rhs
terms. Employing \eqref{l:dep:tal:conc} in
Lemma \ref{l:dep:tal} with  $k=\aDi$, $K=\Mot$, sequence $a=(a_m)_{m\geq1}$ given by $a_m=\VnormS{\DifOp^{-1}}^2$ and
$a_m=n\Ind{\{m=\Mot\}}$,
respectively.  Keeping in mind that in both cases
$\ga_{(K)}K^2\leq n^{3/2}$  we bound  the first and second rhs term as follows
 \begin{multline*}
\Ex\set{\max_{\aDi\leq \Di\leq\Mot}\vectp{\VnormZ{\hDiSo-\DiSo}^2
    -48\maxnormsup^2\penDep  \vB^2\Di\Lambda_\Di^\Op\Delta_{\Di}^\Op n^{-1}}}\\
\leq Cn^{-1} \maxnormsup^2 \big\{\vE^2\bPsi\big(1+(\DiSoInf/\vE)^2\maxnormsup\gB^{1/2}\big)+\vE^2\bPhib\big(1+\DiSoInf/\vE\big)\\\hfill+\vE^2\bPhie\big(1+(\DiSoInf/\vE)^2\maxnormsup\gB^{1/2}\big)
   +  \vE^2 (1 + \DiSoInf/\vE)^2  n^{7/3} q^{-1} \beta_{q+1}   +\Ex(\iE/\vE)^6\big\}
\\
+24\maxnormsup^2\penDep \vB[\Mot]^2\Mot\log(n) P(\bigcup_{k=\aDi}^{\Mot}\Xiset[k]^c)
 +\max_{\aDi\leq \Di\leq\Mot}\VnormZ{\DiSo}^2P(\bigcup_{k=\aDi}^{\Mot}\hOpset[k]^c)
\end{multline*}
with $\bPsi$, $\bPhib$ and $\bPhie$ as in \eqref{de:dep:ora}, i.e., $\bPsi(x)=\sum\nolimits_{m\geq 1} x m^{1/2} \VnormS{\DifOp^{-1}}^2 \exp(- m^{1/2}
    \Lambda_\Di^\Op/(48x) )$, $\bPhib(x)=n^{7/6}x^2\exp(-n^{1/6}/(200x))$ and $\bPhie(x)= x n \exp(-(\Mot)^{1/2}\log(n)/(48x) )$, $x>0$. 
Exploiting that $\vB^2[\Mot]\leq 2(\vE^2+3\DiSoTwo)$, $\penDep\leq 8(1 +
(\DiSoInf/\vE)^2 \gB)$,  $\Mot\log(n)\leq
n$   and $\max_{\aDi\leq \Di\leq\Mot}\VnormZ{\DiSo}^2\leq\DiSoTwo$,
replacing the probability $P(\bigcup_{k=\aDi}^{\Mot}\hOpset[k]^c)$ and $P(\bigcup_{k=\aDi}^{\Mot}\Xiset[k]^c)$ by its upper bound  
given in \eqref{l:dep:sets:hO} and  \eqref{l:dep:sets:B} in Lemma
\ref{l:dep:sets}, respectively, and  employing the definition
of $\SoRaC$ as in \eqref{de:dep:ora:Sigma} we obtain the result of the proposition, 
 which completes the proof.
\end{proof}
\begin{prop}\label{p:dep:Rest} Under the assumptions of Theorem \ref{t:dep:ora}
 there exists a numerical constant $C$ such that  for all $1\leq q\leq n$ 
 \begin{multline*}
\Ex\big(\VnormZ{\hDiSo[\hDi]-\So}^2\1_{\Eset^c}\big)\leq 
C\;n^{-1} [\SoRaC\vee n^3\exp(-n^{1/6}q^{-1}/100) \vee  n^4
q^{-1}\beta_{q+1}]\\
\times \;\maxnormsup^2 \{\vE^2+\DiSoTwo\} (1+(\DiSoInf/\vE)^2\gB).
\end{multline*}
\end{prop}
\begin{proof}[\noindent\textcolor{darkred}{\sc Proof of Proposition \ref{p:dep:Rest}}]
  The proof follows along the lines of the proof of Proposition
  \ref{p:iid:Rest}.  As in \eqref{p:p:iid:Rest:e1} with
  $\maxDi:=\DiMa$ and $\penDep:=8(1 + (\DiSoInf/\vE)^2\gB)$ we obtain
\begin{multline}\label{p:p:dep:Rest:e1}
\Ex\big(\VnormZ{\hDiSo[\hDi]-\So}^2\1_{\Eset^c}\big)\leq 
3\Ex\set{n\vectp{\Vnorm{\fou{V}_{\umaxDi}}^2-24\maxnormsup^2\penDep\vB[\maxDi]^2\maxDi
    \log(n)n^{-1}}} \\+\{72\maxnormsup^2 \maxDi \penDep \vB[\maxDi]^2 \log(n) + 6\DiSoTwo\}P(\Eset^c).
\end{multline}
Considering the first rhs term from \eqref{l:dep:tal:conc} in
Lemma \ref{l:dep:tal} with sequence $\ga=(\ga_m)_{m\geq1}$ given by
$\ga_m=n\Ind{\{m=K\}}$ and $k=K=\maxDi:=\DiMa$  where
$K^2\ga_{(K)}\leq n^{3/2}$ it  follows that
\begin{multline*}
\Ex\big(\VnormZ{\hDiSo[\hDi]-\So}^2\1_{\Eset^c}\big)\leq
 Cn^{-1} \maxnormsup^2 \big\{\vE^2\bPhia\big(1+(\DiSoInf/\vE)^2\maxnormsup\gB^{1/2}\big)+\vE^2\bPhib\big(1+\DiSoInf/\vE\big)\\
  \hfill +  \vE^2 (1 + \DiSoInf/\vE)^2  n^{7/3} q^{-1} \beta_{q+1}   +\Ex(\iE/\vE)^{12}\big\}\\+\{72\maxnormsup^2 \maxDi \penDep \vB[\maxDi]^2 \log(n) + 6\DiSoTwo\}P(\Eset^c)
\end{multline*}
with $\bPhia$ and $\bPhib$ as in \eqref{de:dep:ora}, i.e.,
$\bPhia(x)= x n \exp(- \gauss{n^{1/8}}\log(n)/(48x))$ and $\bPhib(x):=n^{7/6}x^2\exp(-n^{1/6}/(100x))$, $x>0$. 
Exploiting further the definition of $\SoRaC$ as in
\eqref{de:dep:ora:Sigma} and that $\vB[\maxDi]^2\leq
2\{\vE^2+3\DiSoTwo\}$,  $\penDep= 8(1 + (\DiSoInf/\vE)^2 \gB)$ and
$M\log(n)\leq n$ the result of the proposition follows now  by
replacing the probability $P(\Eset^c)$ by its upper bound  
given in \eqref{l:dep:sets:E} in Lemma \ref{l:dep:sets}, which
completes the proof.
\end{proof}
\begin{lem}\label{l:dep:sets}
Under the assumptions of Theorem \ref{t:dep:ora} there exists a numerical constant $C$ such that for all $1\leq q\leq n$
\begin{align}
&\ProbaMeasure\big(\Aset^c)=\ProbaMeasure\big(\{1/2\leq\hsigma^2_Y/\sigma_Y^2\leq
  3/2\}^c\big)
\leq C\; \SoRaC\; n^{-2},
 \label{l:dep:sets:A}\\
&\ProbaMeasure\big(\Bset^c\big)=\ProbaMeasure\big(\bigcup_{\Di=1}^{\Mot+1}\Xiset^c\big)
\leq C\; [\SoRaC\vee n^3\exp(-n^{1/2}q^{-1}/50) \vee n^4 q^{-1}\beta_{q+1}]\; n^{-2}, \label{l:dep:sets:B}\\
&\ProbaMeasure\big(\Cset^c\big)
  \leq C\;[\SoRaC\vee  n^3\exp(-n^{1/6}q^{-1}/100) \vee n^4 q^{-1}\beta_{q+1}]\; n^{-2}, \label{l:dep:sets:C}\\
&\ProbaMeasure\big(\Eset^c\big)
  \leq C\;[\SoRaC\vee  n^3\exp(-n^{1/6}q^{-1}/100)\vee n^4 q^{-1}\beta_{q+1}]\; n^{-2}, \label{l:dep:sets:E}\\
&\ProbaMeasure\big(\bigcup_{\Di=1}^{\Mot}\hOpset\big)\leq
  C\;[\SoRaC\vee n^3\exp(-n^{1/2}q^{-1}/50) \vee  n^4 q^{-1}\beta_{q+1}]\; n^{-2}. \label{l:dep:sets:hO}
\end{align}
\end{lem}
\begin{proof}[\noindent\textcolor{darkred}{\sc Proof of Lemma \ref{l:dep:sets}}]Consider
  \eqref{l:dep:sets:A}.  We note that $\iY=\iE+\ceE(\iZ,\iV)+\So(\iZ)$ and
$\vY^2=\vE^2+\VnormZ{\ceE+\So}^2$. Thereby, setting
$\eta:=\ceE(\iZ,\iV)+\So(\iZ)$ with
$\sigma_\eta^2:=\Ex\eta^2=\Vnorm[Z,W]{\ceE+\So}^2$, where $\eta$ and
$\iE$ are independent, we obtain
  \begin{multline}\label{p:l:dep:sets:e1}
   P\big(|n^{-1}\sum_{i=1}^n\iY_i^2-\vY^2|>\vY^2/2\big)
   \leq   P\big(|n^{-1}\sum_{i=1}^n(\iE_i^2-\vE^2)|> \vY^2/6\big)\\+
   P\big(| n^{-1}\sum_{i=1}^n(\eta_i^2-\sigma_\eta^2)|>\vY^2/6\big)+
    P\big(|n^{-1}\sum_{i=1}^n2\eta_i\iE_i|>\vY^2/6\big)
  \end{multline}
and we bound each rhs term separately.  Consider the first rhs
term. Since
  $\iE_1^2-\vE^2,\dotsc,\iE_n^2-\vE^2 $ are independent and
  and centred random variables with $\Ex(\iE_i^2-\vE^2)^{4}\leq C \Ex(\iE)^{8}$ it follows from Theorem 2.10 in
\cite{Petrov1995}  that $\Ex(n^{-1}\sum_{i=1}^n\iE_i^2-\vE^2)^{4}\leq C
n^{-k} \Ex(\iE)^{8}$. Employing Markov's inequality, the last bound
and $\vY^2\geq\vE^2$ we have
$P\big(|n^{-1}\sum_{i=1}^n\iE_i^2-\vE^2|>\vY^2/6\big)\leq C
n^{-2} \Ex(\iE/\vE)^{8}$. Consider the second rhs
term in \eqref{p:l:dep:tal:conc:e1}. From Lemma \ref{l:mixing:V2} with
$h(\iZ,\iV)=\{\ceE(\iZ,\iV)+\So(\iZ)\}^2=\eta^2$ follows
$\Ex(n^{-1}\sum_{i=1}^n\eta_i^2-\sigma_\eta^2)^{4}\leq n^2
\VnormInf{\ceE+\So}^8\gB$ and hence, applying Markov's inequality, $\vY^2\geq\vE^2$ and $\DiSoInf\geq\VnormInf{\ceE+\So}$ we have
$P\big(|n^{-1}\sum_{i=1}^n\eta_i^2-\sigma_\eta^2|>\vY^2/6\big)\leq C
n^{-2}(\DiSoInf/\vE)^{8}\gB$. It remains to consider the last rhs
term in \eqref{p:l:dep:tal:conc:e1}. Keeping in mind, that
$\{\iE_i\}_{i=1}^n$ are \iid and independent of $\{\eta_i\}_{i=1}^n$ from Theorem 2.10 in
\cite{Petrov1995} follows
\begin{multline*}
  \Ex|\sum_{i=1}^n\eta_i\iE_i|^4=\Ex(\Ex[|\sum_{i=1}^n\eta_i\iE_i|^4|\eta_1,\dotsc,\eta_n])
  \leq
  n^2\Ex(\Ex[|\eta_1\iE_1|^4|\eta_1,\dotsc,\eta_n])=n^2\Ex(\eta)^4\Ex(\iE)^4\\\leq
  (1/2)n^2\{\Ex(\eta)^8+\Ex(\iE)^8\} \leq (1/2)n^2\{\VnormInf{\ceE+\So}^8+\Ex(\iE)^8\}
\end{multline*}
and hence, Markov's inequality, $\vY^2\geq\vE^2$
and $\DiSoInf\geq\VnormInf{\ceE+\So}$
imply together
$P\big(|\sum_{i=1}^n\eta_i\iE_i|>n\vY^2/12\big)\leq C
n^{-2}\{(\DiSoInf/\vE)^8+\Ex(\iE/\vE)^8\}$. Replacing in
\eqref{p:l:dep:tal:conc:e1} each rhs term by its respective bound, we
obtain
$P\big(|n^{-1}\sum_{i=1}^n\iY_i^2/\vY^2-1|>1/2\big)\leq C
n^{-2}\{(\DiSoInf/\vE)^8+\Ex(\iE/\vE)^8\}$. Thereby, the assertion
\eqref{l:dep:sets:A} follows from the last bound by employing the
definition of $\SoRaC$
given in \eqref{de:dep:ora:Sigma} and by exploiting that
$\{1/2\leq\hvY^2/\vY^2\leq3/2\}^c\subset
\{|n^{-1}\sum_{i=1}^n\iY_i^2/\vY^2-1|>1/2\}$. Consider
\eqref{l:dep:sets:B}--\eqref{l:dep:sets:E}.  Let $\ga$
be a sequence given by $\ga_m=\VnormS{\DifOp^{-1}}^2$
where $\ga_{(m)}=\Delta^T_{m}$
and $n_o$
an integer satisfying \ref{de:dep:no}, that is,
$n\geq
1024\maxnormsup^4(6+8(\DiSoInf/\vE)^2\gB)(\Mot+1)^2\Delta_{\Mot+1}^\Op$
for all $n> n_o$.
We distinguish in the following the cases $n\leq n_o$
and $n > n_o$.
Consider \eqref{l:dep:sets:B}. Obviously, we have
$\ProbaMeasure\big(\Bset^c\big)\leq n^{-2}n_o^2$
for all $1\leq n\leq n_o$.
On the other hand, given $n\geq n_o$
and, hence
$n\geq512\maxnormsup^4(1+4\gB)(\Mot+1)^2\Delta_{\Mot+1}^\Op=8c^{-2}\maxnormsup^4
(1+4\gB)K^2\ga_{(K)}$ with sequence $\ga=(\VnormS{\DifOp^{-1}}^2)_{m\geq1}$,
integer $K=\Mot+1$
and constant $c=1/8$
we obtain from \eqref{l:dep:tal2:prob} in Lemma \ref{l:dep:tal2} for
all $1\leq m\leq (\Mot+1)$ \begin{multline*}
  \ProbaMeasure(\Xiset^c)=\ProbaMeasure(\VnormS{\DifOp^{-1}}^2\VnormS{\fou{\Xi}_{\Di}}^2\geq1/16)
  \\\leq6\exp\bigg[\frac{-n}{51200\maxnormsup^2\VnormInf{p_{\iZ,\iV}}\gB^{1/2}(\Mot+1)\Delta_{\Mot+1}^\Op}\vee\frac{-n^{1/2}}{50q}\bigg]+
nq^{-1}\beta_{q+1}
\end{multline*}
and hence,  given  $\bPhid$ as in \eqref{de:dep:ora} and
$\Mot+1\leq n$  it follows 
\begin{multline*}
\ProbaMeasure\big(\Bset^c\big)\leq (\Mot+1)\max_{1\leq m\leq
  (\Mot+1)}\ProbaMeasure(\Xiset^c)\\\leq 6\;\big\{
\bPhid(\maxnormsup^2\VnormInf{p_{\iZ,\iV}}\gB^{1/2})\vee
n^3\exp(-n^{1/2}q^{-1}/50)\vee n^4q^{-1}\beta_{q+1}\big\}\;n^{-2}.
\end{multline*}
By combination of the
two cases and employing the definition of
$\SoRaC$ given in \eqref{de:dep:ora:Sigma} we obtain
\eqref{l:dep:sets:B}. The proof of
\eqref{l:dep:sets:C} follows a long the lines of the proof of
\eqref{l:dep:sets:B} using \eqref{l:dep:tal:prob} in Lemma
\ref{l:dep:tal} rather than  \eqref{l:dep:tal2:prob} in Lemma
\ref{l:dep:tal2}. Precisely, if $1\leq n\leq n_o$ we have $P(\Cset^c)\leq
n^{-2}n_o^2$, while given $n> n_o$ and, hence 
$n \geq 1024\maxnormsup^2 [6+8(\DiSoInf/\vE)^2\gB] (\Mot+1) \Delta_{\Mot+1}^\Op =
4c^{-2}\maxnormsup^2 2[3+4(\DiSoInf/\vE)^2\gB] K\ga_{(K)}$  with sequence $\ga_m=\VnormS{\DifOp^{-1}}^2$, integer
$K=\Mot$ and constant $c=1/16$ from  \eqref{l:dep:tal:prob} in Lemma
\ref{l:dep:tal} we obtain for all $1\leq m\leq \Mot$
\begin{multline*}
\ProbaMeasure(\Vnorm{\DifOp^{-1}\fou{V}_{\uDi}}^2>\tfrac{1}{8}(\VnormZ{\DiSo}^2+
\vY^2))\leq\ProbaMeasure(\ga_{m}\Vnorm{\fou{V}_{\uDi}}^2>16
c^2 \{2\VnormZ{\DiSo}^2+ 2\vY^2\})\\
\hfill\leq
6\exp\bigg[\frac{-n}{51200(1+4(\DiSoInf/\vE)^2\maxnormsup\gB^{1/2})(\Mot)^{1/2}\Delta^\Op_{\Mot}}\vee\frac{-n^{1/6}}{100q}\bigg]\\+ n q^{-1}\beta_{q+1} +
  64(\maxnormsup/\vE)^2n^{-3}\Ex(\iE/\vE)^{12}
\end{multline*}
Exploiting the definition of $\bPhic$ given in \eqref{de:dep:ora}
implies $\ProbaMeasure\big(\Cset^c\big)\leq 6 
\{n^3\exp(-n^{1/6}q^{-1}/100)\}\vee\bPhic(1+(\DiSoInf/\vE)^2\maxnormsup\gB^{1/2})n^{-2}+ n q^{-1}\beta_{q+1}+ 64 (\maxnormsup^2/\vE^2) \Ex(\iE/\vE)^{12}n^{-2} $
The assertion \eqref{l:dep:sets:C} follows employing the definition of
$\SoRaC$ given in \eqref{de:dep:ora:Sigma}. Due to Lemma \ref{app:pre:l5} it holds
$\ProbaMeasure\big(\Eset^c\big)\leq \ProbaMeasure\big(\Aset^c\big)
+\ProbaMeasure\big(\Bset^c\big)
+\ProbaMeasure\big(\Cset^c\big)$. Therefore, the assertion
\eqref{l:dep:sets:E} follows from
\eqref{l:dep:sets:A}--\eqref{l:dep:sets:C}. Consider
\eqref{l:dep:sets:hO}. We distinguish again the two cases $n\leq n_o$
and $n>n_o$, where $\ProbaMeasure\big(\bigcup_{m=1}^{\Mot}\hOpset^c\big)\leq n^{-2}n_o^2$
for all $1\leq n\leq n_o$. On the other hand, for all $n>n_o$ we have 
$n\geq (16/9)\Delta_{\Mot+1}^\Op\geq
(16/9)\VnormS{\DifOp^{-1}}^2$ for all $1\leq m\leq \Mot$, and hence from Lemma
\ref{app:pre:l6} follows $\bigcup_{m=1}^{\Mot}\hOpset^c\subset\bigcup_{m=1}^{\Mot}\Xiset^c\subset\Bset^c$ for all
$n>n_o$. Thereby, \eqref{l:dep:sets:B} implies \eqref{l:dep:sets:hO}  for all
$n>n_o$. By combination of the two cases we obtain
\eqref{l:dep:sets:hO}, which completes the proof. 
\end{proof}
\begin{lem}\label{l:dep:tal} Given  a
  non negative sequence  $\ga:=(\ga_{\Di})_{\Di\in\Nz}$ let
  $\Lambda_\Di^\ga:=\Lambda_\Di(\ga)$ as in \eqref{de:delta},
  $\ga_{(m)}:=\max_{1\leq k\leq m}\ga_k$, for any $x>0$,
  $\Phi_n(x):=n^{7/6}x^2\exp(-n^{1/6}/(200x))$ and
  $\Psi_\ga(x):=\sum_{m\geq 1} x m^{1/2} \ga_m \exp(- m^{1/2} \Lambda_\Di^\ga/(48x) )<\infty$, which by
  construction always exists.  If $\gB\geq  2\sum_{k=0}^\infty(k+1)\beta_k$, $\sup_{\Di\geq
      1}\VnormInf{\ceE+\So-\DiSo}\leq\DiSoInf<\infty$,  $\ga_{(K)}K^2\leq n^{3/2}$,  $M_k:=\gauss{(\vE/\DiSoInf)^2\pdZWInf k}$ and $\penDep\geq 2[3 + 4(\DiSoInf/\vE)^2 \sum_{j>M_k}\beta_j]$  then  there exists a
 finite numerical constant $C>0$ such that
\begin{multline}\label{l:dep:tal:conc}
\Ex\vectp{\max_{k\leq\Di\leq K}\ga_\Di[\Vnorm{\fou{V}_{\uDi}}^2-24\maxnormsup^2\penDep\sigma_{\Di}^2\Di
  \Lambda_\Di^\ga n^{-1}]}\\
\leq Cn^{-1} \maxnormsup^2 \big\{\vE^2\Psi_\ga\big(1+(\DiSoInf/\vE)^2\maxnormsup\gB^{1/2}\big)+\vE^2\Phi_n\big(1+\DiSoInf/\vE\big)\\
   +  \vE^2 (1 + \DiSoInf/\vE)^2  n^{17/6} q^{-1} \beta_{q+1}   +\Ex(\iE/\vE)^{12}\big\}
\end{multline}
Moreover, if   $n \geq 4 c^{-2}\maxnormsup^2 2[3 + 4(\DiSoInf/\vE)^2 \gB]  K \ga_{(K)}$  and $c>0$ then for all $1\leq\Di\leq K$ holds
\begin{multline}\label{l:dep:tal:prob}
\ProbaMeasure\big(\ga_m\Vnorm{\fou{V}_{\uDi}}^2\geq 16c^2\{\vY^2+\VnormZ{\DiSo}^2\}\big)\\
\hfill\leq6\exp\bigg[\frac{-nc^2}{200\{1+
  4(\DiSoInf/\vE)^2\maxnormsup\gB^{1/2}\}K^{1/2}\ga_{(K)}}\vee\frac{-n^{1/6}}{100q}\bigg]\\
\hfill + n q^{-1}\beta_{q+1} +
  (2c)^{-2}(\maxnormsup/\vE)^2n^{-3}\Ex(\iE/\vE)^{12}.
      \end{multline}
\end{lem}
\begin{proof}[\noindent\textcolor{darkred}{\sc Proof of Lemma \ref{l:dep:tal}}]
The proof follows a
  long the lines of the proof of Lemma \ref{l:iid:tal} and, hence
  recall  the decomposition \eqref{p:l:iid:tal:conc:e1}, where  the second rhs term is still bound by \eqref{p:l:iid:tal:conc:e2}
 employing that $\{\iE_i\}_{i=1}^n$ forms an \iid sample independent of
 the instruments $\{\iV_i\}_{i=1}^n$. Precisely, we have 
\begin{multline}\label{p:l:dep:tal:conc:e1}
\Ex\vectp{
  \max_{k\leq\Di\leq K}\set{\ga_\Di[\Vnorm{\fou{V}_{\uDi}}^2
    -24\maxnormsup^2\vB^2 \Di \Lambda_\Di^\ga n^{-1}]}}\\
\leq2\Ex\vectp{\max_{k\leq \Di\leq K}\{\ga_\Di[\sup_{t\in\Bz_\Di}|{\overline{\nu_t}^b}|^2-
  12\maxnormsup^2\vB^2\Di \Lambda_\Di^\ga n^{-1}]\}}+\maxnormsup^2n^{-1}\Ex(\iE/\vE)^{12}.
\end{multline}
Therefore, it remains to
 consider the  first rhs term in  \eqref{p:l:dep:tal:conc:e1}.
Consider $(\iZ_i,\iV_i)_{i\geq 1}=(E_l,O_l)_{l\geq 1}$ and
$(\couZ_i,\couV_i)_{i\geq 1}=(\couE_l,\couO_l)_{l\geq 1}$  obeying the
coupling properties \ref{dd:as:cou1}, \ref{dd:as:cou2} and
\ref{dd:as:cou3}. Moreover, introduce analogously
$(\iE^b_i)_{i\geq1}=(\vec{\iE}^{\;be}_l,\vec{\iE}^{\;bo}_l)_{l\geq1}$. 
Setting
$\vec{v}_t(\vec{e},\vec{z},\vec{w})=q^{-1}\sum_{i=1}^qv_t(e_i,z_i,w_i)$ for $n=2pq$ follows
\begin{equation*}
 \overline{\nu_t}^b =
 \tfrac{1}{2}\big\{\tfrac{1}{p}\sum_{l=1}^p\{\vec{v}_t(\vec{\iE}^{\;be}_l,E_l)-\Ex\vec{v}_t(\vec{\iE}^{\;be}_l,E_l)
 \}+ \tfrac{1}{p}\sum_{l=1}^p\{\vec{v}_t(\vec{\iE}^{\;bo}_l,O_l)-\Ex\vec{v}_t(\vec{\iE}^{\;bo}_l,O_l)\}\big\}=:\tfrac{1}{2}\{\overline{\nu_t}^{be}+\overline{\nu_t}^{bo}\}. 
\end{equation*}
Considering the random variables $(\couZ_i,\couV_i)_{i\geq
  1}=(\couE_l,\couO_l)_{l\geq 1}$ rather than $(\iZ_i,\iV_i)_{i\geq 1}=(E_l,O_l)_{l\geq 1}$ we introduce in addition $ \cou{{\overline{\nu_t}^b}} =\tfrac{1}{2}\{\cou{{\overline{\nu_t}^{be}}}+\cou{{\overline{\nu_t}^{bo}}}\}$.
Keeping in mind that $\overline{\nu_t}^{be}$ and
  $\overline{\nu_t}^{bo}$ (respectively,  $\cou{{\overline{\nu_t}^{be}}}$ and
  $\cou{{\overline{\nu_t}^{bo}}}$) are identically distributed, the
  first rhs term in  \eqref{p:l:dep:tal:conc:e1}  is bounded by
\begin{multline}\label{p:l:dep:tal:conc:e2}
\Ex\vectp{\max_{k\leq \Di\leq
    \Mot}\{\ga_{\Di}[\sup_{t\in\Bz_\Di}|{\overline{\nu_t}^b}|^2-
  12\maxnormsup^2\penDep\vB\Di\Lambda^\ga_{\Di}n^{-1}]\}}\\
\hfill\leq2\Ex\vectp{\max_{k\leq \Di\leq
    \Mot}\{\ga_{\Di}[\sup_{t\in\Bz_\Di}|\cou{{\overline{\nu_t}^{be}}}|^2-
  6\maxnormsup^2\penDep\vB\Di\Lambda^\ga_{\Di}n^{-1}]\}}\\
  +2\ga_{K}\Ex\vectp{\sup_{t\in\Bz_{K}}|\cou{{\overline{\nu_t}^{be}}}-{\overline{\nu_t}^{be}}|^2}
\end{multline} 
where we consider separately each rhs term starting with the second.
From  $
|\cou{{\overline{\nu_t}^{be}}}-{\overline{\nu_t}^{be}}|=|p^{-1}\sum_{l=1}^p\{\vec{v}_t(\vec{\iE}_l^{\;be},\couE_l)-
\vec{v}_t(\vec{\iE}^{\;be}_l,E_l)\}|\leq2\VnormInf{\nu_t^b}\sum_{l=1}^{p}\Ind{\{\couE_l\ne E_l\}}$ and by exploiting  the coupling property
\ref{dd:as:cou3} and \eqref{p:l:iid:tal:conc:h}, $n^{3/2}\geq K\ga_{(K)}$
and $\VnormInf{\ceE+\So-\DiSo[K]}\leq\DiSoInf$ we obtain
\begin{multline}\label{p:l:dep:tal:conc:e3}
  \ga_{(K)}\Ex\vectp{\sup_{t\in\Bz_{K}}|\cou{{\overline{\nu_t}^{be}}}-{\overline{\nu_t}^{be}}|^2}\leq 
4\ga_{(K)}\sup_{t\in\Bz_{K}}\VnormInf{\nu_t^b}^2p
\beta_{q+1}\\\leq 4\{\vE n^{1/3}+
\VnormInf{\ceE+\So-\DiSo[K]}\}^2\maxnormsup^2 K \ga_{(K)} p \beta_{q+1}\leq
2 \vE^2 \maxnormsup^2 \{1 + \DiSoInf/\vE\}^2  n^{17/6} q^{-1} \beta_{q+1}
\end{multline}
Considering  the first
rhs term in \eqref{p:l:dep:tal:conc:e2}   we intend to apply
Talagrand's inequality \eqref{l:talagrand:eq1}
given in Lemma \ref{l:talagrand}. The computation of the
quantities $h$, $H$ and $v$ verifying the three required
inequalities is very similar to the calculations given in
\eqref{p:l:iid:tal:conc:h}--\eqref{p:l:iid:tal:conc:v}. Keeping in
mind that $\VnormInf{\cou{{\nu_t^{be}}}}=\VnormInf{{{\nu_t^{be}}}}\leq\VnormInf{\nu_t^{b}}$ from
\eqref{p:l:iid:tal:conc:h} follows
\begin{equation}\label{p:l:dep:tal:conc:h}
\sup_{t\in\Bz_\Di}\VnormInf{\cou{{\nu_t^{be}}}}\leq\{\vE n^{1/3}+ \VnormInf{\ceE+\So-\DiSo}\}\maxnormsup m^{1/2}=:h.
\end{equation}
Making use of the coupling properties
\ref{dd:as:cou1}-\ref{dd:as:cou3} we observe  that  $\{(\vec{\iE}_l^{\;be},\couE_l)\}_{l=1}^p$
are \iid,  $\vec{\nu}_t(\vec{\iE}_l^{\;be},E_l)$ and
  $\vec{\nu}_t(\vec{\iE}_l^{\;be},\couE_l)$ are identically
  distributed, $\{\vec{\iE}_l^{\;be}\}_{l=1}^p$ and  $\{E_l\}_{l=1}^p$
are independent, and $\vec{\iE}_l^{\;be}$ has \iid
components. Consequently, we have 
\begin{multline}\label{p:l:dep:tal:conc:e4}
\Ex\sup_{t\in\Bz_\Di}|\cou{{\overline{\nu_t}^{be}}}|^2
=p^{-1}\sum_{j=1}^\Di\Var\big[q^{-1}\sum_{i=1}^q\{\iE^{\;b}_i+\ceE(\iZ_i,\iV_i)+\So(\iZ_i)-\DiSo(\iZ_i)\}\basV_j(\iV_i))\big]
\\
 \leq p^{-1}\{q^{-1}\vE^2m+q^{-2}\sum_{j=1}^\Di\Var\big[\sum_{i=1}^q(\ceE(\iZ_i,\iV_i)+\So(\iZ_i)-\DiSo(\iZ_i))\basV_j(\iV_i)\big]
\end{multline}
Considering the second right hand side term, we apply Lemma
\ref{l:mixing:AJ}, and hence given
$M_k:=\gauss{(\vE/\DiSoInf)^2\pdZWInf k}$ we have for all $m\geq k$
\begin{multline*}
  \sum_{j=1}^\Di\Var\big[\sum_{i=1}^q(\ceE(\iZ_i,\iV_i)+\So(\iZ_i)-\DiSo(\iZ_i))\basV_j(\iV_i)\big]\\\leq
  qm\{\maxnormsup^2\Vnorm[\iZ,\iV]{\ceE+\So-\DiSo}^2 +
  2\VnormInf{\ceE+\So-\DiSo}^2 [ \pdZWInf M/\sqrt{m} + 2\maxnormsup^2
  \sum_{j=M+1}^{q-1}\beta_j]\}\\
\leq  qm\{\maxnormsup^2\Vnorm[\iZ,\iV]{\ceE+\So-\DiSo}^2 + \vE^2[2 + 4
(\DiSoInf/\vE)^2 \sum_{j=M_k+1}^{q-1}\beta_j]\}
\end{multline*}
Given $\penDep[k]\geq 2[3 + 4(\DiSoInf/\vE)^2 \sum_{j>M_k}\beta_j]$
by combination of the last bound and \eqref{p:l:dep:tal:conc:e4} we obtain  for all $m\geq k$
\begin{multline}\label{p:l:dep:tal:conc:H}
\Ex\sup_{t\in\Bz_\Di}|\cou{{\overline{\nu_t}^{be}}}|^2
\leq n^{-1}m
\maxnormsup^2\{\Vnorm[\iZ,\iV]{\ceE+\So-\DiSo}^2+\vE^2\}\penDep[k]
\leq \maxnormsup^2\{\vY^2+\VnormZ{\DiSo}^2\}\penDep[k]\Di n^{-1}\\
\leq \maxnormsup^2\vB^2  \penDep[k]  \Di\Lambda_\Di^\ga n^{-1}=:H^2.
\end{multline}
Consider finally $v$. Employing successively \ref{dd:as:cou3},
\ref{dd:as:cou1} and Lemma \ref{l:mixing:V} with $\gB\geq 2\sum_{k=0}^\infty(k+1)\beta_k$ we have
\begin{multline}\label{p:l:dep:tal:conc:v}
	\sup_{t\in\cB_m}\tfrac{1}{p}\sum_{l=1}^p
        \Var(\vec{\nu_t}(\vec{\iE}^{\;be}_l,\couE_l))=\sup_{t\in\cB_m}\Var\big[q^{-1}\sum_{i=1}^q\nu_t(\iE_i^{\;b},\iZ_i,\iV_i)\big]\\
 \leq
\frac{\vE^2}{q}+\sup_{t\in\cB_m}\Var\big[q^{-1}\sum_{i=1}^q\{\ceE(\iZ,\iV)+\So(\iZ_i)-\DiSo(\iZ_i)\}\sum_{j=1}^m\fou{t}_j\basV_j(\iV_i)\big]\\
\leq \frac{\vE^2}{q}+
\frac{4}{q}\sup_{t\in\cB_m}\{\VnormInf{\ceE+\So-\DiSo}^2\{\Ex(\sum_{j=1}^m\fou{t}_j\basV_j(\iV_i))^2\}^{1/2}\VnormInf{\sum_{j=1}^m\fou{t}_j\basV_j}\gB^{1/2}\}\\
\leq \frac{\vE^2}{q}+
\frac{4}{q}\VnormInf{\ceE+\So-\DiSo}^2 m^{1/2}\maxnormsup\gB^{1/2}
\\
\leq q^{-1}m^{1/2} \{\vE^2+ 4\VnormInf{\ceE+\So-\DiSo}^2\maxnormsup\gB^{1/2}\}=:v
\end{multline}
Evaluating \eqref{l:talagrand:eq1} of  Lemma \ref{l:talagrand} with
$h$, $H$, $v$ given by \eqref{p:l:dep:tal:conc:h},
\eqref{p:l:dep:tal:conc:H} and \eqref{p:l:dep:tal:conc:v},
respectively, and exploiting $\penDep[k]\maxnormsup^2\vB^2\geq\vE^2$ and
$\VnormInf{\ceE+\So-\DiSo}\leq\DiSoInf$  there exists a numerical constant $C>0$ such that
\begin{multline*}
\Ex\vectp{\max_{k\leq \Di\leq K}\{\ga_\Di[\sup_{t\in\Bz_\Di}|\cou{{\overline{\nu_t}^{be}}}|^2-
  6\maxnormsup^2\penDep[k]\Di \vB^2\Lambda_\Di^\ga n^{-1}]\}}
 \\
\leq Cn^{-1}\sum_{\Di=k}^{K}\ga_\Di\big\{ \vE^2 (1+(\DiSoInf/\vE)^2\maxnormsup\gB^{1/2})\Di^{1/2}\exp\Big(-\frac{\Di^{1/2}\Lambda_\Di^a}{48(1+(\DiSoInf/\vE)^2\maxnormsup\gB^{1/2})}\Big)\\
\hfill+\maxnormsup^2\vE^2( 1 + \DiSoInf/\vE)^2 \Di
n^{-2+2/3}\exp\big(-n^{1/6}/[200(1+\DiSoInf/\vE)]\big)\big\}
\end{multline*}
Since $\ga_{(K)}(K)^2\leq n^{3/2}$ from the definition of $\Psi_\ga$ and
$\Phi_n$  it follows
\begin{multline*}
\Ex\vectp{\max_{k\leq \Di\leq K}\{\ga_\Di[\sup_{t\in\Bz_\Di}|\cou{{\overline{\nu_t}^{be}}}|^2-
  6\maxnormsup^2 \penDep[k] \Di \sigma^2_{\Di}\Lambda_\Di^\ga n^{-1}]\}} \\
	\leq Cn^{-1} \{\vE^2 \Psi_\ga(1+(\DiSoInf/\vE)^2\maxnormsup\gB^{1/2})+  \maxnormsup^2\vE^2\Phi_n(1+\DiSoInf/\vE)\}.
\end{multline*}
Replacing in \eqref{p:l:dep:tal:conc:e2}  the rhs terms by the last bound and \eqref{p:l:dep:tal:conc:e3}, respectively, we obtain
\begin{multline*}
\Ex\vectp{\max_{k\leq \Di\leq
    K}\{\ga_{\Di}[\sup_{t\in\Bz_\Di}|{\overline{\nu_t}^b}|^2-
  12\maxnormsup^2\penDep[k]\vB\Di\Lambda^\ga_{\Di}n^{-1}]\}}\\
\hfill\leq Cn^{-1} \{\vE^2
\Psi_\ga(1+(\DiSoInf/\vE)^2\maxnormsup\gB^{1/2})+
\maxnormsup^2\vE^2\Phi_n(1+\DiSoInf/\vE)\}\\ +
4 \vE^2 \maxnormsup^2 (1 + \DiSoInf/\vE)^2  n^{17/6} q^{-1} \beta_{q+1}.
\end{multline*}
which together with  \eqref{p:l:dep:tal:conc:e1} implies the assertion
\eqref{l:dep:tal:conc}. Consider now
\eqref{l:dep:tal:prob}. Following  the proof of \eqref{l:iid:tal:prob}
we make use of the bound \eqref{p:l:iid:tal:prob:e1} where  as in
\eqref{p:l:iid:tal:prob:e2} the  second rhs
term  can still be bounded by applying successively Markov's
inequality, $\ga_{(K)}K\leq n$ and
\eqref{p:l:iid:tal:conc:e2}. Thereby,  we obtain for all $1\leq m\leq K$
\begin{equation}\label{p:l:dep:tal:prob:e1}
\ProbaMeasure\big(\Vnorm{\fou{V}_{\uDi}}\geq
4c\ga_m^{-1/2}\big)\leq\ProbaMeasure\big(\sup_{t\in\cB_{\Di}}|\overline{\nu_t}^b|
\geq 2c\ga_m^{-1/2}\big)+
(2c)^{-2}\maxnormsup^2 n^{-3} \Ex(\iE/\vE)^{12}.
\end{equation}
Considering the first rhs term we make use of the notations $ \overline{\nu_t}^b
=:\tfrac{1}{2}\{\overline{\nu_t}^{be}+\overline{\nu_t}^{bo}\}$ and  $
\cou{{\overline{\nu_t}^b}}
=\tfrac{1}{2}\{\cou{{\overline{\nu_t}^{be}}}+\cou{{\overline{\nu_t}^{bo}}}\}$
introduced in the proof of \eqref{l:dep:tal:conc} above, where $\overline{\nu_t}^{be}$ and $\overline{\nu_t}^{bo}$
(respectively, $\cou{{\overline{\nu_t}^{be}}}$ and
$\cou{{\overline{\nu_t}^{bo}}}$) are identically distributed. Thereby,
we have
\begin{multline}\label{p:l:dep:tal:prob:e2}
\ProbaMeasure\big(\sup_{t\in\cB_{\Di}}|\overline{\nu_t}^b|
\geq 2c\ga_m^{-1/2}\big)\leq2\ProbaMeasure\big(\sup_{t\in\cB_{\Di}}|\overline{\nu_t}^{be}|\geq 2c\ga_m^{-1/2}\big)
\\\hfill\leq2\big[\ProbaMeasure(\sup_{t\in\cB_{\Di}}|\cou{{\overline{\nu_t}^{be}}}|\geq 2c\ga_m^{-1/2})+\ProbaMeasure(\bigcup_{i=1}^{p}\couE_l\ne E_l)\big]\\\leq2\big[\ProbaMeasure(\sup_{t\in\cB_{\Di}}|\cou{{\overline{\nu_t}^{be}}}|\geq 2c\ga_m^{-1/2})+p\beta_{q+1}\big]
\end{multline}
where the last inequality follows from the coupling property
\ref{dd:as:cou3}. The first rhs term in the last display we bound
employing Talagrand's inequality \eqref{l:talagrand:eq2} given in
Lemma \ref{l:talagrand} with
$h$, $H$, $v$ as in
\eqref{p:l:dep:tal:conc:h}--\eqref{p:l:dep:tal:conc:v},
respectively. Thereby, using $\penDep[1]=2[3 + 4(\DiSoInf/\vE)^2
\gB]$  we have for all $K\geq m\geq 1$
and for all $\lambda>0$ 
\begin{multline*}
\ProbaMeasure\big(\sup_{t\in\cB_{\Di}}|\cou{{\overline{\nu_t}^{be}}}|
\geq 2 \maxnormsup\{\vY^2+\VnormZ{\DiSo}^2\}^{1/2} (\penDep[1])^{1/2}\Di^{1/2} n^{-1/2}+\lambda\big)\\
\leq3\exp\bigg[-\frac{p}{100}\bigg(\frac{\lambda^2}{q^{-1}m^{1/2} \{\vE^2+ 4(\DiSoInf)^2\maxnormsup\gB^{1/2}\}}\wedge\frac{\lambda}{(\vE+\DiSoInf)\maxnormsup m^{1/2} n^{1/3}}\bigg)\bigg].
\end{multline*}
Since $n \geq 4 c^{-2}\maxnormsup^2 \penDep[1]  K \ga_{(K)}\geq 4 c^{-2}
\maxnormsup^2 \penDep[1] \Di \ga_{\Di}$  letting
$\lambda:=c\{\vY^2+\VnormZ{\DiSo}^2\}^{1/2}\ga_{\Di}^{-1/2}$ and 
using $\{\vY^2+\VnormZ{\DiSo}^2\}^{1/2}\geq \vE$,  $\ga_{m}\leq
\ga_{(K)}$  and $n^{1/2}c\geq2(1+\DiSoInf/\vE)\maxnormsup
  K^{1/2}\ga_{(K)}^{1/2}$ we obtain
\begin{multline*}
\ProbaMeasure\big(\sup_{t\in\cB_{\Di}}|\cou{{\overline{\nu_t}^{be}}}| \geq 2c
\{\vY^2+\VnormZ{\DiSo}^2\}^{1/2}\ga_m^{-1/2}\big)\\
\leq3\exp\bigg[-\frac{p}{100}\bigg(\frac{c^2\vE^2}{q^{-1}K^{1/2}
  \{\vE^2+
  4(\DiSoInf)^2\maxnormsup\gB^{1/2}\}\ga_{(K)}}\wedge\frac{c\vE}{(\vE+\DiSoInf)\maxnormsup
  K^{1/2}\ga_{(K)}^{1/2} n^{1/3}}\bigg)\bigg]\\
\leq3\exp\bigg[\frac{-nc^2}{200K^{1/2}\{1+
  4(\DiSoInf/\vE)^2\maxnormsup\gB^{1/2}\}\ga_{(K)}}\vee\frac{-n^{1/6}}{100q}\bigg]
\end{multline*}
We obtain the assertion \eqref{l:dep:tal:prob} by replacing
successively in \eqref{p:l:dep:tal:prob:e2} the first rhs term by the
last bound, the resulting bound is then used in
\eqref{p:l:dep:tal:prob:e1} to derive the assertion,  which completes the proof.
\end{proof}
\begin{lem}\label{l:dep:tal2}  Let $\ga$ be  a
  non negative sequence  with $\ga_{(m)}:=\max_{1\leq k\leq
 m}\ga_k$ and
$\gB\geq 2\sum_{k=0}^\infty
(k+1)\beta_k$. If $n \geq 8 c^{-2}
\maxnormsup^4 (1+4\gB) K^2 \ga_{(K)}$ for $c>0$ then for all $1\leq\Di\leq K$ holds
  \begin{equation}\label{l:dep:tal2:prob}
    \ProbaMeasure\bigg(\ga_{\Di}\VnormS{\fou{\Xi}_{\uDi}}^2\geq 4c^2\bigg)
\leq 6\exp\bigg[\frac{-nc^2}{800\VnormInf{p_{\iZ,\iV}}K\ga_{(K)}\maxnormsup^2\gB^{1/2}}\vee\frac{-n^{1/2}}{50q}\bigg]+nq^{-1}\beta_{q+1}
  \end{equation}
  where $p_{Z,W}$
  denotes the joint density of $Z$
  and $W$.
\end{lem}
\begin{proof}[\dr Proof of Lemma \ref{l:dep:tal2}] The proof follows a
  long the lines of the proof of Lemma \ref{l:iid:tal2} applying Talagrand's inequality \eqref{l:talagrand:eq2}
  in Lemma \ref{l:talagrand} using
  $\sup_{t\in\cB_{\Di^2}}|\overline{\nu_t}(x)|^2=\sum_{j,l=1}^{\Di}\fou{\Xi}_{j,l}^2\geq
  \VnormS{\fou{\Xi}_{\uDi}}^2$ where  $\nu_t(\iZ,\iV)=\sum_{j,l=1}^{\Di}\fou{t}_{j,l}\basZ_j(\iZ)\basV_l(\iV)$.
 Consider
$(\iZ_i,\iV_i)_{i\geq 1}=(E_l,O_l)_{l\geq 1}$ and
$(\couZ_i,\couV_i)_{i\geq 1}=(\couE_l,\couO_l)_{l\geq 1}$   which
satisfy  the coupling properties \ref{dd:as:cou1}, \ref{dd:as:cou2}
and \ref{dd:as:cou3}. Let
$\vec{v}_t(E_k)=\sum_{j,l=1}^\Di\fou{t}_{j,l}\vec{\psi}_{j,l}(E_k)$ with $\vec{\psi}_{j,l}(E_k)=q^{-1}\sum_{i\in\cI_k^e}\basZ_j(\iZ_i)\basV_l(\iV_i)$, then for $n=2pq$ it follows 
\begin{equation*}
 \overline{\nu_t} =
 \tfrac{1}{2}\big\{\tfrac{1}{p}\sum_{l=1}^p\{\vec{v}_t(E_l)-\Ex\vec{v}_t(E_l)
 \}+ \tfrac{1}{p}\sum_{l=1}^p\{\vec{v}_t(O_l)-\Ex\vec{v}_t(O_l)\}\big\}=:\tfrac{1}{2}\{\overline{\nu_t^{e}}+\overline{\nu_t^{o}}\}. 
\end{equation*}
Considering  the random variables $(\couZ_i,\couV_i)_{i\geq1}$ rather
than $(\iZ_i,\iV_i)_{i\geq1}$ we introduce in addition $\cou{
  \overline{\nu_t}}=\tfrac{1}{2}\{\cou{\overline{\nu_t^{e}}}+\cou{\overline{\nu_t^{o}}}\}$. Keeping
in mind that $\overline{\nu_t^{e}}$ and $\overline{\nu_t^{o}}$
(respectively, $\cou{\overline{\nu_t^{e}}}$ and
$\cou{\overline{\nu_t^{o}}}$) are identically distributed, we have
\begin{multline}\label{p:l:dep:tal2:e1}
\ProbaMeasure(\sup_{t\in\cB_{\Di^2}}|\overline{\nu_t}|\geq x)
\leq 2\ProbaMeasure(\sup_{t\in\cB_{\Di^2}}|\overline{\nu_t^e}|\geq x)
\leq2\big[\ProbaMeasure(\sup_{t\in\cB_{\Di^2}}|\cou{\overline{\nu_t^{e}}}|\geq x)+\ProbaMeasure(\bigcup_{i=1}^{p}\couE_l\ne E_l)\big]\\\leq2\big[\ProbaMeasure(\sup_{t\in\cB_{\Di^2}}|\cou{\overline{\nu_t^{e}}}|\geq x)+p\beta_{q+1}\big]
\end{multline}
where the last inequality follows from the coupling property
\ref{dd:as:cou3}. The first rhs term in the last display we bound by applying Talagrand's inequality \eqref{l:talagrand:eq2}
  in Lemma \ref{l:talagrand}.  Therefore, we compute next the quantities $h$,  $H$
  and $v$
  verifying the three required inequalities. Consider $h$. Exploiting
  Assumption \ref{a:mth:bs}  we have
  \begin{multline}
    \label{p:l:dep:tal2:h}
    \sup_{t\in\cB_{\Di^2}}\VnormInf{\vec{v}_t}^2\leq\sum_{j,l=1}^{\Di}\VnormInf{\vec{v}_{j,l}}^2\leq
    \sum_{j,l=1}^{\Di}\VnormInf{\basZ^2_j}\VnormInf{\basV^2_l}\leq\Di^2\maxnormsup^4=:h^2.\hfill  
\end{multline}
Consider
$H$. Let
$\cB_{q}:=\sum_{k=1}^{q-1}\beta_k$. Exploiting
successively that  $\{\vec{\psi}_{j,l}(\couE_k)\}_{k=1}^p$ form an \iid
sample,  $\vec{\psi}_{j,l}(\couE_1)$ and $\vec{\psi}_{j,l}(E_1)$
are identically distributed and Lemma \ref{l:mixing:V0} we obtain
\begin{multline}\label{p:l:dep:tal2:H}
\Ex\sup_{t\in\cB_{\Di^2}}|\cou{\overline{\nu_t^{e}}}|^2
= p^{-1}\sum_{j,l=1}^{\Di}\Var(\vec{\psi}_{j,l}(E_1))
= p^{-1}q^{-2}\sum_{j,l=1}^{\Di}\Var\big(\sum_{i\in\cI_1^e}\basZ_j(\iZ_i)\basV_l(\iV_i)\big)\\\
\leq\frac{2\Di^2\maxnormsup^4}{n}(1+4\cB_{q}):=H^2.
  \end{multline}
Consider $v$. From Lemma \ref{l:mixing:V} with
$h(\iZ,\iV)=\sum_{j,l=1}^m\fou{t}_{jl}\basZ_j(\iZ)\basV_l(\iV)$ follows
\begin{multline}
    \sup_{t\in\cB_{\Di^2}}\frac{1}{p}\sum_{l=1}^p
    \Var\big(\vec{\nu_t}(\cou{E_l})\big)\leq 4q^{-1}\sup_{t\in\cB_{\Di^2}}\Ex\big[(\sum_{j,l=1}^{\Di}\fou{t}_{j,l}\basZ_j(\iZ_1)\basZ_l(\iV_1))^2b(\iZ_1,\iV_1))\big]\\
\leq 4q^{-1}\sup_{t\in\cB_{\Di^2}}\{\Ex(\sum_{j,l=1}^{\Di}\fou{t}_{j,l}\basZ_j(\iZ_1)\basV_l(\iV_1))^2\}^{1/2}\VnormInf{\sum_{j,l=1}^m\fou{t}_{jl}\basZ_j\basV_l}\{2\sum_{k=0}^\infty
(k+1)\beta_k\}^{1/2}\\
 \label{p:l:dep:tal2:v} 
\hspace*{40ex}\leq 4\Di q^{-1}\maxnormsup^2\VnormInf{p_{\iZ,\iV}}\gB^{1/2}=:v.
 \end{multline}
  Evaluating \eqref{l:talagrand:eq2} of Lemma \ref{l:talagrand} with
  $h$,
  $H$,
  $v$
  given by \eqref{p:l:dep:tal2:H}--\eqref{p:l:dep:tal2:v},
  respectively,  for any $\lambda>0$ we have
  \begin{multline*}
    \ProbaMeasure\big(\sup_{t\in\cB_{\Di^2}}|\cou{\overline{\nu_t^{e}}}|\geq
    2\sqrt{2}m\maxnormsup^2(1+4\gB_q)^{1/2}n^{-1/2}+\lambda\big)\\
 \leq3\exp\bigg[\frac{-n\lambda^2}{800\VnormInf{p_{\iZ,\iV}}m\maxnormsup^2\gB^{1/2}}\vee\frac{-n\lambda}{200q\Di\maxnormsup^2}\bigg]
  \end{multline*}
  Since $ n\geq 8c^{-2} K^2 \ga_{(K)}\maxnormsup^4(1+4\gB) \geq
  8c^{-2} m^2 \ga_m\maxnormsup^4(1+4\gB_q)$, $1\leq m\leq K$, letting $\lambda:=c\ga_m^{-1/2}$ and 
using  $\ga_{m}\leq \ga_{(K)}$  and $n^{1/2}c\geq4\maxnormsup^2K\ga_{(K)}^{1/2}$ we obtain
  \begin{multline*}
    \ProbaMeasure\bigg(\sup_{t\in\cB_{\Di^2}}|\overline{\nu_t}|\geq
    2c\ga_{\Di}^{-1/2}\bigg)\\\hfill\leq
3\exp\bigg[\frac{-nc^2}{800\VnormInf{p_{\iZ,\iV}}K\ga_{(K)}\maxnormsup^2\gB^{1/2}}\vee\frac{-nc}{200qK
  \ga_{(K)}^{1/2}\maxnormsup^2}\bigg]\hfill\\
\leq
3\exp\bigg[\frac{-nc^2}{800\VnormInf{p_{\iZ,\iV}}K \ga_{(K)}\maxnormsup^2\gB^{1/2}}\vee\frac{-n^{1/2}}{50q}\bigg]
  \end{multline*}
A combination
of the last bound, \eqref{p:l:dep:tal2:e1}  and  $\sup_{t\in\cB_{\Di^2}}|\overline{\nu_t}(x)|\geq
\VnormS{\fou{\Xi}_{\uDi}}$ implies  the assertion, which completes
the proof.
\end{proof}


%% file: _proof_4Minimax.tex
%
%
\subsection{Proof of Theorem \ref{t:iid:mm}}\label{a:mm:iid}%
Let us first recall notations and gather preliminary results used in
the sequel.  Keeping in mind the notations given in \eqref{de:delta} and
  \eqref{de:Mn} we assume throughout
  this section $\Op\in\OpcwdD$ and use  in addition to \eqref{a:not:de:no}  for all $m\geq1$ and $n\geq1$
\begin{multline}\label{a:mm:de:no}
\Delta_m^\Opw= \Delta_m(\Opw) ,\,
\Lambda_m^\Opw=\Lambda_m(\Opw),\,\delta_m^\Opw=
m\Delta_m^\Opw\Lambda_m^\Opw,\\
\Mutw= \Mn(4\OpD^2\Opw),\quad \Motw= \Mn(\Opw/(4\Opd^2)).\hfill
\end{multline}
Recall that under Assumption \ref{a:mm:seq}  for all $\Di\geq1$ it holds
  $\OpD^{-2}\leq\Opw_{\Di}^{-1}\VnormS{\DifOp^{-1}}^2\leq
\OpD^2$,
 $\OpD^{-2} \leq \Delta_m^\Op/\Delta^\Opw_m\leq D^2 $,  $(1+2\log
 \OpD)^{-1}\leq \Lambda_m^\Op/\Lambda^\Opw_m\leq (1+2\log \OpD)$, and
 $\OpD^{-2}(1+2\log \OpD)^{-1}\leq \delta_m^\Op/\delta^\Opw_m\leq
 \OpD^2(1+2\log \OpD)$ as well as $\Mutw\leq\Mut\leq\Mot\leq\Motw$,
 for all $n\geq1$. Furthermore,  the elements of $\Socwr$ are bounded uniformly, that
is, $\VnormInf{\phi}^2\leq
\VnormInf{\sum_{j\geq1}\Sow_j\basZ_j^2}\VnormW[\Sow]{\phi}^2\leq
\SowBasZsup^2\Sor^2$ for all $\phi\in\Socwr$. The last estimate  is used in Lemma \ref{app:pre:lb} in the
Appendix \ref{app:pre}  to show that for all $\So\in\Socwr$ and
$\Op\in\OpcwdD$ the approximation $\DiSo$
satisfies  $\Sow_{\Di}^{-1}\biasnivSo[\Di]\leq 4
\OpD^{4}\Sor^2$, $\VnormInf{\So-\DiSo}\leq 2\OpD^2 \SowBasZsup\Sor$ and
$\VnormZ{\DiSo}^2\leq 4\OpD^4\Sor^2$. Thereby, setting
$\DiSowTwo:=\Vnorm[\iZ,\iV]{\ceE}^2\vee4\OpD^4\Sor^2$ and $\DiSowTwo:=\VnormInf{\ceE}+(1+2\OpD^2 )\SowBasZsup\Sor$
the Assumption \ref{a:iid:ora} \ref{a:iid:ora:b} holds with
$\DiSoTwo:=\DiSowTwo$ and $\DiSoInf:=\DiSowInf$ uniformly for all
$\So\in\Socwr$ and $\Op\in\OpcwdD$.
The proof follows along the lines of the proof of Theorem
\ref{t:iid:ora} given in Appendix \ref{a:ora:iid}. We shall prove below the Propositions \ref{p:iid:mm:Eset} and
\ref{p:iid:mm:Rest} which are used in the proof of Theorem \ref{t:iid:mm}.
In the proof the propositions we refer to the three technical Lemma
\ref{l:iid:tal}, \ref{l:iid:tal2} and \ref{l:iid:mm:sets}
 which are shown in Appendix  \ref{a:ora:iid} and the end of
this section. Moreover, we make use of  functions
  $\tbPsi,\tbPhic,\tbPhid,\tbPhie:\Rz_+\to\Rz$
  defined by
  \begin{multline}\label{de:iid:mm}
    \tbPsi(x)=\OpD^2\sum\nolimits_{m\geq 1} x \Opw_m \exp(- m
    \Lambda_\Di^\Opw/(6(1+2\log\OpD)x) ),\\
    \tbPhic(x)= n^3\exp(-n(\Delta^\Opw_{\Motw})^{-1}/(25600\OpD^2x^2)),\\
    \tbPhid(x)=n^3\exp(- n(\Delta_{\Motw+1}^\Opw)^{-1}/(6400\OpD^2x))\hfill\\
    \tbPhie(x)= x n \exp(-\Motw\log(n)/(6x) ).\hfill
  \end{multline}
 Note that each function in \eqref{de:iid:mm}  is non decreasing in $x$
  and for all $x>0$,  $\tbPsi(x)<\infty$,
$\bPsi(x)\leq\tbPsi(x)$, $\bPhic(x)\leq\tbPhic(x)$ and
$\bPhid(x)\leq\tbPhid(x)$ with $\bPsi$, $\bPhic$ and $\bPhid$ as in \eqref{de:iid:ora}.
 Moreover, if
  $\log(n)(\Motw+1)^2\Delta_{\Motw+1}^\Op=o(n)$ as $n\to\infty$  then 
there exists an integer $n_o$  such that 
  \begin{equation}\label{de:iid:mm:no}
1\geq \sup_{n\geq n_o}
\big\{1024\maxnormsup^4\OpD^2(1+\DiSowInf/\vE\big)^2(\Motw+1)^2\Delta_{\Motw+1}^\Opw
n^{-1}\big\},
\end{equation}
 and we have also for all $x>0$,  $\tbPhic(x)=o(1)$, $\tbPhid(x)=o(1)$
  and $\tbPhie(x)=o(1)$  as $n\to\infty$. Consequently, considering
  $\tbPhia$ and $\tbPhib$ as in \eqref{de:iid:ora} under
  Assumption \ref{a:mth:rv} and \ref{a:mth:bs}  there exists 
a finite constant $\SowRaC$  such that for all $n\geq1$,
\begin{multline}\label{de:iid:mm:Sigma}
\SowRaC\geq
\big\{n_o^2\bigvee n^3\exp(-n^{1/6}/50)\bigvee \tbPsi\big(1+\DiSowInf/\vE\big)\bigvee
\tbPhia\big(1+\DiSowInf/\vE\big)\bigvee
\tbPhib\big(1+\DiSowInf/\vE\big)\\\hfill\bigvee 
\tbPhie\big(1+\DiSowInf/\vE\big)\bigvee 
\tbPhic(1+\DiSowInf/\vE)\bigvee
\tbPhid(\VnormInf{p_{\iZ,\iV}})\\
\bigvee\Ex(\iE/\vE)^{8}\bigvee(\DiSowInf/\vE)^{8}\bigvee (\maxnormsup/\vE)^2\Ex(\iE/\vE)^{12} \big\}.
\end{multline}%
\begin{proof}[\noindent\textcolor{darkred}{\sc Proof of Theorem \ref{t:iid:mm}}] 
 We start the
  proof considering the elementary identity \eqref{ri:dec} given in
  the proof of Theorem \ref{t:iid:ora} where we bound the two rhs terms separately.
  The
  second rhs term  we bound with help of Proposition
  \ref{p:iid:mm:Rest}. Thereby, there exists a numerical constant $C$
  such that for all $\So\in\Socwr$ hold
  \begin{equation}\label{ri:dec:mm:iid}
    \Ex\VnormZ{\hDiSo[\hDi]-\So}^2\leq \Ex\left(\Ind{\Eset}\VnormZ{\hDiSo[\hDi]-\So}^2\right)+C\;n^{-1}\;
\maxnormsup^2(1+\vE^2+\DiSowTwo)\SowRaC.
  \end{equation}
Consider the first rhs term. On the event $\Eset$
  the upper bound given in \eqref{t:iid:ora:key:arg} implies
  \begin{equation*}
    \VnormZ{\hDiSo[\hDi]-\So}^2\Ind{\Eset}\leq 582 \min_{1\leq \Di\leq \Mut}\{[\biasnivSo\vee\pen]\} +42 \max_{1\leq k\leq \Mot}\vectp{ \VnormZ{\hDiSo[k]-\DiSo[k]}^2 - \pen[k]/6}.
  \end{equation*}
Keeping in mind that $\pen[k]=144\maxnormsup^2\vB[k]^2\delta_k^\Op n^{-1}$  with $\delta_k^\Op=k\Lambda_k^\Op\Delta_{k}^\Op$ and $\vB[k]^2\leq2(\vE^2+3\DiSowTwo)$  we derive in
  Proposition \ref{p:iid:mm:Eset} below an upper bound for the
  expectation of the second rhs term, the remainder term,
  in the last display. Thereby, from   $\Mut\geq\Mutw$,
  $\biasnivSo\leq \Sow_m 4\OpD^4\Sor^2$ and $\delta_k^\Op\leq
  \OpD^2(1+2\log\OpD)\delta_k^\Opw$ there exists a numerical constant $C$
  such that for all $\So\in\Socwr$ 
 \begin{equation*}
    \Ex\big(\Ind{\Eset}\VnormZ{\hDiSo[\hDi]-\So}^2\big)\leq 
    C\;\maxnormsup^2\OpD^4 (\Sor^2+\vE^2+\DiSowTwo)\{\min_{1\leq \Di\leq
      \Mutw}\{[\Sow_{\Di}\vee n^{-1}\delta_{\Di}^\Opw]\}  +n^{-1}\;\SowRaC\}.
  \end{equation*}
Replacing in  \eqref{ri:dec:mm:iid} the first rhs by the last upper bound
we obtain the assertion of the theorem, which completes the proof.
 \end{proof}%
\begin{prop}\label{p:iid:mm:Eset} Under the assumptions of Theorem \ref{t:iid:mm}
  there exists a numerical constant $C$ such that for all $n\geq1$ 
\begin{equation*}\Ex\set{\max_{1\leq k\leq
      \Mot}\vectp{\VnormZ{\hDiSo-\DiSo}^2 -24\maxnormsup^2\Di n^{-1}\vB^2\Lambda_\Di^\Op\Delta_{\Di}^\Op}}\leq
C n^{-1}\maxnormsup^2(1+\vE^2+\DiSowTwo)\SowRaC.
\end{equation*}
\end{prop}
\begin{proof}[\noindent\textcolor{darkred}{\sc Proof of Proposition
    \ref{p:iid:mm:Eset}}]
  We start the
  proof with an upper bound similar to \eqref{p:p:iid:Eset:e1} using $\Mot\leq\Motw$, that is,
 \begin{multline}\label{p:p:iid:mm:Eset:e1}
\Ex\set{\max_{1\leq \Di\leq\Mot}\vectp{\VnormZ{\hDiSo-\DiSo}^2 -24\maxnormsup^2\Di n^{-1}\vB^2\Lambda_\Di^\Op\Delta_{\Di}^\Op}}\\
\hfill\leq 2\Ex\set{ \max_{1\leq\Di\leq\Mot} \VnormS{\DifOp^{-1}}^2\vectp{\Vnorm{\fou{V}_{\uDi}}^2
    -12\maxnormsup^2\Di n^{-1}\vB^2\Lambda_\Di^\Op}}\\\hfill +
 \Ex\set{n\vectp{\Vnorm{\fou{V}_{\underline{\Motw}}}^2-12\maxnormsup^2\Motw  n^{-1}\vB[\Motw]^2\log(n)}}
\\+12\maxnormsup^2\Motw \vB[\Motw]^2 \log(n)P(\bigcup_{k=1}^{\Motw}\Xiset[k]^c)
 +\max_{1\leq \Di\leq\Mot}\VnormZ{\DiSo}^2P(\bigcup_{k=1}^{\Mot}\hOpset[k]^c)
\end{multline}
where we bound separately each of the four rhs
terms. In order to bound  the first and second rhs term we employ \eqref{l:iid:tal:conc} in
Lemma \ref{l:iid:tal} with $K=\Mot$ and $K=\Motw$, and sequence $\ga=(\ga_m)_{m\geq1}$ given by $\ga_m=\VnormS{\DifOp^{-1}}^2$ and
$\ga_m=n\Ind{\{m=\Motw\}}$, respectively. Keeping in mind  the
definition of $\Mot$, $\Motw$ and $\Mot\leq\Motw\leq\gauss{n^{1/4}}$, 
and hence in both cases
$\ga_{(K)}K^2\leq n^{3/2}$,  there exists a numerical constant $C>0$
such
that 
 \begin{multline*}
\Ex\set{\max_{1\leq \Di\leq\Mot}\vectp{\VnormZ{\hDiSo-\DiSo}^2 -24\maxnormsup^2\Di n^{-1}\vB^2\Lambda_\Di^\Op\Delta_{\Di}^\Op}}\\
\leq C n^{-1}\maxnormsup^2\big\{\vE^2\bPsi(1+\DiSoInf/\vE)+\vE^2\bPhib(1+\DiSoInf/\vE)+
\vE^2\tbPhie(1+\DiSoInf/\vE)+\Ex(\iE/\vE)^{12}\}\\
+6\maxnormsup^2\Mot \vB[\Mot]^2\log(n) P(\bigcup_{k=1}^{\Mot}\Xiset[k]^c)
 +\max_{1\leq \Di\leq\Mot}\VnormZ{\DiSo}^2P(\bigcup_{k=1}^{\Mot}\hOpset[k]^c)
\end{multline*}
with $\bPsi$, $\bPhib$  as in \eqref{de:iid:ora}, 
and $\tbPhie$ as in \eqref{de:iid:mm}. 
Taking into account that $\bPsi(x)\leq\tbPsi(x)$ and that Assumption \ref{a:iid:ora} \ref{a:iid:ora:b} holds with
$\DiSoTwo:=\DiSowTwo$ and $\DiSoInf:=\DiSowInf$ uniformly for all
$\So\in\Socwr$ and $\Op\in\OpcwdD$, it follows that 
 \begin{multline*}
\Ex\set{\max_{1\leq \Di\leq\Mot}\vectp{\VnormZ{\hDiSo-\DiSo}^2 -24\maxnormsup^2\Di n^{-1}\vB^2\Lambda_\Di^\Op\Delta_{\Di}^\Op}}\\
\leq C n^{-1}\maxnormsup^2\big\{\vE^2\tbPsi(1+\DiSowInf/\vE)+\vE^2\bPhib(1+\DiSowInf/\vE)+
\vE^2\tbPhie(1+\DiSowInf/\vE)+\Ex(\iE/\vE)^{12}\}\\
+6\maxnormsup^2\Motw \vB[\Motw]^2\log(n) P(\bigcup_{k=1}^{\Motw}\Xiset[k]^c)
 +\max_{1\leq \Di\leq\Mot}\VnormZ{\DiSo}^2P(\bigcup_{k=1}^{\Mot}\hOpset[k]^c)
\end{multline*}
Exploiting that $\vB^2\leq 2(\vE^2+3\DiSowTwo)$, $\Motw\log(n)\leq
n$   and $\max_{1\leq \Di\leq\Mot}\VnormZ{\DiSo}^2\leq\DiSowTwo$,
replacing the probability $P(\bigcup_{k=1}^{\Motw}\hOpset[k]^c)$ and $P(\bigcup_{k=1}^{\Motw}\Xiset[k]^c)$ by its upper bound  
given in \eqref{l:iid:mm:sets:hO} and  \eqref{l:iid:mm:sets:B} in Lemma
\ref{l:iid:mm:sets}, respectively, and  employing the definition
of $\SowRaC$ as in \eqref{de:iid:mm:Sigma} we obtain the result of the proposition, 
 which completes the proof.
\end{proof}
\begin{prop}\label{p:iid:mm:Rest} Under the assumptions of Theorem \ref{t:iid:mm}  there exists a numerical constant $C$ such that for all $n\geq1$
 \begin{equation*}
\Ex\big(\VnormZ{\hDiSo[\hDi]-\So}^2\1_{\Eset^c}\big)\leq 
C\;n^{-1}\;
\maxnormsup^2(1+\vE^2+\DiSowTwo)\SowRaC.
\end{equation*}
\end{prop}
\begin{proof}[\noindent\textcolor{darkred}{\sc Proof of Proposition
    \ref{p:iid:mm:Rest}}]  
Following line by line the proof of Proposition
    \ref{p:iid:Rest} for $\maxDi:=\DiMa$
there exists a numerical constant $C>0$ such that
\begin{multline*}
\Ex\big(\VnormZ{\hDiSo[\hDi]-\So}^2\1_{\Eset^c}\big)\leq 
Cn^{-1}\maxnormsup^2\big\{\vE^2\bPhia\big(1+\DiSowInf/\vE\big)+\vE^2\bPhib\big(1+\DiSowInf/\vE\big)
        +\Ex(\iE/\vE)^{12}\big\}\\+\{36\maxnormsup^2 \maxDi \vB[\maxDi]^2
        \log(n) + 6\DiSowTwo\}P(\Eset^c)
\end{multline*}
with $\bPhia$ and $\bPhib$ as in \eqref{de:iid:ora}. 
Exploiting further the definition of $\SowRaC$ as in
\eqref{de:iid:mm:Sigma} and that $\vB[\maxDi]^2\leq
2\{\vE^2+3\DiSowTwo\}$ and
$M\log(n)\leq n$ the result of the proposition follows now  by
replacing the probability $P(\Eset^c)$ by its upper bound  
given in \eqref{l:iid:mm:sets:E} in Lemma \ref{l:iid:mm:sets}, which completes the proof.
\end{proof}
\begin{lem}\label{l:iid:mm:sets}
Under the assumptions of Theorem \ref{t:iid:mm} there exists a numerical constant $C$ such that for all $n\geq1$
\begin{align}
&\ProbaMeasure\big(\Aset^c)=\ProbaMeasure\big(\{1/2\leq\hsigma^2_Y/\sigma_Y^2\leq
  3/2\}^c\big)\leq C\; \SowRaC\;n^{-2},
 \label{l:iid:mm:sets:A}\\
&\ProbaMeasure\big(\Bset^c\big)\leq\ProbaMeasure\big(\bigcup_{\Di=1}^{\Motw+1}\Xiset^c\big)\leq C\;
 \SowRaC\; n^{-2}, \label{l:iid:mm:sets:B}\\
&\ProbaMeasure\big(\Cset^c\big)\leq C\; \SowRaC\; n^{-2}, \label{l:iid:mm:sets:C}\\
&\ProbaMeasure\big(\Eset^c\big)\leq C\;\SowRaC\; n^{-2}, \label{l:iid:mm:sets:E}\\
&\ProbaMeasure\big(\bigcup_{\Di=1}^{\Mot}\hOpset\big)\leq  C\;
 \SowRaC\;n^{-2}. \label{l:iid:mm:sets:hO}
\end{align}
\end{lem}
\begin{proof}[\noindent\textcolor{darkred}{\sc Proof of Lemma
    \ref{l:iid:mm:sets}}]The proof of  \eqref{l:iid:mm:sets:A} follows
  line by line the proof of \eqref{l:iid:sets:A} in Lemma
  \ref{l:iid:sets} using the definition of $\SowRaC$ as in
  \eqref{de:iid:mm:Sigma} rather $\SoRaC$  than
  \eqref{de:iid:ora:Sigma} and hence we omit the details.
 Consider 
\eqref{l:iid:mm:sets:B}--\eqref{l:iid:mm:sets:E}. Let  $\ga$ be a sequence given by
$\ga_m=\VnormS{\DifOp^{-1}}^2$ where
$\ga_{(m)}=\Delta^T_{m}$ and $n_o$ an integer satisfying
\eqref{de:iid:mm:no} uniformly for all $\Op\in\OpcwdD$ and $\So\in\Socwr$,
that is,  $n\geq
1024\maxnormsup^4\OpD^2(1+\DiSowInf/\vE)^2(\Motw+1)^2\Delta_{\Motw+1}^\Opw\geq
1024\maxnormsup^4(1+\DiSoInf/\vE)^2(\Mot+1)^2\Delta_{\Mot+1}^\Op$ for
 all $n> n_o$  by construction. We distinguish in the following the cases $n\leq n_o$ and
$n > n_o$. Consider   
\eqref{l:iid:mm:sets:B}. Following line by line the proof of
\eqref{l:iid:sets:B} together with $\bPhid(x)\leq\tbPhid(x)$ and
$\Motw+1\leq n$ we have 
$\ProbaMeasure\big(\bigcup_{\Di=1}^{\Motw+1}\Xiset^c\big)
\leq 3\;
n^{-2}\;\tbPhid(\VnormInf{p_{\iZ,\iV}})\vee
\{n^3\exp(-n^{1/2}/50)\}$.
 By combination of the
two cases and employing the definition of
$\SowRaC$ given in \eqref{de:iid:mm:Sigma} we obtain \eqref{l:iid:mm:sets:B}. The proof of
\eqref{l:iid:mm:sets:C} follows line by line  the proof of
\eqref{l:iid:sets:C} in Lemma \ref{l:iid:sets}. 
Exploiting $\bPhic(x)\leq\tbPhic(x)$ we obtain $\ProbaMeasure\big(\Cset^c\big)\leq 3 
\{n^3\exp(-n^{1/6}/50)\}\vee\tbPhic(1+\DiSowInf/\vE)n^{-2}+ 32(\maxnormsup^2/\vE^2) \Ex(\iE/\vE)^{12}n^{-2} $
The assertion \eqref{l:iid:mm:sets:C} follows employing the definition of
$\SowRaC$ given in \eqref{de:iid:mm:Sigma}. Consider
\eqref{l:iid:mm:sets:E}. Due to Lemma \ref{app:pre:l5} it holds
$\ProbaMeasure\big(\Eset^c\big)\leq \ProbaMeasure\big(\Aset^c\big)
+\ProbaMeasure\big(\Bset^c\big)
+\ProbaMeasure\big(\Cset^c\big)$. Therefore, the assertion
\eqref{l:iid:mm:sets:E} follows from
\eqref{l:iid:mm:sets:A}--\eqref{l:iid:mm:sets:C}. The proof of 
\eqref{l:iid:mm:sets:hO} follows in same manner as the proof of 
\eqref{l:iid:sets:hO}, and we omit the details, which completes the proof.
\end{proof}
\subsection{Proof of Theorem \ref{t:dep:mm}}\label{a:mm:dep}
The proof follows along the lines of the proof of Theorem
\ref{t:dep:ora} given in Appendix \ref{a:ora:dep}. We shall prove below the Propositions \ref{p:dep:mm:Eset} and
\ref{p:dep:mm:Rest} which are used in the proof of Theorem \ref{t:dep:mm}.
In the proof the propositions we refer to the three technical Lemma
\ref{l:dep:tal}, \ref{l:dep:tal2} and \ref{l:dep:mm:sets}
 which are shown in Appendix  \ref{a:ora:dep} and the end of
this section. Moreover, we make use of  functions
  $\tbPsi,\tbPhic,\tbPhid,\tbPhie:\Rz_+\to\Rz$
  defined by
  \begin{multline}\label{de:dep:mm}
    \tbPsi(x)=\OpD^2\sum\nolimits_{m\geq 1} x m^{1/2} \Opw_m \exp(- m^{1/2}
    \Lambda_\Di^\Opw/(48(1+2\log\OpD)x) ),\\
    \tbPhic(x)= n^3\exp(-n(\Motw)^{-1/2}(\Delta^\Opw_{\Motw})^{-1}/(204800\OpD^2x)),\\
    \tbPhid(x)=n^3\exp(- n(\Motw+1)^{-1}(\Delta_{\Motw+1}^\Opw)^{-1}/(51200\OpD^2x))\hfill\\
    \tbPhie(x)= x n \exp(-(\Motw)^{1/2}\log(n)/(48x) ).\hfill
  \end{multline}
 Note that each function in \eqref{de:dep:mm}  is non decreasing in $x$
  and for all $x>0$,  $\tbPsi(x)<\infty$,
$\bPsi(x)\leq\tbPsi(x)$, $\bPhic(x)\leq\tbPhic(x)$ and
$\bPhid(x)\leq\tbPhid(x)$ with $\bPsi$, $\bPhic$ and $\bPhid$ as in \eqref{de:dep:ora}.
 Moreover, if
  $\log(n)(\Motw+1)^2\Delta_{\Motw+1}^\Opw=o(n)$ as $n\to\infty$  then 
there exists an integer $n_o$    such that 
  \begin{equation}\label{de:dep:mm:no}
1\geq \sup_{n\geq n_o}
\big\{1024\maxnormsup^4\OpD^2(6+8(\DiSowInf/\vE\big)^2\gB)(\Motw+1)^2\Delta_{\Motw+1}^\Opw
n^{-1}\big\},
\end{equation}
 and we have also for all $x>0$,  $\tbPhic(x)=o(1)$, $\tbPhid(x)=o(1)$
  and $\tbPhie(x)=o(1)$  as $n\to\infty$. Consequently, considering
  $\tbPhia$ and $\tbPhib$ as in \eqref{de:dep:ora} under
  Assumption \ref{a:mth:rv} and \ref{a:mth:bs}  there exists 
a finite constant $\SowRaC$  such that for all $n\geq1$,
\begin{multline}\label{de:dep:mm:Sigma}
\SowRaC\geq\big\{n_o^2\bigvee
\tbPsi\big(1+(\DiSowInf/\vE)^2\maxnormsup\gB^{1/2}\big)\bigvee\bPhia\big(1+(\DiSowInf/\vE)^2\maxnormsup\gB^{1/2}\big)\bigvee\bPhib\big(1+\DiSowInf/\vE\big)\\\bigvee\tbPhic(1+(\DiSowInf/\vE)^2\maxnormsup\gB^{1/2})\bigvee\tbPhid(\VnormInf{p_{\iZ,\iV}}\gB^{1/2}\maxnormsup^2)\bigvee\tbPhie\big(1+(\DiSowInf/\vE)^2\maxnormsup\gB^{1/2}\big)\\
\bigvee\Ex(\iE/\vE)^8
\bigvee(\DiSowInf/\vE)^8\gB
\bigvee(\maxnormsup/\vE)^2 \Ex(\iE/\vE)^{12}\big\}.
\end{multline}%
\begin{proof}[\noindent\textcolor{darkred}{\sc Proof of Theorem \ref{t:dep:mm}}] 
 We start the
  proof considering the elementary identity \eqref{ri:dec} given in
  the proof of Theorem \ref{t:iid:ora} where we bound the two rhs terms separately.
  The second rhs term  we bound with help of Proposition
  \ref{p:dep:mm:Rest}. Thereby, there exists a numerical constant $C$
  such that for all $\So\in\Socwr$ hold
  \begin{multline}\label{ri:dec:mm:dep}
    \Ex\VnormZ{\hDiSo[\hDi]-\So}^2\leq \Ex\left(\Ind{\Eset}\VnormZ{\hDiSo[\hDi]-\So}^2\right)\\+C\;n^{-1}\;\maxnormsup^2 \{\vE^2+\DiSowTwo\} (1+(\DiSowInf/\vE)^2\gB)[\SowRaC\vee n^3\exp(-n^{1/6}q^{-1}/100) \vee  n^4 q^{-1}\beta_{q+1}].
  \end{multline}
Consider the first rhs term. On the event $\Eset$
  the upper bound given in \eqref{t:iid:ora:key:arg} implies
  \begin{equation*}
    \VnormZ{\hDiSo[\hDi]-\So}^2\Ind{\Eset}\leq 582 \{[\biasnivSo[\taDi]\vee\pen[\taDi]]\} +42 \max_{\taDi\leq k\leq \Mot}\vectp{ \VnormZ{\hDiSo[k]-\DiSo[k]}^2 - \pen[k]/6}.
  \end{equation*}
Keeping in mind that
$\pen[k]=288\penDep\maxnormsup^2\vB[k]^2\delta_k^\Op n^{-1}$  with
$\delta_k^\Op=k\Lambda_k^\Op\Delta_{k}^\Op$, $\penDep\leq 8(1 +(\DiSowInf/\vE)^2 \gB)$ and $\vB[k]^2\leq2(\vE^2+3\DiSowTwo)$   we derive in
  Proposition \ref{p:dep:mm:Eset} below an upper bound for the
  expectation of the second rhs term, the remainder term,
  in the last display. Thereby, from   $\Mut\geq\Mutw$,
  $\biasnivSo\leq \Sow_m 4\OpD^4\Sor^2$ and $\delta_k^\Op\leq
  \OpD^2(1+2\log\OpD)\delta_k^\Opw$ there exists a numerical constant $C$
  such that for all $\So\in\Socwr$ 
 \begin{multline*}
    \Ex\big(\Ind{\Eset}\VnormZ{\hDiSo[\hDi]-\So}^2\big)\leq 
    C\;\big\{[\Sow_{\taDi}\vee n^{-1}\delta_{\taDi}^\Opw] +
n^{-1}[\SowRaC\vee n^3\exp(-n^{1/6}q^{-1}/100) \vee  n^4 q^{-1}\beta_{q+1}]\big\} \\\times\maxnormsup^2 \OpD^4(\Sor^2+\vE^2+\DiSowTwo)(1+(\DiSowInf/\vE)^2 \gB).
  \end{multline*}
Replacing in  \eqref{ri:dec:mm:dep} the first rhs by the last upper bound
we obtain the assertion of the theorem, which completes the proof.
 \end{proof}%
\begin{prop}\label{p:dep:mm:Eset} Under the assumptions of Theorem \ref{t:dep:mm}
  there exists a numerical constant $C$ such that for all $n\geq1$ 
\begin{multline*}\Ex\set{\max_{\taDi\leq k\leq
      \Mot}\vectp{\VnormZ{\hDiSo-\DiSo}^2 -48\maxnormsup^2\vB^2\penDep\Di
      \Lambda_\Di^\Op\Delta_{\Di}^\Op n^{-1}}}\\\leq
C\;n^{-1}\;\maxnormsup^2 \{\vE^2+\DiSowTwo\} (1+(\DiSowInf/\vE)^2\gB)[\SowRaC\vee n^3\exp(-n^{1/6}q^{-1}/100) \vee  n^4 q^{-1}\beta_{q+1}].
\end{multline*}
\end{prop}
\begin{proof}[\noindent\textcolor{darkred}{\sc Proof of Proposition
    \ref{p:dep:mm:Eset}}]
  We start the
  proof with an upper bound similar to \eqref{p:p:iid:Eset:e1} using $\Mot\leq\Motw$, that is,
 \begin{multline}\label{p:p:dep:mm:Eset:e1}
\Ex\set{\max_{\taDi\leq \Di\leq\Mot}\vectp{\VnormZ{\hDiSo-\DiSo}^2
    -48\maxnormsup^2 \penDep\vB^2\Di\Lambda_\Di^\Op\Delta_{\Di}^\Op n^{-1}}}\\
\hfill\leq 2\Ex\set{ \max_{\taDi\leq\Di\leq\Mot} \VnormS{\DifOp^{-1}}^2\vectp{\Vnorm{\fou{V}_{\uDi}}^2
    -24 \maxnormsup^2 \penDep\vB^2\Di\Lambda_\Di^\Op n^{-1}}}\\\hfill +
 \Ex\set{n\vectp{\Vnorm{\fou{V}_{\underline{\Motw}}}^2-24\maxnormsup^2\vB[\Motw]^2\penDep \Motw \log(n) n^{-1}}}
\\+24\maxnormsup^2\vB[\Motw]^2\penDep\Motw  \log(n)P(\bigcup_{k=1}^{\Motw}\Xiset[k]^c)
 +\max_{\taDi\leq \Di\leq\Mot}\VnormZ{\DiSo}^2P(\bigcup_{k=1}^{\Mot}\hOpset[k]^c)
\end{multline}
where we bound separately each of the four rhs
terms. In order to bound  (i) the first and (ii) )second rhs term we employ \eqref{l:dep:tal:conc} in
Lemma \ref{l:dep:tal} with $k=\aDi$, (i)  $K=\Mot$ and (ii) $K=\Motw$,
and sequence $\ga=(\ga_m)_{m\geq1}$ given by (i)
$\ga_m=\VnormS{\DifOp^{-1}}^2$ and (ii) $\ga_m=n\Ind{\{m=\Motw\}}$. Keeping in mind  the
definition of $\Mot$, $\Motw$ and $\Mot\leq\Motw\leq\gauss{n^{1/4}}$, 
and hence in both cases
$\ga_{(K)}K^2\leq n^{3/2}$,  there exists a numerical constant $C>0$
such
that 
 \begin{multline*}
\Ex\set{\max_{\aDi\leq \Di\leq\Mot}\vectp{\VnormZ{\hDiSo-\DiSo}^2
    -48\maxnormsup^2\penDep  \vB^2\Di\Lambda_\Di^\Op\Delta_{\Di}^\Op n^{-1}}}\\
\leq Cn^{-1} \maxnormsup^2 \big\{\vE^2\bPsi\big(1+(\DiSoInf/\vE)^2\maxnormsup\gB^{1/2}\big)+\vE^2\bPhib\big(1+\DiSoInf/\vE\big)\\\hfill+\vE^2\tbPhie\big(1+(\DiSoInf/\vE)^2\maxnormsup\gB^{1/2}\big)
   +  \vE^2 (1 + \DiSoInf/\vE)^2  n^{7/3} q^{-1} \beta_{q+1}   +\Ex(\iE/\vE)^6\big\}
\\
+24\maxnormsup^2\penDep \vB[\Motw]^2\Motw\log(n) P(\bigcup_{k=\taDi}^{\Motw}\Xiset[k]^c)
 +\max_{\taDi\leq \Di\leq\Mot}\VnormZ{\DiSo}^2P(\bigcup_{k=\taDi}^{\Mot}\hOpset[k]^c)
\end{multline*}
with $\bPsi$, $\bPhib$  as in \eqref{de:dep:ora} and $\tbPhie$ as in \eqref{de:dep:mm}.
Taking into account that $\bPsi(x)\leq\tbPsi(x)$ and that Assumption \ref{a:iid:ora} \ref{a:iid:ora:b} holds with
$\DiSoTwo:=\DiSowTwo$ and $\DiSoInf:=\DiSowInf$ uniformly for all
$\So\in\Socwr$ and $\Op\in\OpcwdD$, it follows that 
 \begin{multline*}
\Ex\set{\max_{\aDi\leq \Di\leq\Mot}\vectp{\VnormZ{\hDiSo-\DiSo}^2
    -48\maxnormsup^2\penDep  \vB^2\Di\Lambda_\Di^\Op\Delta_{\Di}^\Op n^{-1}}}\\
\leq Cn^{-1} \maxnormsup^2 \big\{\vE^2\tbPsi\big(1+(\DiSowInf/\vE)^2\maxnormsup\gB^{1/2}\big)+\vE^2\bPhib\big(1+\DiSowInf/\vE\big)\\\hfill+\vE^2\tbPhie\big(1+(\DiSowInf/\vE)^2\maxnormsup\gB^{1/2}\big)
   +  \vE^2 (1 + \DiSowInf/\vE)^2  n^{7/3} q^{-1} \beta_{q+1}   +\Ex(\iE/\vE)^6\big\}
\\
+24\maxnormsup^2\penDep \vB[\Motw]^2\Motw\log(n) P(\bigcup_{k=\taDi}^{\Motw}\Xiset[k]^c)
 +\max_{\taDi\leq \Di\leq\Mot}\VnormZ{\DiSo}^2P(\bigcup_{k=\taDi}^{\Mot}\hOpset[k]^c)
\end{multline*}
Exploiting that $\vB^2[\Mot]\leq 2(\vE^2+3\DiSowTwo)$, $\penDep\leq 8(1 +
(\DiSowInf/\vE)^2 \gB)$,  $\Motw\log(n)\leq
n$   and $\max_{\taDi\leq \Di\leq\Mot}\VnormZ{\DiSo}^2\leq\DiSowTwo$,
replacing the probability $P(\bigcup_{k=\taDi}^{\Mot}\hOpset[k]^c)$ and
$P(\bigcup_{k=\taDi}^{\Motw}\Xiset[k]^c)$ by its upper bound  
given in \eqref{l:dep:mm:sets:hO} and  \eqref{l:dep:mm:sets:B} in Lemma
\ref{l:dep:mm:sets}, respectively, and  employing the definition
of $\SowRaC$ as in \eqref{de:dep:mm:Sigma} we obtain the result of the proposition, 
 which completes the proof.
\end{proof}
\begin{prop}\label{p:dep:mm:Rest} Under the assumptions of Theorem
  \ref{t:dep:mm}  there exists a numerical constant $C$ such that 
 for all $1\leq q\leq n$ 
 \begin{multline*}
\Ex\big(\VnormZ{\hDiSo[\hDi]-\So}^2\1_{\Eset^c}\big)\leq 
C\;n^{-1} [\SowRaC\vee n^3\exp(-n^{1/6}q^{-1}/100) \vee  n^4
q^{-1}\beta_{q+1}]\\
\times \;\maxnormsup^2 \{\vE^2+\DiSowTwo\} (1+(\DiSowInf/\vE)^2\gB).
\end{multline*}
\end{prop}
\begin{proof}[\noindent\textcolor{darkred}{\sc Proof of Proposition
    \ref{p:dep:mm:Rest}}]  
Following line by line the proof of Proposition  \ref{p:dep:Rest}
with  $\maxDi:=\DiMa$,  $\DiSoInf:=\DiSowInf$ and $\penDep:=8(1 + (\DiSowInf/\vE)^2\gB)$
we have
\begin{multline*}
\Ex\big(\VnormZ{\hDiSo[\hDi]-\So}^2\1_{\Eset^c}\big)\leq
 Cn^{-1} \maxnormsup^2 \big\{\vE^2\bPhia\big(1+(\DiSowInf/\vE)^2\maxnormsup\gB^{1/2}\big)+\vE^2\bPhib\big(1+\DiSowInf/\vE\big)\\
  \hfill +  \vE^2 (1 + \DiSowInf/\vE)^2  n^{7/3} q^{-1} \beta_{q+1}   +\Ex(\iE/\vE)^{12}\big\}\\+\{72\maxnormsup^2 \maxDi \penDep \vB[\maxDi]^2 \log(n) + 6\DiSowTwo\}P(\Eset^c)
\end{multline*}
with $\bPhia$ and $\bPhib$ as in \eqref{de:dep:ora}. Exploiting further the definition of $\SowRaC$ as in
\eqref{de:dep:mm:Sigma} and that $\vB[\maxDi]^2\leq
2\{\vE^2+3\DiSowTwo\}$,  $\penDep= 8(1 + (\DiSowInf/\vE)^2 \gB)$ and
$M\log(n)\leq n$ the result of the proposition follows now  by
replacing the probability $P(\Eset^c)$ by its upper bound  
given in \eqref{l:dep:mm:sets:E} in Lemma \ref{l:dep:mm:sets}, which completes the proof.
\end{proof}
\begin{lem}\label{l:dep:mm:sets}
Under the assumptions of Theorem \ref{t:dep:mm} there exists a
numerical constant $C$ such that for all $1\leq q\leq n$
\begin{align}
&\ProbaMeasure\big(\Aset^c)=\ProbaMeasure\big(\{1/2\leq\hsigma^2_Y/\sigma_Y^2\leq
  3/2\}^c\big)\leq C\; \SowRaC\;n^{-2},
 \label{l:dep:mm:sets:A}\\
&\ProbaMeasure\big(\Bset^c\big)\leq\ProbaMeasure\big(\bigcup_{\Di=1}^{\Motw+1}\Xiset^c\big)\leq C\;
 [\SowRaC\vee n^3\exp(-n^{1/2}q^{-1}/50) \vee n^4 q^{-1}\beta_{q+1}]\; n^{-2}, \label{l:dep:mm:sets:B}\\
&\ProbaMeasure\big(\Cset^c\big)\leq C\; [\SowRaC\vee  n^3\exp(-n^{1/6}q^{-1}/100) \vee n^4 q^{-1}\beta_{q+1}]\; n^{-2}, \label{l:dep:mm:sets:C}\\
&\ProbaMeasure\big(\Eset^c\big)\leq C\;[\SowRaC\vee  n^3\exp(-n^{1/6}q^{-1}/100)\vee n^4 q^{-1}\beta_{q+1}]\; n^{-2}, \label{l:dep:mm:sets:E}\\
&\ProbaMeasure\big(\bigcup_{\Di=1}^{\Mot}\hOpset\big)\leq  C\;
 [\SowRaC\vee n^3\exp(-n^{1/2}q^{-1}/50) \vee  n^4 q^{-1}\beta_{q+1}]\;n^{-2}. \label{l:dep:mm:sets:hO}
\end{align}
\end{lem}
\begin{proof}[\noindent\textcolor{darkred}{\sc Proof of Lemma
    \ref{l:dep:mm:sets}}]The proof of  \eqref{l:dep:mm:sets:A} follows
  line by line the proof of \eqref{l:dep:sets:A} in Lemma
  \ref{l:dep:sets} using the definition of $\SowRaC$ as in
  \eqref{de:dep:mm:Sigma} rather $\SoRaC$  than
  \eqref{de:dep:ora:Sigma} and hence we omit the details.
 Consider 
\eqref{l:dep:mm:sets:B}--\eqref{l:dep:mm:sets:E}. Let  $\ga$ be a sequence given by
$\ga_m=\VnormS{\DifOp^{-1}}^2$ where
$\ga_{(m)}=\Delta^T_{m}$ and $n_o$ an integer satisfying
\ref{de:dep:no} uniformly for all $\Op\in\OpcwdD$ and $\So\in\Socwr$,
that is,  $n\geq1024\maxnormsup^4\OpD^2(6+8(\DiSowInf/\vE)^2\gB)(\Motw+1)^2\Delta_{\Motw+1}^\Opw\geq
1024\maxnormsup^4(6+8(\DiSoInf/\vE)^2\gB)(\Mot+1)^2\Delta_{\Mot+1}^\Op$ for
 all $n> n_o$  by construction. We distinguish in the following the cases $n\leq n_o$ and
$n > n_o$. Consider   
\eqref{l:dep:mm:sets:B}. Following line by line the proof of
\eqref{l:dep:sets:B} together with $\bPhid(x)\leq\tbPhid(x)$ and
$\Motw+1\leq n$ we have 
$\ProbaMeasure\big(\bigcup_{\Di=1}^{\Motw+1}\Xiset^c\big)
\leq 6\;
n^{-2}\;\{\tbPhid(\maxnormsup^2\VnormInf{p_{\iZ,\iV}}\gB^{1/2})\vee
n^3\exp(-n^{1/2}q^{-1}/50)\vee n^4q^{-1}\beta_{q+1}\big\}$.
 By combination of the
two cases and employing the definition of
$\SowRaC$ given in \eqref{de:dep:mm:Sigma} we obtain \eqref{l:dep:mm:sets:B}. The proof of
\eqref{l:dep:mm:sets:C} follows line by line  the proof of
\eqref{l:dep:sets:C} in Lemma \ref{l:dep:sets}. 
Exploiting $\bPhic(x)\leq\tbPhic(x)$ we obtain 
$\ProbaMeasure\big(\Cset^c\big)\leq 6 
\{n^3\exp(-n^{1/6}q^{-1}/100)\}\vee\tbPhic(1+(\DiSowInf/\vE)^2\maxnormsup\gB^{1/2})n^{-2}+
n q^{-1}\beta_{q+1}+ 64 (\maxnormsup^2/\vE^2) \Ex(\iE/\vE)^{12}n^{-2}
$. The assertion \eqref{l:dep:mm:sets:C} follows employing the definition of
$\SowRaC$ given in \eqref{de:dep:mm:Sigma}. Consider
\eqref{l:dep:mm:sets:E}. Due to Lemma \ref{app:pre:l5} it holds
$\ProbaMeasure\big(\Eset^c\big)\leq \ProbaMeasure\big(\Aset^c\big)
+\ProbaMeasure\big(\Bset^c\big)
+\ProbaMeasure\big(\Cset^c\big)$. Therefore, the assertion
\eqref{l:dep:mm:sets:E} follows from
\eqref{l:dep:mm:sets:A}--\eqref{l:dep:mm:sets:C}. The proof of 
\eqref{l:dep:mm:sets:hO} follows in same manner as the proof of 
\eqref{l:dep:sets:hO}, and we omit the details, which completes the proof.
\end{proof}
